\documentclass[10pt]{amsart}
\usepackage{mathrsfs, geometry}
\usepackage{graphicx}
\usepackage{epstopdf}
\usepackage{epsfig,psfrag}
\usepackage{amsmath, amsfonts, amsthm,amssymb}
\usepackage{hyperref}
\usepackage{aurical}
\usepackage[T1]{fontenc}
\usepackage{yfonts}
\usepackage[usenames, dvipsnames]{xcolor}
\newtheorem{theorem}{Theorem}[section]
\newtheorem{definition}[theorem]{Definition}
\newtheorem{lemma}[theorem]{Lemma}

\newtheorem{corollary}[theorem]{Corollary}
\newtheorem{remark}[theorem]{Remark}

\newtheorem{proposition}[theorem]{Proposition}

\numberwithin{equation}{section}

  \newcommand{\set}[1]{{\left\{#1\right\}}}

\newcommand{\HS}{{\mathtt{HS}}}
\newcommand{\s}{{\sigma}}
\newcommand{\id}{{\rm Id}}

\newcommand{\ii}{{\rm i}}

\newcommand{\Oo}{\Omega}
\newcommand{\Op}{{\rm Op}}
\newcommand{\divisor}{{\mathtt d}}
\newcommand{\Lipg}{{{\rm Lip}(\g,\Oo)}}
\newcommand{\Lipn}[1]{{{\rm Lip}(\g,\Omega_{#1})}}

\newcommand{\Dc}{{\mathtt{ D}_\g}}
\newcommand{\C}{{\mathbb C}}

\newcommand{\N}{{\mathbb N}}

\newcommand{\R}{{\mathbb R}}

\newcommand{\T}{{\mathbb T}}
\newcommand{\Z}{{\mathbb Z}}
\newcommand{\Ro}{{\mathtt R_0}}
\newcommand{\cB}{{\mathcal B}}

\newcommand{\cD}{{\mathcal D}}

\newcommand{\cF}{{\mathcal F}}

\newcommand{\cH}{{\mathcal H}}

\newcommand{\cM}{{\mathcal M}}

\newcommand{\cP}{{\mathcal P}}

\newcommand{\cR}{{\mathcal R}}

\newcommand{\cV}{{\mathcal V}}
\newcommand{\cW}{{\mathcal W}}

\newcommand{\fp}{{\mathfrak{p}}}

\newcommand{\tD}{{\mathtt{D}}}

\newcommand{\g}{\gamma}
\newcommand{\f}{\varphi}

\newcommand{\im}{{\rm i}}
\newcommand{\jap}[1]{\langle #1 \rangle}
\newcommand{\e}{{\varepsilon}}

\newcommand{\zia}{\eta}

\begin{document}

\title{Linear Schr\"odinger equation with an almost periodic potential}
\author{Riccardo Montalto, Michela Procesi}
\address{Riccardo Montalto: Universit\`a degli Studi di Milano, Dipartimento di Matematica Federigo Enriques, Via Saldini 50, 20133, Milano}
\email{riccardo.montalto@unimi.it}
\address{Michela Procesi: Universit\`a di Roma tre, Dipartimento di Matematica e fisica, Largo S. Leonardo Murialdo, 1, 00146, Roma.}
\email{procesi@mat.uniroma3.it}
\begin{abstract}We study the reducibility of a Linear Schr\"odinger equation subject to a small unbounded almost-periodic perturbation which is analytic in time and space. Under appropriate assumptions on the smallness, analiticity  and on the frequency of the almost-periodic perturbation, we prove that such an equation is reducible to constant coefficients via an anaytic almost-periodic change of variables. This implies control of both Sobolev and Analytic norms for the solution of the corresponding Schr\"odinger equation for all times.
\end{abstract}

\maketitle

\tableofcontents

\section{Introduction}
The problem of control of Sobolev norms for Linear Schr\"odinger operators on a torus with smooth time dependent potential has been studied by various authors. Groundbreaking results were proved by Bourgain in \cite{Bo99a} in the case of quasi-periodic bounded potentials with a Diophantine frequency, then in \cite{Bo99b} for general time dependent potentials. The main result was an upper bound on the growth in time of the Sobolev norm, respectively logaritmic and polynomial in time. Such results were generalized to unbounded potentials in see \cite{D}, \cite{MR}, \cite{Mo:Asym2018}, \cite{BamMo:JMP2018},\cite{BGMR2}, \cite{Mo:JDE2019}, \cite{Mo:IMRN2019}, \cite{BM19}, \cite{MF}.
\\
{The main feature of such results is that they are very general, require little or no conditions on the time dependence of the potential and can often be applied also in non-perturbative settings. At this level of generality such results are in fact optimal as showed in \cite{Bo99b}. See also \cite{Ma18}, \cite{HaMa18} for examples of growth.}
 \\
  A parallel point of view is to study the reducibility of Schr\"odinger operators with quasi-periodic  potentials by requiring stronger non-resonance conditions on the frequency, see \cite{EK1}. We recall that a first order differential equation is said to be reducible if there exists a (uniformly bounded) time dependent operator which conjugates it to an equation whose vector field is diagonal (or block diagonal). 
  Thus one gets a uniform control in time of the Sobolev norms to the price of restricting to small quasi-periodic potentials with rather involuted  non-resonance conditions on the frequency. We remark that reducibility is a key argument in KAM for non-linear PDEs. This is a strong motivation for studying reducibility for linear PDEs. Conversly many KAM results can be adapted to the reducibility setting.
  \\
  As can be expected the (block) diagonalization algorithm relies on lower bounds on the difference of distinct eigenvalues (the spectral gaps) as well as on a strong control on their possible multiplicity. Indeed the first results were for bounded potentials in the case of Dirichlet boundary conditions on $[0,\pi]$, where the eigenvalues are simple (see for instance \cite{K1}, \cite{Poschel:1989}, \cite{Poschel:1996}, \cite{Kuksin-Poschel:1996}, \cite{Ku2}). 
  The last ten years have seen considerable progress in this field, particularly in the case of unbounded potentials. The first results were in \cite{IPT} in the case of periodic potentials and \cite{BBM}, \cite{BBM16} for the quasi-periodic case. Regarding Schr\"odinger equations we mention \cite{FP}, \cite{F},\cite{Bam17},\cite{Bam18}. Note that all the preceding papers deal with Sobolev stability; generalizing to  the analytic case, especially in the case of unbounded potentials of order two and in the context of a nonlinear KAM scheme, is not  straightforward. A strategy was discussed in \cite{CFP},\cite{FP2}. While the literature on reducibility of quasi-periodic potentials is quite extensive in the case of one space dimension, the case of higher dimensional manifolds is still largely open. We mention \cite{EK}, \cite{GP}, \cite{EGK} and finally \cite{BGMR}, \cite{FGMP},\cite{Mo:IMRN2019}, \cite{CoMo}, \cite{BLM} for an unbounded potential. 
 \\
  Common features of the reduction algorithms are : 1. they are perturbative, 2. they require complicated non-resonance conditions depending on the potential, 3. they strongly depend on the  number of frequencies. 
 
% \vskip10pt
  In the present paper we study the reducibility of  Schr\"odinger equations on the circle with a small  {\it unbouned almost periodic} potential of the form
  \vskip-10pt
  \begin{equation}\label{main equation}
  \begin{aligned}
  \partial_t u & = \im   \Big(  \partial_x^2 + \e { P}( t)  \Big)u\,, \quad  \\
  {P}( t ) & := V_2( x, t) \partial_x^2 +V_1( x, t) \partial_x + V_0( x, t)\,, \quad x \in \T := \R/(2 \pi \Z)\,, t\in \R\,.   
  \end{aligned}
  \end{equation}
  Here $V_0, V_1, V_2$ are analytic (in an appropriate sense) almost periodic functions of time with frequency  $\omega$ which is an infinite dimensional Diophantine vector in $\ell^\infty( \N, \R)$ (see definitions  \ref{suca}  and \eqref{diofantino}). For small $\e$ we prove a reducibility result under the assumption  that for any $t \in \R$, the operator ${ P}( t)$ is $L^2$ self-adjoint and  that $\omega$ belongs to some (explicit but convoluted) Cantor set of asymptotically full measure.

    Of course the difficulty of such a result is strongly related to the {\it regularity} of the almost-periodic potential.
    Indeed, by definition, an almost periodic function is the limit  of {\it quasi-periodic}  ones with an increasing number of frequencies.
    If the limit is reached sufficiently fast, the most direct strategy is to diagonalize iteratively the Schr\"odinger operators with {\it quasi-periodic}  potentials, by considering at each step $n$  the operator  as a small perturbation of the one of the previous step. This procedure in fact works if one considers a sufficiently {\it  smoothing} and {\it regular} potentials but becomes very delicate in the case of unbounded potentials.
  \\ 
Good comparisons are:
 \cite{Poschel:2002} which studies a {\it smoothing} nonlinear Schr\"odinger equation with external parameters and proves existence of on almost-periodic solutions with superexponential decay in the Fourier modes.
 \cite{Bourgain:2005}, on almost-periodic solutions for a  nonlinear Schr\"odinger equation with external parameters with subexponential decay in the Fourier modes.
 In the first paper the very fast decay implies that at each KAM step, one only needs to construct quasi-periodic solutions (with increasing number of frequencies) which is a well known result; the only point  is to show that they converge superexponentially to a {\it non-trivial} almost periodic solution. In the second paper the author does not rely on quasi-periodic approximations, this requires to completely revisit the KAM scheme but leads to solutions with much less regularity.
 In this paper we follow the general point of view of \cite{Bourgain:2005}, see also \cite{BMP2}, using the same infinite dimensional Diophantine vectors and various technical lemmata (detailed proofs of all the technical Lemmata can ber found in \cite{BMP1:2018}). 
\\
 In order to give the precise statement of our Theorems, we introduce some notations and definitions. 

%\smallskip
%\noindent
%{\bf The set of infinite dimensional frequencies.}

\noindent
We define the {\it parameter space} of frequencies as a subset of\footnote{Here and in the follwing $\N$ does not contain $\{0\}$.}  $\ell^\infty(\N,\R)$, where we recall that 
$$
\ell^{\infty}(\N,\R) := \Big\{ \omega = (\omega_j)_{j \in \N} \in \R^\N : \| \omega \|_{\infty} := \sup_{j \in \N} |\omega_j|< \infty \Big\}\,. 
$$
More precisely, our set of frequencies is the infinite dimensional cube 
\begin{equation}\label{def spazio dei parametri}
\Ro :=  \Big[ 1\,,\,  2 \Big]^\N\,.
\end{equation} 
We endow the space of parameters $\Ro$ with the $\ell^\infty$ metric, namely we set 
\begin{equation}\label{metrica spazio parametri}
d_\infty(\omega_1, \omega_2) := \| \omega_1 - \omega_2 \|_\infty, \quad \forall \omega_1, \omega_2 \in \Ro\,. 
\end{equation}
Furthermore, we endow $\Ro$ with the probability measure
${\mathbb P}$ induced by the product measure of the infinite-dimensional cube $\Ro$. 
\\
We now define the set of {\sl Diophantine} frequencies. The following definition is a slight generalization of the one given by Bourgain in \cite{Bourgain:2005}.
\begin{definition}\label{diofantino} Given $\gamma \in (0, 1)$, $\mu > 1$, we  denote by $\mathtt{D}_{\gamma, \mu}$ the set of {\it Diophantine} frequencies  
	\begin{equation}\label{diofantinoBIS}
	\mathtt{D}_{\gamma, \mu}:=\set{\omega\in \Ro \,:\;	|\omega\cdot \ell|> \gamma \prod_{j\in \N}\frac{1}{(1+|\ell_j|^{\mu} \jap{j}^{\mu})}\,,\quad \forall \ell\in \Z^\N: 0<\sum_{j\in \N}|\ell_j|<\infty}.
	\end{equation} 
	In the following we shall fix $\mu=2$ and denote $\Dc:= \mathtt{D}_{\gamma, 2}$.
\end{definition}
For all $\mu >1$, Diophantine frequencies are {\sl typical} in the set $\Ro$ in the sense of the following measure estimate, proved in \cite{Bourgain:2005} (see also \cite{BMP1:2018}).
\begin{lemma}\label{misura}
	For $\mu >1$ the exists a positive constant $C(\mu ) > 0$ such that
	\[{\mathbb P}\big(\Ro \setminus \mathtt D_{\g,\mu}\big)
	\leq C(\mu) \g\,.
	\]
\end{lemma}
For $\zia>0$, we define the set of infinite integer vectors with {\it finite support}  
\begin{equation}\label{Z inf *} 
\Z^\infty_* :=   \Big\{ \ell \in \Z^\N : |\ell|_\zia := \sum_{j\in \N}  j ^\zia |\ell_j| < \infty \Big\}. 
\end{equation}
Note that $\ell_j \neq 0$ only for finitely many indices $j \in \N$. 
\begin{definition}\label{suca}
	Given $\omega\in \Dc$ and a Banach space $X,\|\cdot\|_X$,  we say that $F(t):\R\to X$ is almost-periodic in time with frequency $\omega$ and analytic in the strip $\s>0$ 
	 if we may write it in totally   convergent Fourier series
\[
F(t)= \sum_{ \ell \in \Z^\infty_*} \widehat F(\ell) e^{\im \ell \cdot \omega t }
\quad \text{such that} \quad \widehat F(\ell) \in X\,,\;\forall \ell \in \Z^\infty_* \quad \text{and}
\quad
\sum_{ \ell \in \Z^\infty_*} \| \widehat F(\ell) \|_X e^{\s |\ell|_\zia } <\infty.
\]
\end{definition}
We shall be particularly interested in almost-periodic functions where $X=\cH(\T_\s)$ 
\[
\cH(\T_\s):= \Big\{ u= \sum_{n \in\Z}\hat u_n e^{\im n x}\,,\; \hat u_j \in \C\,:\quad \|u\|_{\cH(\T_\s)}:= \sum_{n \in\Z}|\hat u_n| e^{\s |n| } <\infty \Big\}
\]
is the space of analytic functions $\T_\sigma\to \C$, where   
$\T_\sigma := \{ \f \in \C : {\rm Re}(\f) \in \T, \quad |{\rm Im}(\f)| \leq \sigma \}$ is the thickened torus.

\smallskip

\noindent
Now we are ready to state precisely our main result. We make the following assumptions. 
\begin{itemize}
\item{\bf (H1)} The functions $V_0, V_1, V_2 $ are almost-periodic and analytic,  in the sense of Definition \ref{suca}, for $\overline \sigma > 0$ and $X= \cH(\T_{\overline \sigma})$.

\item{\bf (H2)} We assume that 
\begin{equation}\label{condizioni V0 V1 V2}
\begin{aligned}
& V_2(x , t) = \overline{V_2(x , t)}\,, \quad \forall (x , t) \in \T \times \R\,, \\
& V_1(x , t) = 2 \overline{\partial_x V_2(x , t)} - \overline{V_1(x , t)}\,,\quad \forall (x , t) \in \T \times \R \\
& V_0(t, x) = \overline{V_0(x , t)} - \overline{\partial_x V_1(x , t)} + \overline{\partial_{xx} V_2(x , t)}\,,\quad \forall (x , t) \in \T \times \R \,. 
\end{aligned}
\end{equation}
This implies that the operator ${ P}(t)$ in \eqref{main equation} is $L^2$ self-adjoint for $t  \in \R$.  Here and in the following we denote by $\cB(E,F)$ the space of bounded linear operators from $E$ to $F$.
\end{itemize}
\begin{theorem}[\bf Reducibility]\label{teorema principale}
Let $\overline \sigma > 0$ and assume the hypotheses {\bf (H1)} and {\bf (H2)}. Then there exists $\e_0 \in (0, 1)$ small enough such that for any $\e \in (0, \e_0)$ there exists a subset $\Omega_\e \subset \Ro = [1, 2]^\N$ satisfying 
\begin{equation}\label{stima misura main theorem}
\lim_{\e \to 0} {\mathbb P}(\Omega_\e) = 1
\end{equation}
such that the following holds. For any $\omega \in \Omega_\e$, $t \in \R$, $0 < \sigma < \sigma'  \leq \overline \sigma/4$, $\rho > 0$ there exists $\delta = \delta(\sigma, \sigma') \in (0, 1)$ such that if $\e \gamma^{- 1} \leq \delta$ then there exists a unitary (in $L^2(\T)$) operator $W_\infty(t) \equiv W_\infty(t; \omega)$ such that:
\begin{enumerate}
	\item $W_\infty(t), W_\infty(t)^{- 1}$ are almost periodic and analytic maps on the strip $\overline{\sigma}/4$ into $X= {\mathcal B}	\Big({\mathcal H}(\T_{\sigma'} ), {\mathcal H}(\T_{\sigma} ) \Big)$\,.
\item
 $u(\cdot , t)$ is a solution of the Schr\"odinger equation \eqref{main equation} if and only if $v( \cdot, t) = {W}_\infty( t)^{- 1}[u(\cdot, t)]$ is a solution of the time independent equation 
\begin{equation}\label{eq ridotta}
\partial_t v = \ii {\mathcal D}_\infty v
\end{equation}
where ${\mathcal D}_\infty$ is a linear, self-adjoint,  time independent, $2 \times 2$ block-diagonal operator\footnote{We recall that an operator $L$ on a vector space $V$  is $d\times d$  block diagonal if there exists a decomposition of $V=\overline{\oplus V_j}$ such that $L$ maps each $V_j$ in itself and all the $V_j$ have dimension at most $d$.} of order two such that the  commutator $[{\mathcal D}_\infty, \partial_{xx}] = 0$. 
\item For any $s \geq 0$, the maps $\R \to {\mathcal B}\Big(H^s(\T ),H^s(\T) ) \Big)$, $t \mapsto {W}_\infty(t)^{\pm 1}$ are bounded. 
\end{enumerate}
\end{theorem}
From the Theorem stated above, we can deduce the following Corollaries:
\begin{corollary}[Asymptotics of the eigenvalues]
The spectrum of the operator ${\mathcal D}_\infty$ is given by
\begin{align}
\label{spettro}
{\rm spec}({\mathcal D}_\infty) &= \{\mu_0(\omega)\}\cup \{ \mu_j^{(+)}(\omega), \mu_j^{(-)}(\omega)\}_{j\in \N_0} \subset \R \,,\\
 \mu_j^\s(\omega)&= \lambda_2 j^2 + \s \lambda_1 j+ \lambda_0(\omega) + \s\frac{\lambda_{-1}(\omega)}{{j}} + \frac{r_j^\s}{{j}^2} \,,\;j >0\nonumber
\end{align}
where $\lambda_2-1\,,\, \lambda_1 \sim \e$ do not depend on $\omega$, while $\lambda_0, \lambda_{- 1}, r_j^\sigma$ are Lipschitz w.r. to $\omega$ and of order $\e$. Finally $\mu_0$ is Lipschitz w.r. to $\omega$ and of order $\e$.
\end{corollary}
For compactness of notations we set $\mu_0^{(+)} = \mu_0^{(-)} = \mu_0$. 
\begin{corollary}[Characterization of the Cantor set]
	The Cantor set $\Omega_{\e}$, given in Theorem \ref{teorema principale}, is defined explicitly in terms of the spectrum of the block diagonal operator ${\mathcal D}_\infty$. More precisely it is equal to the set $\Omega_\infty(\gamma)$, $\gamma = \e^{a}$ for some $a \in (0, 1)$, where 
\begin{equation}\label{espressione Omega infty autovalori intro}
\begin{aligned}
\Omega_\infty(\gamma) & := \Big\{ \omega \in \mathtt D_\gamma : |\omega \cdot \ell + \mu_j^{(\sigma)} - \mu_{j'}^{(\sigma')}| \geq \frac{2 \gamma}{{\divisor}(\ell)}, \quad \forall (\ell, j, j') \in \Z^\infty_* \times \N_0 \times \N_0, \quad j \neq j', \quad \sigma, \sigma' \in \{ +, - \} \\
& |\omega \cdot \ell + \mu_j^{(\sigma)} - \mu_{j}^{(\sigma')}| \geq \frac{2 \gamma}{{\divisor}(\ell) \langle j \rangle^2}, \quad \forall (\ell, j) \in (\Z^\infty_* \setminus\{ 0 \}) \times \N_0 , \quad \sigma, \sigma' \in \{ +, - \} \Big\} 
\end{aligned}
\end{equation}
where 
\[
{\divisor}(\ell) := \prod_{n\in \N}(1+|\ell_n|^{4} \jap{n}^{4}), \quad \forall \ell \in \Z^\infty_*\,. 
\]
\end{corollary}
	
\begin{corollary}[Dynamical consequences]\label{corollario sobolev}
Under the same assumptions of Theorem \ref{teorema principale} the following holds
\begin{itemize}
\item{\bf Analytic stability.} For any $0 < \sigma < \overline \sigma/4$, $\rho > 0$, $u_0 \in {\mathcal H}(\T_{\overline \sigma})$, the unique solution of the equation \eqref{main equation} with initial datum $u( x, 0) = u_0(x)$ satisfies the estimate $\| u(\cdot, t)\|_{{\mathcal H}(\T_{\sigma})} \lesssim_{\sigma, \overline \sigma} \| u_0\|_{{\mathcal H}(\T_{\overline \sigma })}$ uniformly w.r. to $t \in \R$. 
\item{\bf Sobolev stability.} For any $s \geq 0$, $u_0 \in H^s(\T)$, the unique solution of the equation \eqref{main equation} with initial datum $u( x, 0) = u_0(x)$ satisfies the estimate $\| u(\cdot, t)\|_{H^s(\T)} \lesssim_s \| u_0\|_{H^s(\T)}$ uniformly w.r. to $t \in \R$.
\end{itemize}
\end{corollary}
{\begin{remark}
By Theorem \ref{teorema principale}, items $(1)$ and $(3)$, one gets boundedness properties of the maps $W_\infty(t)^{\pm 1}$ both on analytic and Sobolev spaces. This is the reason why in Corollary \ref{corollario sobolev}, we get a stability result for both analytic and Sobolev initial data, see Section \ref{sezione dim teo finali}. 
\end{remark}}

{ {\bf Strategy of the Proof. } The overall strategy of the proof is the one proposed in \cite{BBM} and consists of two main steps: a {\it regularization procedure} and a {\it KAM reduction scheme}. The aim of the first step is to conjugate \eqref{main equation} to a simpler dynamical system where  the vector field is space and time independent up to a sufficiently smoothing remainder. Here one uses the fact that the linear operator in \eqref{main equation} has a pseudo-differential structure. 
\\
In the second step one completes the reduction by applying a KAM  scheme, which relies on the fact that the eigenvalues are at most double, with a quantitative control on the differences.
\\
In order to explain which are the main difficulties to overcome in order to deal with almost-periodic potentials let us describe the strategy more in  detail.

It is convenient  to  think of almost-periodic in time functions as restrictions  functions on an infinite dimensional torus. To this purpose we define analytic functions of infinitely many angles as the class of totally  convergent Fourier series with a prescribed (and very strong) decay on the Fourier coefficients. We  show that in fact this definition coincides with the set of holomorphic functions on a {\it thickened torus} (see Appendix \ref{appendice funzioni olomorfe}) and discuss properties of our set of functions which shall be needed in order to perform the reduction procedure. The interesting point is that we  work with functions  on the thickened torus:
\[
\T^\infty_\s:= \set{\f=(\f_j)_{j\in \N}\,,\quad \f_j\in\C\,:\; {\rm Re}(\f_j)\in \T\;,\;|{\rm Im}(\f_j)|\le \s \jap{j}^\zia}\,. 
\]
so not only we consider analytic functions but the {\it radius of analiticity} increases as $j\to \infty$. This is quite a strong condition but it is not at all clear to us whether it may be weakened, even in apparently harmless ways like requiring  $|{\rm Im}(\f_j)|\le \s \log(1+\jap{j})^p$ with $p\gg 1$.
In the description of the strategy we shall point out  where such a strong assumption is needed.

\vskip10pt
\noindent In the {\it regularization} procedure the first step is to reparameterize the $x$ variable ($x\rightsquigarrow x+ \beta(x,\omega t)$), in order to remove the space dependence in the leading order term $V_2$ of \eqref{main equation}. This induces an invertible  linear operator which acts on the dynamical system removing the $x$ dependence from $V_2$. Here the time behaves as a parameter, so no condition on the time dependence of the potential is needed. Note however that this change of variables {\it mixes} time and space. Namely if we start with a potential which is analytic in time but only Sobolev in space,  after the change of variables it will have finite regularity both in time and in space. For this reason, since we need to preserve analiticity in time throughout our procedure, we require that our potentials are analytic also in space.

\vskip10pt
\noindent In the second step one reparametrizes the  variables $\f\in \T^\infty_\s$ so as to remove the {\it angle}  dependence in $V_2$. Here there are various non-trivial points to discuss, both in order to guarantee that the change of variables is well defined and "invertible" and in order to describe the action on analytic functions.
\\
Indeed even in the case of a finite number of angles,  the regularization procedure is performed on $C^\infty$ potentials and working in the analytic class requires some extra care (see also \cite{FP2}).
\\
In this step one uses the fact that $\omega$ is Diophantine in the sense of \eqref{diofantinoBIS} as well as the fact that the potentials are analytic  with growing radius of analiticity as $j\to \infty$. 
\vskip10pt
\noindent
The remaining steps in the regularization procedure do not introduce further problems w.r.t. the first two steps.
As is typical in this kind of results one could further push the regularization procedure up to an arbitrarily smoothing remainder. We have chosen to regularize our problem up to order $-2$ because this is the {\it minimal action} required in order to complete the successive KAM iterative  procedure.
\\
An interesting point  is that all the regularization steps apart from the first {three}, do not mix the regularity of time and space so that one could work with potentials that are only analytic in time.  
A simple consequence is that if in \eqref{main equation} {we assume that $V_2$ and $V_1$ are constant in time} then we can require that {$V_0$} has only finite regularity in space (but is still analytic in time).  

\smallskip

\noindent
Since we work with a perturbation which is a differential operator whose coefficients are analytic both in time and space, we cannot apply as a {\it black box} the regularization procedure as in \cite{BGMR2}, \cite{Mo:Asym2018}, which is based on Egorov-type theorems and is  developed for general pseudo-differential perturbations of class ${\mathcal C}^\infty$. Indeed developing a general Egorov-type theorem in analytic class does not appear a straightforward question (actually the quantitative estimates that we need might not hold true in a general setting). 

Therefore we perform the regularization procedure in the class of analytic functions,  with quantitative estimates, see Sections \ref{astratto operatori Toplitz} and \ref{sezione riduzione ordine}. The main feature which we exploit is that our perturbation $P$  is a classical pseudo-differential differential operator (i.e. it admits an expansion in homogeneous symbols of decreasing order).

\smallskip

\noindent
We remark that in the regularization procedure, one could impose much weaker analiticity conditions. One sees that in fact the only condition needed here is that there exists $\rho>0$ such that
	\begin{equation}
	\label{zuzzu}
	\sup_{ \ell \in \Z^\infty_*}  \prod_{i\in \N}(1+\jap{i}^2 \ell_i^2) e^{-\rho\sum_j \jap{j}^\eta |\ell_j|} <\infty\,.
	\end{equation}
If we choose different radii of analiticity, such as
\[
\widehat\T^\infty_\rho:= \set{\f=(\f_j)_{j\in \N}\,,\quad \f_j\in\C\,:\; {\rm Re}(\f_j)\in \T\;,\;|{\rm Im}(\f_j)|\le \rho F(j)}\,, \;F(j)\ge 1\,, 
\]
condition \eqref{zuzzu} becomes
\[
\sup_{ \ell \in \Z^\infty_*}  \prod_{i\in \N}(1+\jap{i}^2 \ell_i^2) e^{-\rho \sum_j |\ell_j|F(j)} <\infty\,.
\]
and one can construct many examples where this holds.
\vskip15pt

In the KAM scheme  most diffculties come from quantitative issues, particularly measure estimates.
At a purely formal level our scheme is essentially classical. 
At each step one considers a linear operator of the form $\cD + \cP(\f)$ where $\cP$ is very small while $\cD$ is time independent and block-diagonal with blocks of dimension at most two. First we introduce an "ultraviolet cut-off" operator, so that $ \Pi_N {\mathcal P}$ depends on finitely many angles (depending on $N$), while the remainder $(\id - \Pi_N) {\mathcal P}$ is very small.
\\
Then one applies a linear change of variables  $e^{\cF(\f)}$ where $\cF$ solves the homological equation
\[
-  \omega \cdot \partial_\f {\mathcal F} + [\ii {\mathcal D}, {\mathcal F}] +  \Pi_N {\mathcal P} = [\widehat{\mathcal P}(0)]\,,
\]
where $[\widehat{\mathcal P}(0)]$ is the time-independent and block-diagonal part of $P$.
\\
Direct computations show that (at least at a purely formal level) this change of variables conjugates $\cD + \cP(\f)$ to an operator of the form 
$\cD_{+} + \cP_{+}(\f)$ where $\cP_{+}(\f)\ll \cP(\f)$.  In order to  ensure that a solution to the homological equation exists and in order to give quantitative estimates, one restricts $\omega$ to a set where the spectrum of the operator
\begin{equation}
\label{melnikov}
L(\f) \mapsto -  \omega \cdot \partial_\f L(\f) + [\ii {\mathcal D}, L(\f)] 
\end{equation}
is appropriately bounded from below.
Iterating this {\it KAM step} infinitely many times  one  reduces the operator $\cD + \cP(\f)$,  for all $\omega $ in some implicitly defined set where the condition \eqref{melnikov} holds througout the procedure.
\\The difficult part is to verify that the Melnikov conditions \eqref{espressione Omega infty autovalori intro} are such that:
\;
1. The Cantor set $\Omega_\infty(\gamma)$ has positive measure;\;
2. for all $\omega \in \Omega_\infty(\gamma)$  \eqref{melnikov} holds at each KAM step  with a quantitative control in the solution of the homological equation;\;
3. the iterative scheme converges.
\\
Here one needs not only for \eqref{zuzzu} to hold for all $\rho>0$ but also that the supremum in \eqref{zuzzu} does NOT diverge too badly when $\rho\to 0$. It is here that the special choice of analiticity comes into play, and it is not clear to us if it can be weakened in any significant way.
}

\bigskip

\noindent
The paper is organized as follows. In Section \ref{sez funzioni analitiche generali} we state the properties of the analytic functions on the infinite dimensional torus that we need in our proofs. In Section \ref{sezione astratta op lineari}, we provide some definitions and quantitative estimates for the class of linear operators that we deal with. In particular we define the norms that we use in Sections \ref{sezione riduzione ordine}, \ref{sez KAM redu} and their corresponding properties. In Section \ref{sezione riduzione ordine} we show that our equation can be reduced to another one whose vector field is a two-smoothing perturbation of a diagonal one. This is enough to perform the KAM reducibility scheme of Section \ref{sez KAM redu}. In Section \ref{sezione stime di misure} we provide the measure estimate of the non resonant set of parameters $\Omega_\infty(\gamma)$ (see \eqref{espressione Omega infty autovalori intro}) and in Section \ref{sezione dim teo finali} we conclude the proofs of Theorem \ref{teorema principale} and Corollary \ref{corollario sobolev}. Finally, in the appendices \ref{appendice funzioni olomorfe}, \ref{appendiceA} and \ref{appendiceB} we collect some technical proofs of some lemmas that we use along our proofs. 

\bigskip

\noindent
{\sc Acknowledgements.} Riccardo Montalto is supported by INDAM-GNFM. Michela Procesi is supported by PRIN 2015, "Variational methods with applications to problems in Mathematical Physics and Geometry". The authors wish to thank L. Biasco, J. Massetti and E. Haus for helpful suggestions.
	
\section{Analytic functions on an infinite dimensional torus}\label{sez funzioni analitiche generali}

As is habitual in the theory of quasi-periodic functions we shall study  almost periodic functions in the context of analytic functions on an infinite dimensional torus. To this purpose, for $\zia,\s>0$, we define
 the  {\it thickened} infinite dimensional torus $\T^\infty_\s$ as 
\[
\f=(\f_j)_{j\in \N}\,,\quad \f_j\in\C\,:\; {\rm Re}(\f_j)\in \T\;,\;|{\rm Im}(\f_j)|\le \s \jap{j}^\zia\,. 
\]
Given a Banach space $(X, \| \cdot \|_X )$ we consider the space $\cF$ of 
pointwise absolutely convergent formal Fourier series $\T^\infty_\sigma \to X$
\begin{equation}
\label{fourier}
u(\f) = \sum_{\ell \in \Z^\infty_*} \widehat u(\ell) e^{\im \ell \cdot \f }\,,\quad \widehat u(\ell) \in X
\end{equation}
and define the analytic functions as follows.
\begin{definition}\label{funzioni analitiche T infty}
	Given a Banach space $(X, \| \cdot \|_X )$ and $\sigma > 0$, we define the space of analytic functions $\T^\infty_\sigma \to X$ as the  subspace 
	$$
	{\mathcal H}( \T^\infty_\sigma, X) := \Big\{ u(\f) = \sum_{\ell \in \Z^\infty_*} \widehat u(\ell) e^{\im \ell \cdot \f}\in \cF \; \;:\quad \| u \|_{\sigma} := \sum_{\ell \in \Z^\infty_*} e^{\s |\ell|_\zia} \| \widehat u(\ell)\|_X < \infty  \Big\}\,.
	$$
	In the case ${\mathcal H}( \T^\infty_\sigma, \C)$ se shall use the shortened notation ${\mathcal H}( \T^\infty_\sigma)$
\end{definition}

\begin{remark} We have chosen to work with an infinite torus $\T^\infty_\s$ whose angles are $\f_j$ with $j\in \N$ which in our notations does NOT contain $0$. Of course it would be completely equivalent to working on $\T_{\s}\times\T^\infty_{\s}$ with angles $\theta_j$ with $j\in \N_0:=\N\cup\{0\}$. 
	\\	
	To this purpose one just needs to define $ \widehat\Z^\infty_*:= \{ k\in \Z^{\N_0}\,:\; |k |_\zia:= \sum_{i \in \N_0}\jap{i}^\zia |k_i|<\infty \} =\Z\times \Z^\infty_*
	$ and consider Fourier series 
	\[
	u=	\sum_{ k\in \widehat\Z^\infty_*}\widehat u(k)e^{\im  k \cdot \theta}\quad \mbox{such that} 	\sum_{ k\in \widehat\Z^\infty_*}|\widehat u(k)|e^{\s|k|_\zia}<\infty.
	\]
\end{remark}
This notation is useful when working with the space ${\mathcal H}(\T^\infty_\sigma, {\mathcal H}(\T_\sigma))$ which can thus be identified with $  {\mathcal H}(\T_{\s}\times\T^\infty_{\s}, \C)\equiv {\mathcal H}(\T_{\s}\times\T^\infty_{\s})$. 
Indeed $u\in {\mathcal H}(\T^\infty_\sigma, {\mathcal H}(\T_\sigma))$ means
\[
u= \sum_{ \ell \in \Z^\infty_*}\widehat u(\ell,x)e^{\im  \ell \cdot \f}= \sum_{ (\ell,n )\in \Z^\infty_*\times \Z}\widehat u_n(\ell)e^{\im  \ell \cdot \f +\im n x}= \sum_{ k\in \widehat\Z^\infty_*}\widehat u(k)e^{\im  k \cdot \theta}
\]
where $\theta= (x,\f)\in\T_\s\times\T^\infty_\s $ and $k  =(n,\ell)$.

\medskip

With this definitions an almost-periodic function as in Definition \ref{suca} is the restriction of a function in ${\mathcal H}(\T^\infty_\sigma,X)$ to $\f=\omega t$.  Given $\cF\in{\mathcal H}(\T^\infty_\sigma,X)$ we define $f(t)= \cF(\omega t)$. Note that the condition $u\in {\mathcal H}( \T^\infty_\sigma, X)$ implies that the series in \eqref{fourier} is totally convergent for $\f\in\T^\infty_\s$. 

\medskip

\subsection{ Reformulation of the reducibility problem.}
In order to prove Thorem \ref{teorema principale}, we then  consider  analytic $\f$-dependent families of linear operators $ {\mathcal R} : \T^\infty_\sigma \to {\mathcal B}( L^2_0(\T_x))  $, 
$\f \mapsto {\mathcal R}(\f) $.
 Given a frequency vector $\omega \in \Ro$ and two operators ${\mathcal L}, \Phi : \T^\infty_\sigma \to {\mathcal B}(L^2_x)$, under the change of coordinates $u = \Phi(\omega t) v$, the dynamical system 
$$
\partial_t u =  {\mathcal L}(\omega t) u
$$
transforms into 
\begin{equation}\label{push forward}
\partial_t v =  {\mathcal L}_+(\omega t) u, \quad {\mathcal L}_+(\f) \equiv  (\Phi_{\omega*}){\mathcal L}(\f) := \Phi(\f)^{- 1} {\mathcal L}(\f) \Phi(\f) - \Phi(\f)^{- 1} \omega \cdot \partial_\f \Phi(\f)\,,
\end{equation}
where \footnote{ If we set $F(t)= \Phi(\omega t)$, since the  series expansion for $t\in \R$  is totally convegent we have 
	clearly $\partial_t F(t)= \omega\cdot\partial_{\f} \Phi (\omega t) \,.$}
\begin{equation}
\label{defdeomega}
\omega\cdot\partial_{\f} \Phi :=  \sum_{ \ell \in \Z^\infty_*} \im  (\ell\cdot \omega) \widehat \Phi(\ell) e^{\im \ell\cdot\f}\,.
\end{equation}
A direct calculation shows that if ${\mathcal L}(\omega t)$ is skew-self adjoint and $\Phi(\omega t)$ is unitary, then ${\mathcal L}_+(\omega t)$ is skew self-adjoint too. 
\\
In conclusion our goal is to prove the existence of maps
$\cW,\cW^{-1}\in \cH(\T^\infty_{\bar{\s}/4}, \cB(\cH(\T_\s),\cH(\T_{\s'}))$, such that $W(t)= \cW(\omega t)$ and $W(t)= \cW^{-1}(\omega t)$ which solve the reduction equation:
\begin{equation}
\label{riduci}
  \cW(\f)^{- 1} \im (\partial_{x}^2 + \e  \cP(\f)) \cW(\f) - \cW(\f)^{- 1} \omega \cdot \partial_\f \cW(\f) = \im \cD_\infty
\end{equation}
where
 the operator $\cP(\f)\in \cH(\T^\infty_\sigma, \cB(\cH(\T_\s),\cH(\T_{\s'})))$ is of the form $\cP(\f)= \cV_2( x, \f) \partial_x^2 +\cV_1( x, \f) \partial_x + \cV_0( x, \f)$  with $\cV_i\in {\mathcal H}(\T^\infty_\sigma,\cH(\T_\s))$ and is such that $P(t)= \cP(\omega t)$. 
Note that for $\f\in\T^\infty$, $(\partial_{x}^2 + \e  \cP(\f))$ is self-adjoint, hence $\cW(\f)$ is unitary.
\medskip
We remark that solving \eqref{riduci}  is equivalent to diagonalizing the linear  operator 
\[
  \im \omega \cdot\partial_{\f} + \partial_x^2 + \e \cP \in \cB(\cH(\T_{\s}\times\T^\infty_{\s},\C),\cH(\T^\infty_{\s'}\times\T_{\s'},\C))
\]
via a bounded change of variables with the special property that it is {\it T\"oplitz in time}. 
\medskip
\subsection{Properties  of analytic functions}
We now discuss  some fundamental properies of the space $\cH(\T^\infty_{\s},X)$, note  that all the results hold verbatim for $\cH(\T_\s\times\T_{\s}^\infty,X)$. For completeness, in the appendix \ref{appendice funzioni olomorfe}, we discuss another (equivalent) way of defining the space ${\mathcal H}(\T^\infty_\sigma, X)$ by approximation with holomorphic functions of a finite number of variables.
% the space of holomorphic functions on the thickened torus $\T^\infty_\sigma$. 

\smallskip

For any function $u \in {\mathcal H}(\T^\infty_\sigma, X)$, given $N > 0$, we define the projector $\Pi_N u$ as 
\begin{equation}
\Pi_N u(\f) := \sum_{|\ell|_\zia \leq N} \widehat u(\ell) e^{\ii \ell \cdot \f}\quad \text{and} \quad \Pi_N^\bot u := u - \Pi_N u\,.
\end{equation}
the following Lemma holds: 
\begin{lemma}\label{lemma astratto proiettore}
Let $\sigma, \rho > 0$, $u \in {\mathcal H}(\T^\infty_{\sigma + \rho}, X)$. Then the following holds: 
$$
\| \Pi_N^\bot u \|_{\sigma} \leq e^{- \rho N} \| u \|_{\sigma + \rho}\,.  
$$
\end{lemma}
\begin{proof}
One has 
$$
\| \Pi_N^\bot u \|_\sigma = \sum_{|\ell|_\zia > N} e^{\sigma |\ell|_\zia} \| \widehat u(\ell)\|_X \leq e^{- \rho N} \sum_{\ell \in \Z^\infty_*} e^{(\sigma + \rho)|\ell|_\zia} \| \widehat u(\ell)\|_X
$$
and the lemma follows. 
\end{proof}
\begin{lemma}\label{embedding L infty}
Let $\sigma > 0$, $u \in {\mathcal H}(\T^\infty_\sigma, X)$. Then $\| u \|_{L^\infty(\T^\infty_\sigma, X)} \leq \| u \|_\sigma$.
\end{lemma}
\begin{proof}
For any $\f \in \T^\infty_\sigma$, one has 
$$
\| u (\f) \|_X \leq \sum_{\ell \in \Z^\infty_*} \| \widehat u(\ell)\|_X e^{\sigma |\ell|_\zia} = \| u \|_\sigma\,. 
$$
\end{proof}
\begin{lemma}\label{Lemma prodotto}
Assume that $X$ is a Banach algebra and $u, v \in {\mathcal H}(\T^\infty_\sigma, X)$. Then $u v \in {\mathcal H} (\T^\infty_\sigma, X)$ and $\| u v \|_\sigma \leq \| u \|_\sigma \| v \|_\sigma$. 
\end{lemma}
\begin{proof}
One has 
$$
u(\f ) v(\f) = \sum_{\ell, k \in \Z^\infty_*} \widehat u(\ell - k) \widehat v(k) e^{\ii \ell \cdot \f}
$$
and therefore, one obtains that 
$$
\begin{aligned}
\| u v \|_\sigma & \leq \sum_{\ell, k \in \Z^\infty_*} e^{\sigma |\ell|_\zia} \| \widehat u(\ell - k)\|_X \| \widehat v(k)\|_X\,.
\end{aligned}
$$
Using the triangular inequality $|\ell|_\zia \leq |\ell - k|_\zia + |k|_\zia$, one gets $e^{\sigma |\ell|_\zia} \leq e^{\sigma |\ell - k|_\zia} e^{\sigma |k|_\zia}$, implying that 
$$
\| u v \|_\sigma \leq \sum_{\ell , k \in \Z^\infty_*} e^{\sigma |\ell - k|_\zia} \| \widehat u(\ell - k)\|_X  e^{\sigma |k|_\zia} \| \widehat v(k)\|_X \leq \| u \|_\sigma \| v \|_\sigma\,. 
$$

\end{proof}
\begin{lemma}
	Let $u\in {\mathcal H}(\T^\infty_\sigma, X)$. Then 
	\begin{equation}\label{proprieta media infinita}
	\int_{\T^\infty} u(\f)\, d \f := \lim_{N \to + \infty} \frac{1}{(2 \pi)^N} \int_{\T^N} u(\f) d \f_1 \ldots d \f_N = \widehat u(0)\,.
	\end{equation}
	Moreover, for any $\ell\in  \Z^\infty_* \setminus \{ 0 \}$:
	\begin{equation}\label{coeff fourier media infinita}
	\widehat u(\ell)= \int_{\T^\infty} u(\f) e^{- \ii \ell \cdot \f}\, d \f = \lim_{N \to \infty} \frac{1}{(2\pi)^N}\int_{\T^N} u(\f) e^{-\ii \ell\cdot\f}\,. 
	\end{equation}
\end{lemma}
\begin{proof}
	Let $\ell \in \Z^\infty_* \setminus \{ 0 \}$ and let $N^\zia \le |\ell|_\zia$. Then surely $\ell_j=0$ for all $j>N$, thus
	$$
	e^{\ii \ell \cdot \f} = e^{\ii \ell_1 \f_1} \ldots e^{\ii \ell_N \f_N}
	$$
	implying that 
	$$
	\frac{1}{(2 \pi)^N} \int_{\T^N} e^{\ii \ell \cdot \f}\, d \f_1 \ldots d \f_N = 0\,. 
	$$
Hence
	$$
	\begin{aligned}
	\frac{1}{(2 \pi)^N}\int_{\T^N}u(\f) d \f_1 \ldots d \f_N & = \frac{1}{(2 \pi)^N}\int_{\T^N} \Big(\widehat u(0) + \sum_{0 < |\ell|_\zia \leq N^\zia} \widehat u(\ell) e^{\ii \ell \cdot \f} + \sum_{|\ell|_\zia > N^\zia} \widehat u(\ell) e^{\ii \ell \cdot \f} \Big) d \f_1 \ldots d \f_N \\
	& = \widehat u(0) + \frac{1}{(2 \pi)^N}\int_{\T^N} \sum_{|\ell|_\zia > N^\zia} \widehat u(\ell) e^{\ii \ell \cdot \f}  d \f_1 \ldots d \f_N\,.
	\end{aligned}
	$$
	Since $u \in {\mathcal H}(\T^\infty_\sigma, X)$, the tail of the series $\sum_{|\ell|_\zia > N^\zia}$ goes to zero as $N \to \infty$. This proves \eqref{proprieta media infinita}.  
	
	\noindent
	Now let $\ell \in \Z^\infty_* \setminus \{ 0 \}$. Then we set
	$$
	u_\ell(\f) := u(\f) e^{- \ii \ell \cdot \f} = \sum_{k \in \Z^\infty_*} \widehat u(k) e^{\ii (k - \ell) \cdot \f} = \sum_{h \in \Z^\infty_*} \widehat u(h + \ell) e^{\ii h \cdot \f}\,.
	$$
	By applying the claim \eqref{proprieta media infinita} to the function $u_\ell$ and observing that $\widehat u_\ell(0) = \widehat u(\ell)$, the equality \eqref{coeff fourier media infinita} follows. 
\end{proof}
Given two Banach spaces $X$ and $Y$, for any $k \in \N$, we define the space ${\mathcal M}_k(X, Y)$ of the $k$-linear and continuous forms endowed by the norm 
\begin{equation}\label{norma operatori multilineari}
\| M \|_{{\mathcal M}_k(X, Y)} := \sup_{\| u_1\|_X, \ldots, \| u_k\|_X \leq 1} \|  M[u_1, \ldots, u_k]\|_Y, \quad \forall M \in {\mathcal M}_k(X, Y)\,. 
\end{equation}

 To shorthen notations, we denote $\ell^\infty:= \ell^\infty(\N,\C)$, moreover for  $k \in \N$, we write ${\mathcal M}_{k}$ instead of ${\mathcal M}_k(\ell^{\infty}, X)$ where $X$ is an arbitrary Banach space.
 
 Let us now discuss the differentiability of functions.
We define for $\widehat \f_1,\dots,\widehat{\f}_k\in \ell^\infty$ 
 \begin{equation}
 \label{kdiff}
 d^k_\f u[\widehat \f_1,\dots,\widehat \f_k] := \sum_{ \ell \in \Z^\infty_*} \im^k \prod_{j=1}^k (\ell\cdot \hat \f_j) \widehat u(\ell) e^{\im \ell\cdot\f}
 \end{equation}
 Note that  if $u \in \cH(\T^\infty_{\s+\rho},X)$ for any $\rho>0$ then the series in \eqref{kdiff} is totally convergent on $\T^\infty_\s$.
 \begin{lemma}[\bf Cauchy estimates]\label{stima di cauchy}
 Let $\sigma, \rho > 0$ and $u \in {\mathcal H}(\T^\infty_{\sigma + \rho}, X)$. Then for any $k \in \N$, the $k$-th differential $d^k_\f u$ satisfies the estimate 
$$
\| d^k_\f u \|_{{\mathcal H}(\T^\infty_\sigma, {\mathcal M}_{k})} \lesssim_k \rho^{- k} \| u \|_{\sigma + \rho}\,. 
$$
\end{lemma}
\begin{proof}
For any $k \in \N$, $\f \in \T^\infty_\sigma$, $\widehat \f_1, \ldots, \widehat \f_k \in \ell^{\infty}$, $\| \widehat \f_j \|_{\infty} \leq 1$ for any $j = 1, \ldots, k$, one has  by duality $|\ell \cdot \widehat \f| \leq \| \ell\|_1 \| \widehat \f\|_\infty \leq  |\ell|_\zia \| \widehat \f \|_{\infty} $, and substituting in \eqref{kdiff} one gets 
$$
\begin{aligned}
\| d^k_\f  u(\f)[\widehat \f_1, \ldots, \widehat \f_k] \|_{\sigma} & \leq \sum_{\ell \in \Z^\infty_*} |\ell|_\zia^k e^{\sigma |\ell|_\zia} \| \widehat u(\ell)\|_X \leq \sup_{\ell \in \Z^\infty_*} \Big( |\ell|_\zia^k e^{- \rho |\ell|_\zia} \Big) \| u \|_{\sigma + \rho} \,.
\end{aligned} 
$$
A straightforward calculation shows that 
$$
\sup_{\ell \in \Z^\infty_*} |\ell|_\zia^k e^{- \rho |\ell|_\zia} \leq \sup_{x \geq 0} x^k e^{- \rho x} = k^k \rho^{- k} e^{- k} \lesssim_k \rho^{- k}
$$
which implies the claimed estimate. 
\end{proof}
\begin{remark}
Note that if we endow the torus $\T^\infty_\s$ with the $\ell^\infty$ metric ,  namely given two angles $\f_1 = (\f_{1, j})_{j \in \N} \in \T^\infty_\sigma$ and $\f_2 = (\f_{2, j})_{j \in \N} \in \T^\infty_\sigma$, we define
\begin{equation}\label{distanza toro infinito}
d_{\infty}( \f_1, \f_2) := {\rm sup}_{j \in \N}\Big( \left| {\rm Re}({\f_{1, j} - \f_{2, j}})\right|_{\text{mod}\, 2\pi} + |{\rm Im}(\f_{1, j}) - {\rm Im}(\f_{2, j})| \Big)\,. 
\end{equation}
then \eqref{kdiff} is the $k$'th differential in the usual sense. Moreover  the tangent space to $\T^\infty_\s$ is $\ell^\infty(\C)$. 
\end{remark}

\medskip

Given a frequency vector $\omega \in \Ro$ and $u \in {\mathcal H}^\sigma(X)$, we define $\omega \cdot \partial_\f u$ as  in \ref{defdeomega}
\begin{equation}\label{definizione omega dot partial vphi}
\omega \cdot \partial_\f u (\f) := \sum_{\ell \in \Z^\infty_*} \im (\omega \cdot \ell )\widehat u(\ell) e^{\im \ell \cdot \f} = d u(\f)[\omega]\,. 
\end{equation}
If we set $f(t)= u(\omega t)$, since the  series expansion for $t\in \R$  is totally convegent we have 
 clearly $\partial_t f(t)= \omega\cdot\partial_{\f} u (\omega t) \,.$
 
 \medskip
 
The following Lemma holds
\begin{lemma}\label{lemma om dot partial vphi}
$(i)$ Let $\sigma, \rho > 0$, $u \in {\mathcal H}^{\sigma + \rho}(X)$, $\omega \in \Ro$. Then 
$$
\| \omega \cdot \partial_\f u \|_\sigma \lesssim  \rho^{- 1} \| u \|_{\sigma + \rho}\,.
$$
\end{lemma}
\begin{proof}
The lemma follows by the formula \eqref{definizione omega dot partial vphi} and by applying Lemma \ref{stima di cauchy} in a straightforward way. 
\end{proof}
{\bf Parameter dependence.} Let $Y$ be a Banach space and $\gamma \in (0, 1)$. If $f : \Oo \to Y$, $\Oo \subseteq \Ro := [1, 2]^\N$ is a Lipschitz function we define 
\begin{equation}\label{norme Lip gamma}
\begin{aligned}
& \| f \|_Y^{\rm sup} := \sup_{\omega \in \Oo } \| f(\omega)\|_Y, \quad \| f \|_Y^{\rm lip} := \sup_{\begin{subarray}{c}
	\omega_1, \omega_2 \in \Oo \\
	\omega_1 \neq \omega_2
	\end{subarray}} \frac{\| f(\omega_1) - f(\omega_2) \|_Y}{\| \omega_1 - \omega_2\|_\infty}\,, \\
& \| f \|_Y^\Lipg := \| f \|_Y^{\rm sup} + \gamma \| f \|_Y^{\rm lip}\,. 
\end{aligned}
\end{equation}
If $Y = {\mathcal H}(\T^\infty_\sigma, X)$ we simply write $\| \cdot \|_\sigma^{\rm sup}$, $\| \cdot \|_\sigma^{\rm lip}$, $\| \cdot \|_\sigma^{\Lipg}$. If $Y$ is a finite dimensional space, we write $\| \cdot \|^{\rm sup}$, $\| \cdot \|^{\rm lip}$, $\| \cdot \|^{\Lipg}$.

\vskip15pt
The following result follows directly
\begin{lemma}\label{stime lip}
In Lemmata \ref{lemma astratto proiettore}, \ref{Lemma prodotto},\ref{stima di cauchy}, \ref{lemma om dot partial vphi}, if $u(\cdot; \omega)$ is Lipschitz w.r. to $\omega \in \Oo \subseteq \Ro$, the same estimates hold verbatim   replacing  $\| \cdot \|_\sigma$ by $\| \cdot \|_\sigma^\Lipg$.
\end{lemma}
As is tipical in KAM reduction schemes, a fundamental tool in reducibility is to solve the "homological equation", i.e. to invert the operator
$\omega\cdot \partial_{\f}$. 
\begin{lemma}[\bf Homological equation] \label{stima diofantea D omega inv}
Let $\sigma, \rho > 0$, $f \in {\mathcal H}(\T^\infty_{\sigma + \rho}, X)$, $\omega \in \mathtt D_{\g,\mu} $ (see \eqref{diofantinoBIS}). with $\widehat f(0) = 0$. Then there exists a unique solution $u := (\omega \cdot \partial_\f)^{- 1} f  \in {\mathcal H}(\T^\infty_{\sigma}, X)$ of the equation 
$$
\omega \cdot \partial_\f u = f
$$
satisfying the estimates 
\begin{equation}\label{stima om partial f inv}
\| u \|_\sigma \lesssim {\rm exp}\Big(\frac{\tau}{\rho^{\frac{1}{\zia}}} \ln\Big(\frac{\tau}{\rho} \Big) \Big) \| f \|_{\sigma + \rho }\,. 
\end{equation}
for some constant $\tau = \tau(\zia, \mu) > 0$. If $f(\cdot; \omega) \in {\mathcal H}(\T^\infty_{\sigma + \rho},X)$ is Lipschitz w.r. to $\omega \in \Oo \subseteq \Dc$, then 
$$
\| u \|_\sigma^\Lipg \lesssim {\rm exp}\Big(\frac{\tau}{\rho^{\frac{1}{\zia}}} \ln\Big(\frac{\tau}{\rho} \Big) \Big) \| f \|_{\sigma + \rho }^\Lipg\,. 
$$
for some constant $\tau(\zia, \mu) > 0$ (eventually larger than the one in \eqref{stima om partial f inv}). 
\end{lemma}
\begin{proof}
Since $\omega \in \Dc$, the solution $u$ of the equation $\omega \cdot \partial_\f u = f$ is given by 
$$
u (\f) = (\omega \cdot \partial_\f)^{- 1} f(\f) = \sum_{\ell \in \Z^\infty_* \setminus \{ 0 \}} \frac{\widehat f(\ell)}{\ii \omega \cdot \ell} e^{\ii \ell \cdot \f}\,. 
$$
Hence, using that $\omega \in \tD_{\g,\mu}$
$$
\begin{aligned}
\| u \|_\sigma & \leq \gamma^{- 1} \sum_{\ell \in \Z^\infty_* \setminus \{ 0 \}} \prod_{i} (1 + \langle i \rangle^{\mu} |\ell_i|^{\mu}) \| \widehat f(\ell)\|_X e^{\sigma |\ell|_\zia} \\
& \leq \gamma^{- 1} \sup_{\ell \in \Z^\infty_*} \Big( e^{- \rho |\ell|_\zia} \prod_{i} (1 + \langle i \rangle^{\mu} |\ell_i|^{\mu}) \Big) \| f \|_{\sigma + \rho}\,. 
\end{aligned}
$$
and the claimed estimate follows by applying Lemma \ref{bound per stima di Cauchy}-$(i)$. Regarding the Lipschitz estimates we remark that
\[
u(\omega_1)- u(\omega_2) =  -\im \sum_{\ell \in \Z^\infty_* \setminus \{ 0 \}} \Big(\frac{ \widehat f(\ell,\omega_1) - \widehat f(\ell,\omega_2)}{ (\omega_2 \cdot \ell)}- \widehat f(\ell,\omega_1) \frac{(\omega_1-\omega_2)\cdot\ell}{{( \omega_2 \cdot \ell)(\omega_1 \cdot \ell)}} 
\Big )
e^{\ii \ell \cdot \f}
\]
\end{proof}
We conclude this section by discussing how the definition of ${\mathcal H}(\T^\infty_\sigma, X)$  (or equivalently  ${\mathcal H}(\T^\infty_\sigma\times\T_\s,  X)$) depends on the coordinates on $\T^\infty_\s$.
\begin{definition}
Recall $\ell^\infty:= \ell^\infty(\N,\C)$.	We say that a  function $a \in {\mathcal H}(\T^\infty_{\sigma + \rho})$ is {\it real on real} if $a(\f)\in\R$ for all $\f\in \T^\infty$. Similarly, 
	$\alpha \in {\mathcal H}(\T^\infty_{\sigma + \rho},\ell^\infty)$ is {\it real on real} if $\alpha_j(\f)\in\R$, for all $\f\in \T^\infty, j\in \N$.
\end{definition}
\begin{proposition}[\bf Torus diffeomorphism]\label{lemma diffeo inverso}
	Let $\alpha \in {\mathcal H}(\T^\infty_{\sigma + \rho}, \ell^\infty)$ be real on real. Then there exists $\e = \e(\rho) $ such that if $\| \alpha \|_{\sigma + \rho} \leq \e$, then the map $ \f \mapsto \f+ \alpha(\f)$ is an invertible diffeomorphism of the infinite dimensional torus $\T^\infty_\s$ (w.r. to the $\ell^\infty$-topology)  and its inverse is given by the map $\vartheta \mapsto \vartheta + \widetilde \alpha(\vartheta)$, where $\widetilde \alpha \in {\mathcal H}(\T^\infty_{\sigma+\frac\rho 2}, \ell^\infty)$ is real on real and satisfies the estimate $\| \widetilde \alpha \|_{\sigma+\frac{\rho}{2}} \lesssim  \| \alpha \|_{\sigma + \rho}$. Furthermore if $\alpha(\cdot; \omega) \in {\mathcal H}(\T_{\sigma + \rho}^\infty, \ell^\infty)$ is Lipschitz w.r. to $\omega \in \Oo \subseteq \Ro$, then $\| \widetilde \alpha \|_{\sigma +\frac{\rho}{2}}^\Lipg \lesssim  \| \alpha \|_{\sigma + \rho}^\Lipg$. 
\end{proposition}
\begin{corollary}
Given $\alpha \in {\mathcal H}(\T^\infty_{\sigma + \rho}, \ell^\infty)$  as in Theorem \ref{lemma diffeo inverso}, the operators
\begin{align}
\label{diffeofun}
&\Phi_\alpha : {\mathcal H}(\T^\infty_{\sigma + \rho}, X) \to {\mathcal H}(\T^\infty_\sigma, X), \quad u(\f) \mapsto u(\f + \alpha(\f))\,,\\
& \Phi_{\tilde\alpha} : {\mathcal H}(\T^\infty_{\sigma + \frac\rho 2}, X) \to {\mathcal H}(\T^\infty_\sigma, X), \quad u(\vartheta) \mapsto u(\vartheta + {\tilde\alpha}(\vartheta)\nonumber
\end{align}
are bounded, satisfy \[\| \Phi_\alpha \|_{{\mathcal B}\Big({\mathcal H}(\T^\infty_{\sigma + \rho}, X), {\mathcal H}(\T^\infty_\sigma, X) \Big)}, \| \Phi_{\tilde\alpha} \|_{{\mathcal B}\Big({\mathcal H}(\T^\infty_{\sigma + \rho}, X), {\mathcal H}(\T^\infty_\sigma, X) \Big)} \leq 1\,.\]
and for any $\f\in \T^\infty_\s$, $u\in {\mathcal H}(\T^\infty_{\sigma + \rho}, X), v\in {\mathcal H}(\T^\infty_{\sigma + \frac\rho 2}, X)$ one has 
\[
\Phi_{\tilde\alpha}\circ \Phi_{\alpha} u (\f)= u(\f) \,,\quad \Phi_\alpha\circ \Phi_{\tilde\alpha} v (\f)= u(\f)\,.
\]
\end{corollary}
 In order to prove our result we shall proceed in steps, proving a series of technical lemmata.
\begin{lemma}\label{composizione funzioni analitiche T infty}
For $\s,\rho>0$, let $u \in {\mathcal H} (\T^\infty_{\sigma + \rho}, X)$ and $\alpha \in {\mathcal H}(\T^\infty_\sigma, \ell^\infty)$ with $\| \alpha \|_{\sigma} \leq \rho$. Then the function $f(\f ) := u(\f + \alpha(\f)))$ belongs to the space ${\mathcal H}(\T^\infty_\sigma, X)$ and $\| f \|_\sigma \leq \| u \|_{\sigma + \rho}$. As a consequence, the linear operator 
$$
\Phi_\alpha : {\mathcal H}(\T^\infty_{\sigma + \rho}, X) \to {\mathcal H}(\T^\infty_\sigma, X), \quad u(\f) \mapsto u(\f + \alpha(\f))
$$
is  bounded and satisfies $\| \Phi_\alpha \|_{{\mathcal B}\Big({\mathcal H}(\T^\infty_{\sigma + \rho}, X), {\mathcal H}(\T^\infty_\sigma, X) \Big)} \leq 1$. 
\end{lemma}
\begin{proof}
One has that 
\begin{equation}\label{detersivo 1}
f (\f) = \sum_{ \ell \in \Z^\infty_*} \widehat u(\ell) e^{\ii \ell \cdot \f} e^{\ii \ell \cdot \alpha(\f)}\,.
\end{equation}
Moreover for any $\ell \in \Z^\infty_*$, one has 
\begin{equation}\label{detersivo 2}
\begin{aligned}
e^{\ii \ell \cdot \alpha(\f)} & = \sum_{n \in \N} \frac{\ii^n}{n!} (\ell \cdot \alpha(\f))^n = \sum_{n \in \N} \sum_{\ell_1, \ldots, \ell_n \in \Z^\infty_*} \frac{\ii^n}{n!} (\ell \cdot \widehat \alpha(\ell_1)) \ldots (\ell \cdot \widehat \alpha(\ell_n)) e^{\ii (\ell_1 + \ldots + \ell_n) \cdot \f} \,. 
\end{aligned}
\end{equation}
By the formulae \eqref{detersivo 1}, \eqref{detersivo 2} one then gets that 
\begin{equation}\label{espansione fourier composizione}
\begin{aligned}
& f(\f) = \sum_{k \in \Z^\infty_*} \widehat f(k) e^{\ii k \cdot \f}, \\
& \widehat f(k) :=\sum_{n \in \N} \frac{\ii^n}{n!} \sum_{\ell + \ell_1 + \ldots + \ell_n = k} (\ell \cdot \widehat \alpha(\ell_1)) \ldots (\ell \cdot \widehat \alpha(\ell_n)) \widehat u(\ell)\,.
\end{aligned}
\end{equation}
Using that for $k = \ell + \ell_1 + \ldots + \ell_n$, one has that $e^{\sigma |k|_\zia} \leq e^{\sigma |\ell|_\zia} e^{\sigma |\ell_1|_\zia} \ldots e^{\sigma |\ell_n|_\zia}$, and $|(\ell \cdot \widehat \alpha(\ell_i))| \le \|\ell\|_{1} \|\widehat \alpha(\ell_i)\|_\infty$ one gets that 
\begin{equation}\label{detersivo 3}
\begin{aligned}
\| f \|_\sigma & = \sum_{k \in \Z^\infty_*} e^{\sigma |k|_\zia} \| \widehat f(k)\|_X \\
& \leq \sum_{n \in \N} \frac{1}{n!} \sum_{\ell, \ell_1, \ldots, \ell_n \in \Z^\infty_*}  (\|\ell \|_1)^n e^{\sigma |\ell|_\zia} \| \widehat u(\ell)\|_X e^{\sigma |\ell_1|_\zia} \| \widehat \alpha(\ell_1)\|_{\infty} \ldots e^{\sigma |\ell_n|_\zia} \|\widehat \alpha(\ell_n) \|_{\infty} \\
& \stackrel{\| \ell \|_1 \leq |\ell|_\zia}{\leq} \sum_{\ell \in \Z^\infty_*} e^{\sigma |\ell|_\zia} \| \widehat u(\ell)\|_X \sum_{n \in \N} \frac{|\ell|_\zia^n}{n!} \prod_{j = 0}^n \sum_{\ell_j \in \Z^\infty_* }e^{\sigma |\ell_j|_\zia} \| \widehat \alpha(\ell_j)\|_{\infty}  \\
& \leq \sum_{\ell \in \Z^\infty_*} e^{\sigma |\ell|_\zia} \| \widehat u(\ell)\|_X \sum_{n \in \N} \frac{|\ell|_\zia^n \| \alpha\|^n_{\sigma}}{n!} \\
& \leq \sum_{\ell \in \Z^\infty_*} e^{\sigma |\ell|_\zia} \| \widehat u(\ell)\|_X {\rm exp}\Big(|\ell|_\zia \| \alpha\|_{\sigma} \Big) \\
& \stackrel{\| \alpha\|_{\sigma} \leq \rho}{\leq} \sum_{\ell \in \Z^\infty_*} e^{(\sigma + \rho) |\ell|_\zia} \| \widehat u(\ell)\|_X = \| u \|_{\sigma + \rho}\,. 
\end{aligned}
\end{equation}
\end{proof}
For $\alpha\in {\mathcal H}(\T^\infty_{\sigma + \rho}, \ell^{\infty})$ we now consider the map 
\begin{equation}\label{mappa punto fisso diffeo del toro}
\Psi_\alpha(u) (\f) := - \alpha(\f + u(\f))
\end{equation}
which, by Lemma \ref{composizione funzioni analitiche T infty} (with $\s\rightsquigarrow \s+\frac\rho 2$ and $\rho \rightsquigarrow \frac\rho 2$ ) is well defined  $ {\mathcal B}_{\sigma +\frac\rho 2}(0, R) \to {\mathcal H}(\T^\infty_{\sigma +\frac\rho 2}, \ell^{\infty})$, where 
\[
u\in {\mathcal B}_\sigma(0, R) := \Big\{ u \in {\mathcal H}(\T^\infty_\sigma, \ell^{\infty}) : \| u \|_{\sigma} \leq R \Big\}\,. 
\]
provided $R< \frac\rho 2$.
\begin{lemma}\label{punto fisso diffeo del toro 1}
Let $\alpha \in {\mathcal H}(\T^\infty_{\sigma + \rho}, \ell^{\infty})$. Then there exists $\e = \e(\rho)$ such that if $\| \alpha \|_{\sigma + \rho} \leq \e$, there exists a unique solution $u \in {\mathcal H}(\T^\infty_{\sigma +\frac\rho 2}, \ell^\infty)$ of the fixed point equation $u = \Psi_\alpha(u)$ satisfying the estimate $\| u \|_{\sigma +\frac\rho 2} \leq  \| \alpha \|_{\sigma + \rho}$. If $\alpha(\cdot; \omega) \in {\mathcal H}(\T^\infty_{\sigma + \rho}, \ell^\infty)$, $\omega \in \Oo \subseteq \Ro = [1, 2]^\N$ is Lipschitz, then
$\| u \|_{\sigma}^\Lipg \lesssim \| \alpha \|_{\sigma + \rho}^\Lipg$. 
\end{lemma}
\begin{proof}
To start with we show the following claim.
\begin{itemize}
\item{\sc Claim.} There exist $ \e = \e(\rho), R = R(\rho) > 0$ such that if $\| \alpha \|_{\sigma + \rho} \leq \e$, then the map \ref{mappa punto fisso diffeo del toro} 
is a contraction on  
$$
{\mathcal B}_\sigma(0, R) := \Big\{ u \in {\mathcal H}(\T^\infty_\sigma, \ell^{\infty}) : \| u \|_{\sigma} \leq R \Big\}\,. 
$$
\end{itemize}
{\sc Proof of the claim.}  By taking $R = R(\rho)$ sufficiently small, by applying Lemma \ref{composizione funzioni analitiche T infty}, one gets that for any $u \in {\mathcal B}_{\sigma +\frac\rho 2}(0, R)$, $\Psi_\alpha(u) \in {\mathcal H}(\T^\infty_{\sigma +\frac\rho 2}, \ell^\infty)$ and $\| \Psi_\alpha(u) \|_{\sigma +\frac\rho 2} \leq \| \alpha\|_{\sigma + \rho}$. Then, if $\| \alpha \|_{\sigma + \rho} \leq \e \leq R$, one has that $\Psi_\alpha : {\mathcal B}_{\sigma +\frac\rho 2}(0, R) \to {\mathcal B}_{\sigma +\frac\rho 2}(0, R)$. Now, given $u_1, u_2 \in {\mathcal B}_{\sigma +\frac\rho 2}(0, R)$, we want to bound $\| \Psi_\alpha(u_1) - \Psi_\alpha(u_2)\|_{\sigma}$. By the mean value thoerem, one has 
\begin{equation}\label{nutella 10}
\begin{aligned}
\Psi_\alpha(u_1) - \Psi_\alpha(u_2) = \int_0^1 d_\f \alpha\Big(\f +  t u_1(\f) + (1 - t) u_2(\f)  \Big)[u_2 - u_1]\, d t \,.
\end{aligned}
\end{equation}
Since $\| u_1 \|_{\sigma +\frac\rho 2}, \| u_2\|_{\sigma +\frac\rho 2} \leq R$, by taking $R \leq \frac{\rho}{4}$, by Lemmata \ref{stima di cauchy} and \ref{composizione funzioni analitiche T infty} one has the estimate 
\begin{equation}\label{nutella 11}
\begin{aligned}
\| \Psi_\alpha(u_1) - \Psi_\alpha(u_2) \|_{\sigma +\frac\rho 2} & \leq \| d_\f \alpha\|_{{\mathcal H}(\T^\infty_{\s+\frac{3\rho}{2}}, \cM_{1})} \| u_1 - u_2\|_{\sigma +\frac\rho 2}   \\
& \lesssim \rho^{- 1} \| \alpha\|_{\sigma + \rho} \| u_1 - u_2\|_{\sigma +\frac\rho 2}
\end{aligned}
\end{equation}
Hence by taking $ \| \alpha\|_{{\sigma + \rho}} \leq  \e(\rho) $ small enough, one gets that the map $\Psi_\alpha$ is a contraction and by recalling Lemma \ref{composizione funzioni analitiche T infty} the unique solution of the fixed point equation satisfies $\| u \|_{\sigma +\frac\rho 2} \leq \| \alpha\|_{\sigma + \rho}$. Now assume that $\alpha(\cdot; \omega)$, $\omega \in \Oo$ is Lipschitz w.r. to $\omega$. Recalling the definition \eqref{mappa punto fisso diffeo del toro} and using the fixed point equation $u = \Psi_\alpha(u)$, one computes for any $\omega_1, \omega_2 \in \Oo$
$$
\begin{aligned}
\Delta_{\omega_1 \omega_2} u (\f) & = \alpha(\f + u(\f; \omega_1); \omega_1) - \alpha(\f + u(\f; \omega_2); \omega_2) \\
& = \alpha(\f + u(\f; \omega_1); \omega_1) - \alpha(\f + u(\f; \omega_1); \omega_2) \\
& \quad + \alpha(\f + u(\f; \omega_1); \omega_2) - \alpha(\f + u(\f; \omega_2); \omega_2)\,. 
\end{aligned}
$$
By taking $R = R(\rho)$ small enough, using the mean value Theorem, the Cauchy estimates of Lemma \ref{stima di cauchy} and the composition Lemma \ref{composizione funzioni analitiche T infty}, one gets $$\| \Delta_{\omega_1 \omega_2} u \|_{\sigma +\frac\rho 2} \leq \| \Delta_{\omega_1 \omega_2} \alpha\|_{\sigma + \rho} + C(\rho) \sup_{\omega \in \Oo}\| \alpha(\cdot; \omega)\|_{\sigma + \rho} \| \Delta_{\omega_1 \omega_2} u \|_{\sigma +\frac\rho 2}\,.$$
Hence, by taking $C(\rho) \sup_{\omega \in \Oo}\| \alpha(\cdot; \omega)\|_{\sigma + \rho} \leq \frac12$, one gets that $\| \Delta_{\omega_1 \omega_2} u \|_{\sigma +\frac\rho 2} \leq 2 \| \Delta_{\omega_1 \omega_2} \alpha\|_{\sigma + \rho}$ and the claimed Lipschitz estimate follows. 
\end{proof}

\begin{proof}[Proof of Proposition \ref{lemma diffeo inverso}]
Clearly the map $\f  \mapsto \f + \alpha(\f)$ is invertible by taking $ \| \alpha \|_{\sigma + \rho}  \leq \e$ small enough. By applying Lemma \ref{punto fisso diffeo del toro 1} there exists a unique $\widetilde \alpha \in {\mathcal H}(\T^\infty_{\sigma +\frac\rho 2}, \ell^{\infty})$ with $\| \widetilde \alpha \|_{\sigma +\frac\rho 2} \lesssim \| \alpha\|_{\sigma + \rho}$ satisfying the equation
$$
\widetilde \alpha(\vartheta) + \alpha(\vartheta + \widetilde \alpha(\vartheta)) = 0
$$
for $\vartheta\in \T^\infty_{\sigma +\frac\rho 2}$. The same holds exchanging $\vartheta\rightsquigarrow \f$ and $\alpha\rightsquigarrow\widetilde{\alpha}$ for $\f\in \T^\infty_\s$. Hence $\vartheta \mapsto \vartheta + \widetilde \alpha(\vartheta)$ is the inverse of $\f \mapsto \f + \alpha(\f)$ and viceversa and the proof  is concluded. 
\end{proof}
\section{Linear operators}\label{sezione astratta op lineari}
Given a linear operator ${\mathcal R} : L^2(\T) \to L^2(\T)$, we identify it with its matrix representation $({\mathcal R}_k^{k'})_{k, k' \in \Z}$ with respect to the exponential basis  where 
$$
{\mathcal R}_k^{k'} := \frac{1}{2 \pi} \int_\T {\mathcal R}[e^{\im k' x} ] e^{- \im k x}\, d x\,. 
$$
Clearly given ${\mathcal R}$ as above, the adjoint w.r.t the standard hermitian product in $ L^2(\C)$ is given by \begin{equation}
\label{aggiunto}
{(\mathcal R^*)}_k^{k'} = \overline{\mathcal R}_{ k'}^{ k}\,.
\end{equation}

We may also give a block-matrix  decomposition by grouping together the matrix-Fourier indices with the same absolute values. More precisely, we define for any $j \in \N_0$ the space ${\bf E}_j$  as 
\begin{equation}\label{bf E alpha}
\begin{aligned}
& {\bf E}_0 := {\rm span}\{ 1 \}\,,  \\
& {\bf E}_j := {\rm span} \{ e^{\im j x}, e^{- \im j x}\}, \quad \forall j \in \N
\end{aligned}
\end{equation}

and we define the corresponding projection operator $\Pi_j$ as 
\begin{equation}\label{def Pi j}
\begin{aligned}
& \Pi_0 : L^2(\T) \to L^2(\T), \quad u(x) = \sum_{j \in \Z} \widehat u(j) e^{\ii j x} \mapsto \Pi_0 u(x) := \widehat u(0)\,, \\
& \Pi_j : L^2(\T) \to L^2(\T), \quad u(x) = \sum_{j \in \Z} \widehat u(j) e^{\ii j x} \mapsto \Pi_j u(x) := \widehat u(j) e^{\ii j x} + \widehat u(- j) e^{- \ii j x}\,, \quad j \in \N\,. 
\end{aligned}
\end{equation}
The following properties follow directly from the definitions \eqref{bf E alpha}, \eqref{def Pi j}:
\begin{equation}\label{proprieta Pi j E j}
\begin{aligned}
& \Pi_j^2 = \Pi_j, \quad \forall j \in \N_0, \quad \Pi_j \Pi_{j'} = 0, \quad \forall j, j' \in \N_0, \quad j \neq j'\,, \\
& \sum_{j \in \N_0} \Pi_j = {\rm Id}, \quad L^2(\T) = \oplus_{j \in \N_0} {\bf E}_j\,. 
\end{aligned}
\end{equation}
Hence, any linear operator ${\mathcal R} : L^2(\T) \to L^2(\T)$ can be written in $2 \times 2$ {\it block}-decomposition
\begin{equation}\label{notazione a blocchi}
{\mathcal R} = \sum_{j, j' \in \N_0} \Pi_j {\mathcal R} \Pi_{j'} \,.
\end{equation}
 where $j, j' \in \N_0$ the operator $\Pi_j {\mathcal R} \Pi_{j'}$ is a linear operator in ${\mathcal B}({\bf E}_{j'}, {\bf E}_j)$. If $j, j' \in \N$, the operator $\Pi_j {\mathcal R} \Pi_{j'}$ can be identified with the $2 \times 2$ matrix defined by  
\begin{equation}\label{definizione blocco operatore}
 \begin{pmatrix}
{\mathcal R}_j^{j'} & {\mathcal R}_j^{- j'} \\
{\mathcal R}_{- j}^{j'}& {\mathcal R}_{- j}^{- j'} 
\end{pmatrix}\,.
\end{equation}
The action of any linear operator $M \in {\mathcal B}({\bf E}_{j'}, {\bf E}_j)$, $j, j' \in \N$ is given by  
\begin{equation}\label{azione blocco finito dimensionale su funzioni}
M u (x)= \sum_{\begin{subarray}{c}
k = \pm j \\
k' = \pm j'
\end{subarray}} M_k^{k'} \widehat u(k') e^{\im k  x}\,, \quad \forall u \in {\bf E}_{j'}\,, \quad u(x) = \widehat u(j') e^{\im j' x} + \widehat u( - j') e^{- \im j' x}\,. 
\end{equation}
The operator $\Pi_0 {\mathcal R} \Pi_0 \in {\mathcal B}({\bf E}_0, {\bf E}_0)$ is identified with the multiplication operator by the matrix element ${\mathcal R}_0^0$ and if $j, j' \in \N$, the operators $\Pi_j {\mathcal R} \Pi_0$, $\Pi_0 {\mathcal R} \Pi_j$ are identified with the vectors
$$
\begin{pmatrix}
{\mathcal R}_j^0 \\
{\mathcal R}_{- j}^0 
\end{pmatrix} \quad \text{and} \quad \big( {\mathcal R}_0^{j'}, {\mathcal R}_0^{- j'} \big)\,. 
$$
 We denote by $[{\mathcal R}]$ the block-diagonal part of the operator ${\mathcal R}$, namely
\begin{equation}\label{notazione operatore diagonale a blocchi}
[{\mathcal R}] : = {\sum}_{j \in \N_0} \Pi_j {\mathcal R} \Pi_j\,. %{\bf R}_j^{j}\,. 
\end{equation}
If $\Pi_j {\mathcal R} \Pi_{j'} = 0$, for any $j \neq j'$, we have ${\mathcal R} = [{\mathcal R}]$ and we refer to such operators as $2 \times 2$ block-diagonal operators. 
Note that for any $j, j' \in \N_0$, the adjoint operator $M^* \in {\mathcal B}({\bf E}_j, {\bf E}_{j'})$ is thus defined as \footnote{If $j, j' \in \N$, $A \in {\mathcal B}({\bf E}_0, {\bf E}_0)$, $B \in {\mathcal B}({\bf E}_{j'}, {\bf E}_0) $, $C \in {\mathcal B}({\bf E}_0, {\bf E}_j)$, then 
$$
(A^*)_0^0 :=  \overline{A_0^0}\,, \quad (B^*)_k^0 = \overline{B_0^k}, \quad k = \pm j', \quad (C^*)_0^k = \overline{C_k^0}, \quad k = \pm j\,. 
$$}
\begin{equation}\label{definizione matrice aggiunta}
(M^*)_k^{k'} :=  \overline{M_{ k'}^{  k}}\,.
\end{equation}

We denote by ${\mathcal S}({\bf E}_j)$ the space of self-adjoint matrices in ${\mathcal B}({\bf E}_j, {\bf E}_{j})$.
\\
For any $j, j' \in \N_0$, we endow ${\mathcal B}({\bf E}_{j'}, {\bf E}_{j})$ with the {\it Hilbert-Schmidt} norm
\begin{equation}\label{norma L2 blocco}
\| X \|_{\mathtt{HS}} := \sqrt{{\rm Tr}(X X^*)} = \Big( \sum_{\begin{subarray}{c}
	|k| = j \\
	|k'| = j'
	\end{subarray}} |X_k^{k'}|^2 \Big)^{\frac12}\,.
\end{equation}
For any $\sigma > 0$, $m \in \R$ we define the class of linear operators of order $m$ (densely defined on $L^2(\T)$) ${\mathcal B}^{\sigma, m}$ as 
\begin{equation}\label{definizione classe cal B sigma}
\begin{aligned}
& {\mathcal B}^{\sigma, m} := \Big\{ {\mathcal R} : L^2(\T) \to L^2(\T) : \| {\mathcal R}\|_{{\mathcal B}^{\sigma, m}} < \infty \Big\} \quad \text{where} \\
& \quad \| {\mathcal R}\|_{{\mathcal B}^{\sigma, m}} := \sup_{j' \in \N_0} \sum_{j \in \N_0} e^{\sigma|j - j'| } \| \Pi_j {\mathcal R} \Pi_{j'}\|_{\HS} \langle j' \rangle^{- m} \,. 
\end{aligned}
\end{equation}
The following monotonicity properties hold: 
\begin{equation}\label{prop monotonia norma}
\begin{aligned}
& \| {\mathcal R} \|_{{\mathcal B}^{\sigma, m}} \leq \| {\mathcal R} \|_{{\mathcal B}^{\sigma', m}}, \quad \sigma < \sigma'\,, 
\quad \| {\mathcal R} \|_{{\mathcal B}^{\sigma, m}} \leq \| {\mathcal R} \|_{{\mathcal B}^{\sigma,  m'}}, \quad m' \leq m\,. 
\end{aligned}
\end{equation}
As a notation, if $m = 0$, we write ${\mathcal B}^\sigma$ instaead of ${\mathcal B}^{\sigma, 0}$. Note that a direct consequence of the definition is that if ${\mathcal R} \in {\mathcal B}^{\sigma, m}$ then  (recall that $D=-\im \partial_x$)
\begin{equation}\label{norma m norma 0 matrici}
\|{\mathcal R}\|_{{\mathcal B}^{\sigma, m}} = \| {\mathcal R} \langle D \rangle^{- m} \|_{{\mathcal B}^\sigma}\,. 
\end{equation}
Note that ${\mathcal B}^\sigma$ is contained in the set of bounded linear operators ${\mathcal B}({\mathcal H}(\T_\sigma), {\mathcal H}(\T_\sigma))$ as shown in the following.
\begin{lemma}\label{lemma azione nostri pseudo}
	Let $\sigma > 0$ and $\Phi \in {\mathcal B}^\sigma$. Then 
	
	\noindent
	$(i)$ $ \| \Phi\|_{{\mathcal B}({\mathcal H}(\T_\sigma), {\mathcal H}(\T_\sigma))} \leq \|\Phi \|_{{\mathcal B}^\sigma}$
	
	\noindent
	$(ii)$ For any $s \geq 0$, $\| \Phi\|_{{\mathcal B}(H^s(\T), H^s(\T))} \lesssim_s \sigma^{- s} \|\Phi \|_{{\mathcal B}^\sigma}$. 
\end{lemma}
\begin{proof}
	{\sc Proof of $(i)$} Let $\Phi \in {\mathcal B}^\sigma$. According to \eqref{def Pi j}, \eqref{notazione a blocchi}, given $u \in {\mathcal H}(\T_\sigma)$, we write $\Phi u(x) = \sum_{j, j' \in \N_0} \Pi_j \Phi \Pi_{j'}[\Pi_{j'} u]$. Then, using that for any $j, j' \in \N_0$, $e^{\sigma |j|} \leq e^{\sigma |j - j'|} e^{\sigma |j'|}$, one gets the chain of inequalities 
	$$
	\begin{aligned}
	\| \Phi u \|_\sigma & = \sum_{j \in \N_0} e^{\sigma |j|} \Big\| \sum_{j' \in \N_0}  \Pi_j \Phi \Pi_{j'}[\Pi_{j'} u] \Big\|_{L^2} \\
	& \leq \sum_{ j' \in \N_0} e^{\sigma|j'|} \| \Pi_{j'} u \|_{L^2} \Big( \sum_{j \in \N_0} e^{\sigma |j - j'|} \| \Pi_j \Phi \Pi_{j'}\|_\HS \Big) \\
	& \leq \sup_{j' \in \N_0} \Big( \sum_{j \in \N_0} e^{\sigma |j - j'|} \| \Pi_j \Phi \Pi_{j'}\|_\HS \Big) \| u \|_\sigma \stackrel{\eqref{definizione classe cal B sigma}}{\leq} \| \Phi\|_{{\mathcal B}^\sigma} \| u \|_\sigma\,. 
	\end{aligned}
	$$
	{\sc Proof of $(ii)$.} Let $s \geq 0$ and $u \in H^s(\T)$. Then, using that for any $j, j' \in \N_0$, $\langle j \rangle \lesssim \langle j' \rangle + \langle j - j' \rangle \lesssim \langle j' \rangle \langle j - j' \rangle$, one gets that 
	$$
	\begin{aligned}
	\| \Phi u \|_{H^s}^2 & = \sum_{j \in \N_0} \langle j \rangle^{2 s} \Big\| \sum_{j' \in \N_0} \Pi_j \Phi \Pi_{j'}[\Pi_{j'} u] \Big\|_{L^2}^2 \leq \sum_{j \in \N_0}  \Big\| \sum_{j' \in \N_0} \langle j \rangle^{ s}\Pi_j \Phi \Pi_{j'}[\Pi_{j'} u] \Big\|_{L^2}^2 \\
	& \lesssim_s \sum_{j \in \N_0} \Big( \sum_{j' \in \N_0} \langle j' \rangle^s \langle j - j' \rangle^s \| \Pi_j \Phi \Pi_{j'}\|_\HS \| \Pi_{j'} u\|_{L^2} \Big)^2
	\end{aligned}
	$$
	Moreover, by using the Cauchy-Schwartz inequality, one gets 
	$$
	\begin{aligned}
	\| \Phi u \|_{H^s}^2 & \lesssim_s \sum_{j' \in \N_0} \langle j' \rangle^{2 s} \| \Pi_{j'} u\|_{L^2}^2 \sum_{j \in \N_0} \langle j - j' \rangle^{2(s + 1)} \| \Pi_j \Phi \Pi_{j'}\|_\HS^2 \\
	& \stackrel{\eqref{definizione classe cal B sigma}}{\lesssim_s} \sup_{k \in \N_0} \langle k \rangle^{2 (s + 1)} e^{- \sigma |k|} \| \Phi\|_{{\mathcal B}^\sigma}  \| u \|_{H^s} \lesssim_s \sigma^{- s} \| \Phi\|_{{\mathcal B}^\sigma} \| u \|_{H^s}
	\end{aligned}
	$$
	which proves the claimed estimate. 
\end{proof}
Further properties of $\cB^{\s,m}$ can be found in the appendix \ref{linop}.
\subsection{T\"oplitz in time linear operators}\label{astratto operatori Toplitz}
We now consider $\f$-dependent families of linear operators on $L_2(\T)$ i.e. absolutely convergent Fourier series  $\T^\infty_\s \to L^2_0(\T)$.
\begin{definition}
	For $\sigma > 0$, $m \in \R$ we consider  ${\mathcal R} \in {\mathcal H}(\T^\infty_\sigma, {\mathcal B}^{\sigma, m})$. We define the decay norm 
\begin{equation}\label{definizione norma del decay}
|{\mathcal R}|_{\sigma, m} := \sum_{\ell \in \Z^\infty_*} e^{\sigma |\ell|_\zia} \| \widehat{\mathcal R}(\ell)\|_{{\mathcal B}^{\sigma, m}}\,. 
\end{equation}
Moreover, given $\gamma \in (0, 1)$ and if ${\mathcal R} = {\mathcal R}(\f; \omega)$ depends on the parameter $\omega \in \Oo$, we define 
\begin{equation}\label{definizione norma lip gamma matrici}
\begin{aligned}
& |{\mathcal R}|_{\sigma, m}^\Lipg := \sup_{\omega \in \Oo}|{\mathcal R}(\omega)|_{\sigma, m} + \gamma |{\mathcal R}|^{\rm lip}_{\sigma, m+ 2}\,, \\
& |{\mathcal R}|^{\rm lip}_{\sigma, m+ 2} := \sup_{\begin{subarray}{c}
\omega_1, \omega_2 \in \Oo \\
\omega_1 \neq \omega_2
\end{subarray}} \frac{|{\mathcal R}(\omega_1) - {\mathcal R}(\omega_2)|_{\sigma , m + 2}}{\| \omega_1 - \omega_2 \|_{\infty}}\,.
\end{aligned}
\end{equation}
\end{definition}
If $m = 0$ we write $| \cdot |_\sigma$ instead of $| \cdot |_{\sigma, m}$. By recalling \eqref{prop monotonia norma}, one can easily see that the following properties hold: 
\begin{equation}\label{prop monotonia norma vphi x}
\begin{aligned}
& |\cdot |_{\sigma, m} \leq |\cdot |_{\sigma', m}, \quad |\cdot |_{\sigma, m}^\Lipg \leq |\cdot |_{\sigma', m}^\Lipg \quad \forall \sigma \leq \sigma'\,, \\
& |\cdot |_{\sigma, m} \leq |\cdot |_{\sigma, m'}, \quad |\cdot |_{\sigma, m}^\Lipg \leq |\cdot |_{\sigma, m'}^\Lipg \quad \forall m' \leq m\,. 
\end{aligned}
\end{equation}
\begin{definition}
	We say that ${\mathcal R} \in {\mathcal H}(\T^\infty_\sigma, {\mathcal B}^{\sigma, m})$ is self-adjoint (resp. skew self-adjoint or unitary) if for all $\f\in\T^\infty$, the operato $\cR(\f)\in {\mathcal B}^{\sigma, m}$ is self-adjoint (resp. skew self-adjoint or unitary ).
\end{definition}
\begin{lemma}\label{proprieta norma sigma riducibilita vphi x}
	 Let $N,\sigma, \rho > 0$, $m, m' \in \R$ ${\mathcal R} \in {\mathcal H}(\T^\infty_\sigma, {\mathcal B}^{\sigma, m})$, ${\mathcal Q} \in {\mathcal H}(\T^\infty_{\sigma + \rho}, {\mathcal B}^{\sigma + \rho, m'})$.
	 \\
$(i)$ The product operator ${\mathcal R} {\mathcal Q} \in {\mathcal H}(\T^\infty_\sigma, {\mathcal B}^{\sigma, m + m'})$ with $|{\mathcal R} {\mathcal Q}|_{\sigma, m + m'} \lesssim_m  \rho^{- |m|} |{\mathcal R}|_{\sigma, m} |{\mathcal Q}|_{\sigma + \rho, m'}$. If ${\mathcal R}(\omega), {\mathcal Q}(\omega)$ depend on a parameter $\omega \in \Oo \subseteq \Ro$, then $|{\mathcal R}{\mathcal Q}|_{\sigma, m + m'}^\Lipg \lesssim_m \rho^{- (|m| + 2)} |{\mathcal R}|_{ \sigma, m}^\Lipg |{\mathcal Q}|_{\sigma + \rho, m'}^\Lipg$.

\smallskip

\noindent
$(ii)$ The projected operator $|\Pi_N^\bot {\mathcal R}|_{\sigma, m} \leq e^{- \rho N} |{\mathcal R}|_{\sigma + \rho, m}$. If ${\mathcal R}(\omega)$ depends on a parameter $\omega \in \Oo \subseteq  \Ro$, then
the same statement holds by replacing $|\cdot |_{\sigma, m}$ with $|\cdot |_{\sigma, m}^\Lipg$.

\smallskip

\noindent
$(iii)$ The mean value  $|[\widehat{\mathcal R}(0)]|_{\sigma, m} \leq |{\mathcal R}|_{\sigma, m}$. Moreover, if ${\mathcal R}= {\mathcal R}(\omega)$ depends on a parameter $\omega \in \Oo \subseteq \Ro$,  then
the same statement holds by replacing $|\cdot |_{\sigma, m}$ with $|\cdot |_{\sigma, m}^\Lipg$.
\end{lemma}
\begin{proof}
{\sc Proof of $(i)$.} We write 
$$
{\mathcal R}(\f){\mathcal Q}(\f) = \sum_{\ell, k \in \Z^\infty_*} \widehat{\mathcal R}(\ell - k) \widehat{\mathcal Q}(k) e^{\ii \ell \cdot \f}\,. 
$$
Using that by triangular inequality, for any $\ell, k \in \Z^\infty_*$, $e^{\sigma|\ell|_\zia} \leq e^{\sigma |\ell - k|_\zia} e^{\sigma |k|_\zia}$ 
$$
\begin{aligned}
|{\mathcal R} {\mathcal Q} |_{\sigma , m + m'} & \leq \sum_{\ell, k \in \Z^\infty_*}  e^{\sigma |\ell - k|_\zia} e^{\sigma |k|_\zia}\| \widehat{\mathcal R}(\ell - k) \widehat{\mathcal Q}(k)\|_{{\mathcal B}^{\sigma, m + m'}}  \\
& \stackrel{Lemma \ref{B algebra cal B sigma}-(i)}{\lesssim} \rho^{- |m|} \sum_{\ell, k \in \Z^\infty_*} e^{\sigma |\ell - k|_\zia} e^{\sigma |k|_\zia} \| \widehat{\mathcal R}(\ell - k)\|_{{\mathcal B}^{\sigma, m}} \|\widehat{\mathcal Q}(k) \|_{{\mathcal B}^{\sigma + \rho, m'}} \\
& \lesssim \rho^{- |m|} |{\mathcal R}|_{\sigma, m} |{\mathcal Q}|_{\sigma + \rho, m'}\,. 
\end{aligned}
$$
 
Now we prove the Lipschitz estimate. Given $\omega_1, \omega_2 \in \Ro$, we use the notation $\Delta_{\omega_1 \omega_2} f := f(\omega_1) - f(\omega_2)$. One has that 
$$
\Delta_{\omega_1 \omega_2} ({\mathcal R}{\mathcal Q}) = (\Delta_{\omega_1 \omega_2} {\mathcal R}){\mathcal Q}(\omega_1) + {\mathcal R}(\omega_2)(\Delta_{\omega_1 \omega_2} {\mathcal Q})\,.
$$
Hence by the previous estimate one gets 
$$
\begin{aligned}
|\Delta_{\omega_1 \omega_2} ({\mathcal R}{\mathcal Q})|_{\sigma, m + m' + 2} & \lesssim_{m} \rho^{- |m| - 2} |\Delta_{\omega_1\omega_2} {\mathcal R}|_{\sigma , m + 2} |{\mathcal Q}(\omega_1)|_{m', \sigma + \rho} + \rho^{- |m|} |{\mathcal R}|_{\sigma, m} |\Delta_{\omega_1 \omega_2} {\mathcal Q}|_{m' + 2, \sigma + \rho} \\
& \stackrel{\eqref{definizione norma lip gamma matrici}}{\lesssim_m} \rho^{- (|m| + 2)} |{\mathcal R}|_{\sigma, m}^\Lipg |{\mathcal Q}|_{\sigma + \rho, m'}^\Lipg \| \omega_1 - \omega_2\|_\infty\,.
\end{aligned}
$$
The claimed statement then follows. 

\noindent
{\sc Proof of $(ii)$.} The proof is the same as the one of Lemma \ref{lemma astratto proiettore}. 

\noindent
{\sc Proof of $(iii)$.} By recalling the definitions \eqref{notazione operatore diagonale a blocchi}, \eqref{definizione norma del decay}, \eqref{definizione norma lip gamma matrici}, one obtains that
$$
\begin{aligned}
& |[\widehat{\mathcal R}(0)]|_{\sigma, m} = \sup_{j \in \N_0} \| \Pi_j \widehat{\mathcal R}(0) \Pi_j\| \langle j \rangle^{- m}\,, \\
& |[\widehat{\mathcal R}(0)]|_{\sigma, m + 2}^{\rm lip} = {\rm sup}_{\begin{subarray}{c}
\omega_1, \omega_2 \in \Ro \\
\omega_1 \neq \omega_2
\end{subarray}} \frac{1}{\| \omega_1 - \omega_2 \|_\infty} \sup_{j \in \N_0} \| \Pi_j \big(\Delta_{\omega_1\omega_2} \widehat{\mathcal R}(0) \big) \Pi_j\| \langle j \rangle^{- m - 2}\,. 
\end{aligned}
$$ 
Hence, one has that $|[\widehat{\mathcal R}(0)]|_{\sigma, m} \leq |{\mathcal R}|_{\sigma, m}$ and $|[\widehat{\mathcal R}(0)]|_{\sigma, m + 2}^{\rm lip} \leq |{\mathcal R}|_{\sigma, m + 2}^{\rm lip}$ which implies the claimed statement. 
\end{proof}
Iterating the estimate of Lemma \ref{proprieta norma sigma riducibilita vphi x}-$(i)$, one has that if ${\mathcal R} \in {\mathcal H}^{\sigma + \rho}({\mathcal B}^{\sigma + \rho, m})$, then there exists a constant $C_0(m) > 0$ such that for any $N \geq 1$, ${\mathcal R}^N \in {\mathcal H}^\sigma({\mathcal B}^{\sigma, m N})$ and 
\begin{equation}\label{stima comp iterata norma dec}
\begin{aligned}
& |{\mathcal R}^N|_{\sigma, m N} \leq  \Big(C_0(m) \rho^{- |m|} |{\mathcal R}|_{\sigma + \rho, m} \Big)^{N - 1} |{\mathcal R}|_{\sigma, m}\,, \\
& |{\mathcal R}^N|_{\sigma, m N}^\Lipg \leq \Big( C_0(m)^{N - 1} \rho^{- (|m| + 2)} |{\mathcal R}|_{\sigma + \rho, m}^\Lipg \Big)^{N - 1} |{\mathcal R}|_{\sigma, m}^\Lipg\,. 
\end{aligned}
\end{equation}
\begin{lemma}\label{lemma moltiplicazione vphi x}
Let $\T_\s\times\T^\infty_\s \to \C$, $(x,\f) \mapsto a(x, \f)$ be in ${\mathcal H}(\T_{\s+\rho}\times\T^\infty_{\s+\rho})$. Then the multplication operator ${\mathcal M}_a : u \mapsto a u$ satisfies $|{\mathcal M}_a|_{\sigma} \lesssim \rho^{- 1} \| a \|_{\sigma + \rho}$. If $a(x, \f; \omega)$, $\omega \in \Oo \subseteq \Ro$ is Lipschitz w.r. to $\omega$, then $|{\mathcal M}_a|_{\sigma}^\Lipg \lesssim \rho^{- 1} \| a \|_{\sigma + \rho}^\Lipg$.  
\end{lemma}
\begin{proof}
We write 
$$
a(\f, \cdot) = \sum_{\ell \in \Z^\infty_*} \widehat a(\ell, \cdot) e^{\ii \ell \cdot \f}
$$
and consequently
$$
{\mathcal M}_a (\f) = \sum_{\ell \in \Z^\infty_*} \widehat{\mathcal M}_a(\ell) e^{\ii \ell \cdot \f} \quad \text{where} \quad \widehat{\mathcal M}_a(\ell) := {\mathcal M}_{\widehat a(\ell, \cdot)}\,. 
$$
Therefore 
$$
\begin{aligned}
|{\mathcal M}_a|_\sigma & = \sum_{\ell \in \Z^\infty_*} e^{\sigma |\ell|_\zia} \| \widehat{\mathcal M}_a(\ell)\|_{{\mathcal B}^\sigma} \stackrel{Lemma\,\ref{stima moltiplicazione}}{\lesssim} \rho^{- 1} \sum_{\ell \in \Z^\infty_*} e^{\sigma |\ell|_\zia} \| \widehat a(\ell, \cdot)\|_{\sigma + \rho} \lesssim \rho^{- 1} \| a \|_{\sigma + \rho}\,. 
\end{aligned}
$$
Given $\omega_1, \omega_2 \in \Ro$, arguing as above, one can estimate $\Delta_{\omega_1 \omega_2}{\mathcal M}_a = {\mathcal M}_{\Delta_{\omega_1 \omega_2} a}$ in terms of $\Delta_{\omega_1 \omega_2} a$, therefore the Lipschitz estimate follows. 
\end{proof}
Let $m \in \Z$. We recall that the operator $\partial_x^m$ is defined by setting 
\begin{equation}\label{definizione partial x m}
\partial_x^m [1] = 0, \quad \partial_x^m[e^{\im j x}] = \ii^m j^m e^{\ii j x} \quad j \neq 0\,. 
\end{equation}
\begin{lemma}\label{lemma norma an simboli omogenei} Let $\sigma, \rho > 0$, $m, m'  \in \Z$, $a \in {\mathcal H}(\T_{\s+\rho}\times\T^\infty_{\s+\rho})$.
	\\
$(i)$  We have $\partial_x^m a \partial_x^{m'} \in {\mathcal H}(\T^\infty_\sigma, {\mathcal B}^{\sigma, m + m'})$ and $| \partial_x^m a \partial_x^{m'} |_{\sigma, m + m'} \lesssim \rho^{- |m|} \| a \|_{\sigma + \rho}$. If $a(\cdot; \omega)$, $\omega \in \Oo$ is Lipschitz w.r. to $\omega$, then $| \partial_x^m a \partial_x^{m'} |_{\sigma, m + m'}^\Lipg \lesssim \rho^{- |m|} \| a \|_{\sigma + \rho}^\Lipg$. 

\noindent
$(ii)$ For any $N \in \N$
\begin{equation}\label{espansione partial x m m' a}
\partial_x^m a \partial_x^{m'} = \sum_{i = 0}^{N - 1} c_{i, m}(\partial_x^i a) \partial_x^{m + m' - i} + {\mathcal R}_N(a)
\end{equation}
where the remainder ${\mathcal R}_N(a)$ satisfies the estimate 
\begin{equation}\label{stima resti simboli omogenei astratto}
| {\mathcal R}_N(a)|_{\sigma, m + m' - N} \lesssim_{m, N} \rho^{- (2 N + |m| + 1)} \| a \|_{\sigma + \rho} \,. 
\end{equation}
Moreover, one has $c_{0, m} = 1$, $c_{1, m} = m$. If $a(\cdot; \omega)$, $\omega \in \Oo$ is Lipschitz w.r. to $\omega$, then 
\begin{equation}\label{stima resti simboli omogenei astratto Lipg}
| {\mathcal R}_N(a)|_{\sigma, m + m' - N}^\Lipg \lesssim_{m, N} \rho^{- (2 N + |m| + 1)} \| a \|_{\sigma + \rho}^\Lipg \,. 
\end{equation}

\noindent
$(iii)$ Let $ b(\cdot; \omega) \in {\mathcal H}(\T_{\s+\rho}\times\T^\infty_{\s+\rho})$, $\omega \in \Oo$ and set ${\mathcal A} =  a \partial_x^{m}$, ${\mathcal B} := b \partial_x^{m'}$. Then ${\mathcal A} {\mathcal B} \in {\mathcal H}(\T^\infty_\sigma,{\mathcal B}^{\sigma , m + m'})$ satisfies, for any $N \geq 1$,  the expansion 
\begin{equation}\label{espansione simboli cal A cal B}
{\mathcal A} {\mathcal B} = a b \partial_x^{m + m'  } + m a b_x \partial_x^{m + m' - 1}+ \sum_{i = 2}^{N - 1} c_{i, m} a (\partial_x^i b) \partial_x^{m + m'  - i} + {\mathcal R}_N(a, b)\,,
\end{equation}
where $c_{m, i} \in \R$ for any $i = 2, \ldots, N - 1$, the remainder ${\mathcal R}_N(a, b) $ satisfies the estimate
\begin{equation}\label{stime espansione simboli omogenei}
\begin{aligned}
| {\mathcal R}_N(a, b) |_{\sigma, m + m' - N}^\Lipg & \lesssim_{m, m', N} \rho^{- \kappa} \| a \|_{\sigma + \rho}^\Lipg \| b \|_{\sigma + \rho}^\Lipg
\end{aligned}
\end{equation}
for some constant $\kappa = \kappa(m, m', N) > 0$. As a consequence for any $N \geq 1$, the commutator $[{\mathcal A}, {\mathcal B}]$, admits the expansion 
$$
[{\mathcal A}, {\mathcal B}] = (m a b_x - m' a_x b) \partial_x^{m + m' - 1} + \sum_{i = 2}^{N - 1} \big( c_{m , i} a (\partial_x^i b) - c_{m', i} (\partial_x^i a) b \big) \partial_x^{m + m' - i} + {\mathcal R}_N(a, b) - {\mathcal R}_N(b, a)\,. 
$$

\end{lemma}
\begin{proof}
{\sc Proof of $(i)$.} It follows by Lemmata \ref{proprieta norma sigma riducibilita vphi x}, \ref{lemma moltiplicazione vphi x} and using that for any $p \in \Z$, $\sigma > 0$, $|\partial_x^p|_{\sigma, p} = | \partial_x^p |_{\sigma, p}^\Lipg \leq 1$. 

\noindent
{\sc Proof of $(ii)$.} Let ${\mathcal R} :=\partial_x^m a \partial_x^{m'}$. Then ${\mathcal R}(\f) = \sum_{\ell \in \Z^\infty_*} \widehat{\mathcal R}(\ell) e^{\ii \ell \cdot \f}$, where for any $\ell \in \Z^\infty_*$, the operator $\widehat{\mathcal R}(\ell)$ admits the matrix representation $(\widehat{\mathcal R}_j^{j'}(\ell))_{j, j' \in \Z}$ 
\begin{equation}\label{def xm a xm'}
\widehat{\mathcal R}_{j }^{j'}(\ell) =  \ii^{m + m'} j^m \widehat a(\ell, j - j') j'^{m'}, \quad \forall j, j' \in \Z \setminus \{ 0 \}\,. 
\end{equation}
We write the Taylor expansion 
\begin{equation}\label{expansion jm}
j^m = j'^{m} + m j'^{m - 1} (j - j') + \sum_{k  = 2}^{N - 1} c_{m, k} j'^{m - k} (j - j')^k + r_N(j, j')
\end{equation}
where the remainder $r_N(j, j')$ is given by 
\begin{equation}\label{resto r_N j j'}
r_N(j, j') := c_{N, m} \int_0^1 (1 - \tau)^{N - 1} (j' + \tau(j - j'))^{m - N}\, d \tau (j - j')^N\,. 
\end{equation}
By using the Petree inequality, one has that 
$$
\frac{(j' + \tau(j - j'))^{m - N}}{j'^{m - N}} \lesssim_{m, N} \langle  j - j' \rangle^{N + |m|}\,.
$$
This latter inequality, implies that 
\begin{equation}\label{stima r N j j'}
|r_N(j, j')| \lesssim_{m, N} \langle j' \rangle^{m - N} \langle j - j' \rangle^{2 N + |m|}\,. 
\end{equation}
By the definition \eqref{def xm a xm'} and using the expansion \eqref{def xm a xm'}, we get the the operator ${\mathcal R}$ can be expanded as 
$$
{\mathcal R}(\f ) = a \partial_x^{m + m'} + m (\partial_x a) \partial_x^{m + m' - 1} + \sum_{i = 2}^{N - 1} c_{m, i} (\partial_x^i a) \partial_x^{m + m' - i} + {\mathcal R}_N(\f)
$$
where the operator ${\mathcal R}_N(\f) = \sum_{\ell \in \Z^\infty_*} \widehat{\mathcal R}_N(\ell) e^{\ii \ell \cdot \f}$ and for any $\ell \in \Z^\infty_*$, the operator $\widehat{\mathcal R}_N(\ell)$ admits the matrix representation 
\begin{equation}\label{cal R N rapp mat}
(\widehat{\mathcal R}_N(\ell))_j^{j'} := \ii^{m + m'}  \widehat a(\ell, j - j') r_N(j, j') j'^{m'}, \quad j, j' \in \Z \setminus \{ 0 \}\,. 
\end{equation}
By \eqref{stima r N j j'}, using that $\widehat a(\ell, \cdot) \in {\mathcal H}(\T_{\sigma + \rho})$, one gets the estimate 
\begin{equation}\label{stima coeff Fourier RN}
|\widehat{\mathcal R}_j^{j'}(\ell)| \lesssim \langle j - j' \rangle^{2 N + |m|}e^{- (\sigma + \rho)|j - j'|} \langle j' \rangle^{m + m' - N}  \| \widehat a(\ell, \cdot) \|_{\sigma + \rho}\,. 
\end{equation}
Furthermore, using that 
$$
\langle j - j' \rangle^{2 N + |m|} e^{- \frac{\rho}{2}|j - j'|} \lesssim_{N, m} \rho^{- (2 N + |m|)}, 
$$
one gets the estimate 
\begin{equation}\label{stima coeff Fourier RN 2}
|\widehat{\mathcal R}_j^{j'}(\ell)| \lesssim \rho^{- (2 N + |m|)}e^{- (\sigma + \frac{\rho}{2})|j - j'|} \langle j' \rangle^{m + m' - N}  \| \widehat a(\ell, \cdot) \|_{\sigma + \rho}\,. 
\end{equation}
Now if $j, j' \in \N_0$, using the for any $\delta > 0$, $e^{- \delta |j + j'|} \leq e^{- \delta |j - j'|}$, the latter estimate implies also the estimate on the $2 \times 2$ block $\Pi_j \widehat{\mathcal R}_N(\ell) \Pi_{j'}$ of the form 
\begin{equation}\label{stima coeff Fourier RN 3}
\| \Pi_j \widehat{\mathcal R}_N(\ell) \Pi_{j'} \| \lesssim_{m, N} \rho^{- (2 N + |m|)} e^{- (\sigma + \frac{\rho}{2})|j - j'|} \langle j' \rangle^{m + m' - N}  \| \widehat a(\ell, \cdot) \|_{\sigma + \rho}, \quad \forall j, j' \in \N_0\,. 
\end{equation}
Then for any $j' \in \N_0$, one has that 
$$
\begin{aligned}
\sum_{j \in \N_0} e^{\sigma|j - j'|} \| \Pi_j \widehat{\mathcal R}_N(\ell) \Pi_{j'}\| \langle j' \rangle^{N - (m + m')} & \lesssim_{m, N} \rho^{- (2 N + |m|)} \| \widehat a(\ell, \cdot) \|_{\sigma + \rho} \sum_{j \in \N_0} e^{- \frac{\rho}{2}|j - j'|} \\
& \lesssim_{m, N} \rho^{- (2 N + |m| + 1)} \| \widehat a(\ell, \cdot) \|_{\sigma + \rho}  
\end{aligned}
$$
which implies that 
$$
\| \widehat{\mathcal R}_N(\ell)\|_{{\mathcal B}^{\sigma, m + m' - N}} \lesssim_{m, N} \rho^{- (2 N + |m| + 1)} \| \widehat a(\ell, \cdot) \|_{\sigma + \rho}\,. 
$$
By using this latter estimate one gets that 
$$
|{\mathcal R}_N|_{\sigma, m + m' - N} \lesssim_{m, N} \rho^{- (2 N + |m| + 1)} \sum_{\ell \in \Z^\infty_*} e^{\sigma |\ell|_\zia} \| \widehat a(\ell, \cdot)\|_{\sigma + \rho} \lesssim_{m, N} \rho^{- (2 N + |m| + 1)} \| a \|_{\sigma + \rho}
$$
which is exactly the claimed estimate \eqref{stima resti simboli omogenei astratto}. 

\noindent
If $a$ depends on the parameter $\omega \in \Oo\subseteq \Ro$, given $\omega_1, \omega_2 \in \Oo$, one expands the operator $\partial_x^m (\Delta_{\omega_1 \omega_2} a) \partial_x^{m'}$ as in \eqref{espansione partial x m m' a} where $a$ is replaced by $\Delta_{\omega_1 \omega_2} a$ and the remainder ${\mathcal R}_N(\Delta_{\omega_1 \omega_2} a)$ is estimated in term of $\Delta_{\omega_1 \omega_2} a$. The Lipschitz estimate then follows. 

\noindent
{\sc Proof of $(iii)$.} The claimed expansion \eqref{espansione simboli cal A cal B} follows by a repeated application of the item $(i)$.  The estimates of the remainder ${\mathcal R}_N(a, b)$ follows by using the estimates of the items $(i)$ and $(ii)$ and by using the composition Lemma \ref{proprieta norma sigma riducibilita vphi x}. The expansion of the commutator follows easily by expanding ${\mathcal A} {\mathcal B}$ and ${\mathcal B}{\mathcal A}$. 
\end{proof}
\begin{lemma}[Exponential map]\label{lemma mappa esponenziale}
Let $\sigma > 0$, $\rho \in (0, 1)$, $m \geq 0$ and ${\mathcal R}(\omega) \in {\mathcal H}(\T^\infty_{\sigma + \rho}, {\mathcal B}^{{\sigma + \rho}, - m})$, $\omega \in \Omega \subseteq \mathtt R_0$ and assume that 
\begin{equation}\label{smallness lemma exp}
\rho^{- 2} |{\mathcal R}|_{\sigma + \rho}^\Lipg \leq \delta 
\end{equation}
for some $\delta \in (0, 1)$ small enough. Then, for any $N \geq 1$, the map $\Phi_N := {\rm exp}({\mathcal R})- \sum_{n = 0}^{N - 1} \frac{{\mathcal R}^n}{n!} \in {\mathcal H}(\T^\infty_\sigma, {\mathcal B}^{\sigma, - N m})$ with 
\begin{equation}\label{smoothing estimate exponential map}
\begin{aligned}
& | \Phi_N |_{\sigma , - N m}^\Lipg \lesssim  \Big( C_0 \rho^{- (|m| + 2)} |{\mathcal R}|_{\sigma + \rho, - m} \Big)^{N } 
\end{aligned}
\end{equation}
As a consequence ${\rm exp}({\mathcal R}) \in {\mathcal H}(\T^\infty_\sigma, {\mathcal B}^\sigma)$ and 
\begin{equation}\label{estimate exponential map}
|{\rm exp}({\mathcal R})|_\sigma^\Lipg \leq 1 + C \rho^{- (|m| + 2)} |{\mathcal R}|_{\sigma + \rho, - m}^\Lipg
\end{equation}
for some constant $C > 0$. 
\end{lemma}
\begin{proof}
In order to simplify notations for any $n \in \R$, we write $|\cdot |_{\sigma, n}$ instaed of $|\cdot |_{\sigma, n}^\Lipg$. Let $\Phi := {\rm exp}({\mathcal R})$. Then $\Phi - {\rm Id} = \sum_{n \geq 1} \frac{{\mathcal R}^n}{n!}$. By \eqref{prop monotonia norma vphi x}, one has that since ${\mathcal R} \in {\mathcal H}(\T^\infty_\sigma, {\mathcal B}^{\sigma, - m})$, then ${\mathcal R} \in {\mathcal H}(\T^\infty_\sigma, {\mathcal B}^\sigma)$ and $|{\mathcal R}|_{\sigma} \leq |{\mathcal R}|_{\sigma, - m}\,. $
By using the estimate \eqref{stima comp iterata norma dec}, one obtains that for any integer $n \geq 1$, ${\mathcal R}^n \in {\mathcal H}(\T^\infty_\sigma, {\mathcal B}^{\sigma})$ and 
\begin{equation}\label{stima potenza nel lemma exp}
|{\mathcal R}^n|_{\sigma} \leq \Big( C_0 \rho^{- 2} |{\mathcal R}|_{\sigma + \rho} \Big)^{n - 1} |{\mathcal R}|_\sigma 
\end{equation}
for some constant $C_0 > 0$. Now, we write 
$$
\Phi_N = \sum_{n \geq N} \frac{{\mathcal R}^n}{n !} =  \sum_{k \geq 0} \frac{{\mathcal R}^{k}}{(k + N) !}{\mathcal R}^N\,.
$$
By using the estimate \eqref{stima potenza nel lemma exp}, one gets that $\sum_{k \geq 0} \frac{{\mathcal R}^{k}}{(k + N) !} \in {\mathcal H}(\T^\infty_\sigma, {\mathcal B}^\sigma)$ and 
\begin{equation}\label{mafia romana 0}
\begin{aligned}
\Big| \sum_{k \geq 0} \frac{{\mathcal R}^{k}}{(k + N) !} \Big|_\sigma  \leq 1 + \sum_{k \geq 1} \frac{1}{k!} \Big( C_0 \rho^{- 2} |{\mathcal R}|_{\sigma + \rho} \Big)^{k - 1} |{\mathcal R}|_\sigma  \stackrel{\eqref{smallness lemma exp}}{\leq} C_1
\end{aligned}
\end{equation}
for some constant $C_1 > 0$. 
By applying Lemma \ref{proprieta norma sigma riducibilita vphi x}, one has that ${\mathcal R}^N \in {\mathcal H}(\T^\infty_\sigma, {\mathcal B}^{\sigma, - Nm})$ and $\Phi_N = \sum_{k \geq 0} \frac{{\mathcal R}^{k}}{(k + N) !} {\mathcal R}^N \in {\mathcal H}(\T^\infty_\sigma, {\mathcal B}^{\sigma, - Nm})$ and using also the estimate \eqref{mafia romana 0}, one obtains that 
\begin{equation}\label{mafia romana 2}
|\Phi_N|_{\sigma, - N m} \lesssim \rho^{- 2}  |{\mathcal R}^N|_{\sigma + \frac{\rho}{2}, - N m}\,.
\end{equation}
The claimed estimate \eqref{smoothing estimate exponential map} then follows by applying \eqref{stima comp iterata norma dec}. 
 The estimate \eqref{estimate exponential map} follows by triangular inequality and by applying the estimate \eqref{smoothing estimate exponential map} for $N = 1$. 
\end{proof}
\section{Normal form}\label{sezione riduzione ordine}
As we said in the introduction we want to conjugate to constant coefficients the Sch\"odinger equation $\partial_t u = {\mathcal L}(\omega t) u$ where 
$$
{\mathcal L}(\f) :=  \ii (1 + \e \cV_2(x, \f)) \partial_{xx} + \e \ii \cV_1(x, \f) \partial_x + \e \ii \cV_0(x, \f)\,.
$$
We assume that the functions $\cV_0, \cV_1, \cV_2 \in {\mathcal H}(\T^\infty_{\bar \sigma} \times \T_{\bar \sigma})$, for some $\bar \sigma > 0$  satisfy the condition \eqref{condizioni V0 V1 V2}, so that ${\mathcal L}(\f)$ is an $L^2$ skew selfadjoint linear operator.
\subsection{Normalization of the $x$-dependence of the highest order term}
We consider an operator induced by an analytic diffeomorphism of the torus 
$$
(x, \f) \mapsto (x+ \beta(x, \f),\f)
$$
where $\beta$ is a real on real analytic function on the infinite dimensional torus that will be determined later. We make tha ansatz that 
\begin{equation}\label{ansatz beta}
\beta \in {\mathcal H}(\T_{\s_1}\times\T^\infty_{\s_1}), \quad \| \beta \|_\sigma \lesssim_{\sigma_1, \bar \sigma} \delta, \quad \forall 0 < \sigma_1 < \bar \sigma\,. 
\end{equation}
By Proposition \ref{lemma diffeo inverso}, for any $0 < \sigma_1 < \bar \sigma$ there exists $\delta_0(\sigma_1, \bar \sigma) $ such that for any $\delta \leq \delta_0$, the map $(x, \f) \mapsto ( x + \beta(x, \f),\f)$ is invertible, with inverse given by $(y, \f) \mapsto (y+ \widetilde \beta(y, \f),\f)$ and 
\begin{equation}\label{ansatz beta tilde}
\widetilde \beta \in {\mathcal H}(\T_{\s_2}\times\T^\infty_{\s_2}), \quad \| \widetilde \beta \|_{\s_2} \lesssim_{\sigma_1, \sigma_2 } \| \beta\|_{\sigma_1}, \quad \forall \sigma_2 < \sigma_1 < \bar \sigma\,. 
\end{equation} We now define the operator 
\begin{equation}\label{prima trasformazione}
\Phi^{(1)}(\f)[u] := \sqrt{1 + \beta_x(x, \f)} u(x + \beta(x, \f))\,. 
\end{equation}
A direct calculation shows that this map is unitary and, if $\beta$ is appropriately small, invertible  with inverse  given by
\begin{equation}\label{inverso prima trasformazione}
\Phi^{(1)}(\f)^{- 1}[u] := \sqrt{1 + \widetilde \beta_y(y, \f)} u(y + \widetilde \beta(y, \f))\,.
\end{equation}
for $\f\in \T^\infty_\s$ with $\s<\s_2$.
Note that one has the relation 
\begin{equation}\label{formula beta beta tilde}
 1 + \widetilde \beta_y(y, \f) = \frac{1}{1 + \beta_x(y + \widetilde \beta(y, \f),\f)}, \quad   1 + \beta_x(x, \f) = \frac{1}{1 + \widetilde \beta_y(x +  \beta(x, \f),\f)}\,. 
\end{equation}
The following lemma holds. 
\begin{lemma}\label{lemma stime step 1}
For any $\sigma < \s'<\bar \sigma$, there exists $\delta \equiv \delta (\sigma, \sigma', \overline \sigma) \in (0, 1)$ such that if  $\e \in (0, \delta)$  the following holds. Define
\begin{equation}\label{soluzione equazione omologica grado 1}
\begin{aligned}
& m_2 (\f) := \Big( \frac{1}{2 \pi} \int_\T \frac{d x}{\sqrt{1 + \e \cV_2(x, \f)}}\, dx  \Big)^{- 2}\\
& \beta(x, \f) := \partial_x^{- 1}\Big[ \frac{\sqrt{m_2(\f)}}{\sqrt{1 + \e \cV_2(x, \f)}} - 1 \Big] \,. 
\end{aligned}
\end{equation}
\noindent
$(i)$  the map $\T^\infty_\sigma \to {\mathcal B}({\mathcal H}\big(\T_{\sigma'}), {\mathcal H}(\T_{\sigma}) \big)$, $\f \mapsto \Phi^{(1)}(\f)^{\pm 1}$ is bounded. 

\noindent
$(ii)$ For any $s \geq 0$, the map $\T^\infty \to {\mathcal B}(H^s(\T), H^s(\T) \big)$, $\f \mapsto \Phi^{(1)}(\f)^{\pm 1}$ is bounded.

\noindent
$(iii)$  $\Phi^{(1)}(\f)$ transforms the operator ${\mathcal L}(\f)$ into 
\begin{equation}\label{forma finale cal L (1)}
{\mathcal L}^{(1)}(\f) := (\Phi^{(1)}_{\omega*}){\mathcal L}(\f)= \ii m_2(\f) \partial_x^2 + a_1(x, \f) \partial_x + a_0(x, \f)\,. 
\end{equation}
where the functions $m_2 \in {\mathcal H}(\T^\infty_\sigma), \beta, \widetilde \beta, a_1, a_0 \in {\mathcal H}(\T_\s\times\T^\infty_\s)$ are independent of the parameter $\omega$ and  satisfy the estimates
\begin{equation}\label{stima m2 beta tilde beta}
\| m_2 - 1 \|_{\sigma}, \| \beta \|_{\sigma}, \| \widetilde \beta\|_{\sigma}\,, \| a_1\|_{\sigma}, \| a_0\|_{\sigma} \lesssim_{\sigma, \bar \sigma} \e\,. 
\end{equation}
Finally ${\mathcal L}^{(1)}$ is kew self-adjoint, hence $m_2(\f), a_1(x,\f)$ are real on real while 
$
 a_0 = - \overline{a}_0+ {\partial_x a_1}
$.
\end{lemma}
\begin{proof}
	The proof of the item $(i)$ follows by the definitions \eqref{prima trasformazione}, \eqref{inverso prima trasformazione}, by using the estimates on $\beta, \widetilde \beta$ \eqref{stima m2 beta tilde beta} and by applying Lemmata \ref{Lemma prodotto}, \ref{composizione funzioni analitiche T infty}. 
	\\
	To prove the item $(ii)$ we argue as follows. Since $\beta$ and $\widetilde \beta$ are analytic, then for any $\f \in \T^\infty$ one has $\beta(\f, \cdot), \widetilde \beta(\f, \cdot) \in {\mathcal C}^\infty(\T)$ and 
	$\sup_{\f \in \T^\infty} \| \beta(\f, \cdot)\|_{{\mathcal C}^s(\T)}\,,\, \sup_{\f \in \T^\infty} \| \widetilde \beta(\f, \cdot)\|_{{\mathcal C}^s(\T)} < \infty$ for any $s \geq 0$. A direct calculation then shows that $\sup_{\f \in \T^\infty}\| \Phi(\f)\|_{{\mathcal B}(H^s(\T), H^s(\T))} \leq C\Big( \sup_{\f \in \T^\infty} \| \beta(\f, \cdot)\|_{{\mathcal C}^s(\T)} \Big)$ and 
	$$
	\sup_{\f \in \T^\infty}\| \Phi(\f)^{- 1}\|_{{\mathcal B}(H^s(\T), H^s(\T))} \leq C\Big( \sup_{\f \in \T^\infty} \| \widetilde \beta(\f, \cdot)\|_{{\mathcal C}^s(\T)} \Big)
	$$ 
		and the result follows. 
		\\ 
In order to prove $(iii)$ we remark that the map $\Phi^{(1)}(\f)$ satisfies the following conjugation rules: 
\begin{equation}\label{regole di coniugazione}
\begin{aligned}
& \Phi^{(1)}(\f)^{- 1} \circ  a(x, \f) \circ \Phi^{(1)}(\f) = a(y+ \widetilde \beta(y, \f),\f)\,, \\
& \Phi^{(1)}(\f)^{- 1} \circ \partial_x \circ \Phi^{(1)}(\f) = (1 + \beta_x(y + \widetilde \beta(y, \f),\f) ) \partial_y + \frac12(1 + \widetilde \beta_y(y, \f)) \beta_{xx}(y+ \widetilde \beta(y, \f),\f)  \,, \\
& \Phi^{(1)}(\f)^{- 1} \omega \cdot \partial_\f \Phi^{(1)}(\f) = \omega \cdot \partial_\f \beta(y+ \widetilde \beta(y, \f),\f) \partial_y + \frac12 (1 + \widetilde \beta_y(y, \f))\omega \cdot \partial_\f \beta_x(y+ \widetilde \beta(y, \f),\f)\,.
\end{aligned}
\end{equation}
Then, recalling \eqref{push forward}, the transformed operator is
\begin{equation}\label{cal L (1)}
\begin{aligned}
 {\mathcal L}^{(1)}(\f) & = \ii a_2 (y, \f) \partial_y^2 + a_1(y, \f) \partial_y + a_0 (y, \f)\,, \\
 a_2 & := \Big( (1 + \e \cV_2) (1 + \beta_x)^2 \Big)_{x = y + \widetilde \beta(y, \f)}\,, \\
 a_1 & := \Big( 2\ii (1 + \e \cV_2) \beta_{xx} + \e \ii \cV_1 (1 + \beta_x) - \omega \cdot \partial_\f \beta\Big)_{x = y + \widetilde \beta(y, \f)} \,, \\
 a_0 & := \ii \sqrt{1 + \widetilde \beta_y}\Big( (1 + \e \cV_2) \partial_{xx} \sqrt{1 + \beta_x} \Big)\Big|_{x = y + \widetilde \beta(y, \f)}  \\
 & \qquad + \frac12  \ii (1 + \widetilde \beta_y) \Big( \e \cV_1 \beta_{xx} + \omega \cdot \partial_\f \beta_x  \Big)\Big|_{x = y + \widetilde \beta(y, \f)} + \e \cV_0(y, \f + \widetilde \beta(y, \f))\,. 
\end{aligned}
\end{equation}
By the definitions of the functions $\beta(x, \f)$ and $m_2(\f)$ given in \eqref{soluzione equazione omologica grado 1} one gets  
\begin{equation}\label{omologica grado max}
a_2(x, \f) = m_2(\f), \quad \text{namely} \quad (1 + \e \cV_2) (1 + \beta_x)^2= m_2(\f) 
\end{equation}
hence the operator ${\mathcal L}^{(1)}(\f)$ in \eqref{cal L (1)} takes the form \eqref{forma finale cal L (1)}. Since $\Phi^{(1)}$ is unitary, by construction ${\mathcal L}^{(1)}$ is skew self-adjoint.

Since $\cV_2 \in {\mathcal H}^{\bar \sigma}_{x, \f}$, by applying Lemma \ref{moser type lemma}, (applied to the analytic function  $f(u) = \frac{1}{\sqrt{1 + u}}$, $|u| \leq \frac12$) and by the definition \eqref{soluzione equazione omologica grado 1}, one gets that for $\e$ small enough, $\beta \in {\mathcal H}(\T_{\s_1}\times\T^\infty_{\s_1})$, $m_2 \in {\mathcal H}(\T^\infty_{\s_1})$ for any $0 < \sigma_1 < \bar \sigma$. Using tha mean value theorem, one gets the estimate, $\| \beta \|_{\sigma_1}, \| m_2 - 1\|_{\sigma_1} \lesssim_{\sigma_1, \bar \sigma} \e$. The ansatz \eqref{ansatz beta} is then proved. The ansatz \eqref{ansatz beta tilde}, follows by Proposition \ref{lemma diffeo inverso}. Finally, by applying Lemmata \ref{moser type lemma}, \ref{composizione funzioni analitiche T infty}, \ref{stima di cauchy}, and using that $\cV_2, \cV_1, \cV_0 \in {\mathcal H}(\T_{\overline \sigma} \times \T^{\infty}_{\overline \sigma})$, one deduces the claimed properties on the functions $a_0$ and $a_1$. 
\end{proof}
\subsection{Reduction to constant coefficients of the highest order term}\label{sezione riparametrizzazione tempo}
Our next purpose is to eliminate the $\f$-dependence from the highest order coefficient $m_2(\f) \partial_{xx}$ of the operator ${\mathcal L}^{(1)}(\f)$ in \eqref{forma finale cal L (1)}. To achieve this we conjugate the equation $\partial_t u = \ii {\mathcal L}^{(1)}(\omega t) u$ by means of a reparametrization of time $t \mapsto t +  \alpha(\omega t)$, where $\alpha$ is a suitable analytic function which has to be determined. More precisely we consider the change of varialbles 
\begin{equation}\label{def Phi (2)}
u(t, x) = \Phi^{(2)} v(t, x) := v(x,t + \alpha(\omega t)), \quad (x, t) \in \T \times \R
\end{equation} 
We assume that $\alpha(\f)$ is real on real and  satisfies the ansatz  
\begin{equation}\label{ansatz alpha riparametrizzazione}
\alpha \in {\mathcal H}(\T^\infty_{\sigma_1}), \quad \| \alpha\|_{\sigma_1} \lesssim_{\sigma_1, \bar \sigma} \delta, \quad \forall 0 < \sigma_1 < \bar \sigma\,. 
\end{equation}
By applying Proposition \ref{lemma diffeo inverso}, for any $\sigma_2 < \bar \sigma$ there exists $\delta_0 = \delta_0(\sigma_2,\s_1, \bar \sigma)$ small enough such that if $\delta \leq \delta_0$, the map $\f \mapsto \f + \omega \alpha(\f)$ is invertible with inverse given by $\vartheta \mapsto \vartheta + \omega \widetilde \alpha(\vartheta)$ and 
\begin{equation}\label{proprieta tilde alpha}
\widetilde \alpha \in {\mathcal H}(\T^\infty_{\sigma_2}), \quad \| \widetilde \alpha\|_{\sigma_2} \lesssim_{\sigma_1, \sigma_2} \| \alpha\|_{\sigma_1}, \quad \forall \sigma_2 < \sigma_1 < \bar \sigma\,. 
\end{equation}
The inverse of the map $\Phi^{(2)}$ in \eqref{def Phi (2)} is then given by 
\begin{equation}\label{def Phi (2) inv}
(\Phi^{(2)})^{- 1} u ( x, \tau) := u(x,\tau +  \widetilde \alpha(\omega \tau))\,.
\end{equation}
%and the equation for $v(x, \tau)$ is given by 
%$$
%\partial_\tau v= {\cL}^{(2)}(\omega \tau) v, \quad {\cL}^{(2)}(\vartheta) := 
%$$

%******************
% passing to the $\T_\s^\infty$ we conjugate ${\mathcal L}^{(1)}(\f)$, through formula \eqref{push forward}, with the unitary operator
%\begin{equation}\label{def Phi (2)}
%\Phi^{(2)} u := u(x,\f + \alpha(\f)\omega)\,.
%\end{equation} 
%We assume that $\alpha(\f)$ is real on real and  satisfies the ansatz  
%\begin{equation}\label{ansatz alpha riparametrizzazione}
%\alpha \in {\mathcal H}(\T^\infty_{\sigma_1}), \quad \| \alpha\|_{\sigma_1} \lesssim_{\sigma_1, \bar \sigma} \delta, \quad \forall 0 < \sigma_1 < \bar \sigma\,. 
%\end{equation}
%By applying Proposition \ref{lemma diffeo inverso}, for any $\sigma_2 < \bar \sigma$ there exists $\delta_0 = \delta_0(\sigma_2,\s_1, \bar \sigma)$ small enough such that if $\delta \leq \delta_0$, the map $\f \mapsto \f + \omega \alpha(\f)$ is invertible with inverse given by $\vartheta \mapsto \vartheta + \omega \widetilde \alpha(\vartheta)$ and 
%\begin{equation}\label{proprieta tilde alpha}
%\widetilde \alpha \in {\mathcal H}(\T^\infty_{\sigma_2}), \quad \| \widetilde \alpha\|_{\sigma_2} \lesssim_{\sigma_1, \sigma_2} \| \alpha\|_{\sigma_1}, \quad \forall \sigma_2 < \sigma_1 < \bar \sigma\,. 
%\end{equation}
%The inverse of the map $\Phi^{(2)}$ in \eqref{def Phi (2)} is then given by 
%\begin{equation}\label{def Phi (2) inv}
%(\Phi^{(2)})^{- 1} u ( x,\f) := u(x,\vartheta + \omega \widetilde \alpha(\vartheta))\,.
%\end{equation}
\begin{remark}\label{remark riparametrizzazione tempo}
If $u(x)$ is a function independent of the $\f$, then $(\Phi^{(2)})^{\pm 1} u = u$. 
\end{remark}

The following lemma holds. 

\begin{lemma}\label{lemma stime step 2}
Let $\omega \in \Dc$. For any $\sigma < \bar \sigma$ there exists $\delta(\sigma, \bar \sigma) > 0$ such that if $\e \gamma^{- 1} \leq \delta$, then, setting
\begin{equation}\label{definizione alpha lambda 2}
\begin{aligned}
\lambda_2 := \widehat m_2(0) = \int_{\T^\infty} m_2(\f)\, d \f\,, \quad \alpha  := (\omega \cdot \partial_\f)^{- 1}\Big[ \frac{m_2}{\lambda_2} - 1 \Big]\,,
\end{aligned}
\end{equation}
then $\Phi^{(2)}$ transforms the operator ${\mathcal L}^{(1)}(\f)$ in 
\begin{equation}\label{terza espressione cal L (2)}
{\mathcal L}^{(2)}(\vartheta) = \ii \lambda_2 \partial_x^2 + b_1(\vartheta, x) \partial_x + b_0(\vartheta, x) 
\end{equation}
The constant $\lambda_2 \in \R$ is independent of $\omega$.  For all $\omega\in\Dc$ the functions $ \alpha(\cdot; \omega), \widetilde \alpha(\cdot; \omega) \in {\mathcal H}(\T^\infty_\sigma ), b_1(\cdot; \omega),\im  b_0(\cdot; \omega) \in {\mathcal H}(\T_\s\times\T^\infty_\s)$ are  well defined and real on real. Furthermore, for any $\Oo\subseteq \Dc$ the following estimates hold:
$$
\begin{aligned}
& |\lambda_2 - 1|, \| b_0\|_{\sigma}^\Lipg, \| b_1\|_{\sigma}^\Lipg \lesssim \e\,, \quad  \| \alpha\|_{\sigma}^\Lipg\,,\, \| \widetilde \alpha\|_{\sigma}^\Lipg \lesssim \e \gamma^{- 1}\,. 
\end{aligned}
$$
\end{lemma}
\begin{proof}
A direct calculation shows that formula \eqref{push forward} reads
\begin{equation}\label{prima definizione cal L (2)}
{\mathcal L}^{(2)}(\vartheta) := (\Phi^{(2)}_{\omega*}){\mathcal L^{(1)}}(\f)=\frac{1}{\rho(\vartheta)} {\mathcal L}^{(1)}(\vartheta + \omega \widetilde \alpha(\vartheta))\,, \quad \rho(\vartheta) := 1 + \omega \cdot \partial_\f \alpha(\vartheta + \omega \widetilde \alpha(\vartheta))\,. 
\end{equation}
Note that, since ${\mathcal L}^{(1)}(\omega t)$ is skew self-adjoint then also ${\mathcal L}^{(2)}(\omega t)$ is skew self-adjoint. By \eqref{prima definizione cal L (2)}, one has 
\begin{equation}\label{seconda espressione cal L (2)}
\begin{aligned}
{\mathcal L}^{(2)}(\vartheta) & = \ii b_2(\vartheta) \partial_x^2 + b_1(\vartheta, x) \partial_x + b_0(\vartheta, x)\,, \\
b_2(\vartheta) & := \Big[\dfrac{m_2}{1 + \omega \cdot \partial_\f \alpha} \Big]\Big|_{\f = \vartheta + \omega \widetilde \alpha(\vartheta)}\,, \\
b_1(\vartheta, x) & := \Big[\dfrac{a_1}{1 + \omega \cdot \partial_\f \alpha} \Big]\Big|_{\f = \vartheta + \omega \widetilde \alpha(\vartheta)}\,, \\
b_0(\vartheta, x) & := \Big[\dfrac{a_0}{1 + \omega \cdot \partial_\f \alpha} \Big]\Big|_{\f = \vartheta + \omega \widetilde \alpha(\vartheta)}\,. 
\end{aligned}
\end{equation}
By the definitions of $\alpha(\f)$ and $\lambda_2 \in \R$ given in \eqref{definizione alpha lambda 2}, one obtains that  
\begin{equation}\label{equazione omologica riparametrizzazione tempo}
b_2(\vartheta) = \lambda_2, \quad \text{namely} \quad \dfrac{m_2(\f)}{1 + \omega \cdot \partial_\f \alpha(\f)} = \lambda_2 
\end{equation}
 and therefore the linear operator ${\mathcal L}^{(2)}(\f)$ defined in \eqref{seconda espressione cal L (2)} takes the form given in \eqref{terza espressione cal L (2)}. 
 Note that the function $m_2(\f)$ defined in \eqref{soluzione equazione omologica grado 1} is independent of $\omega$ and therefore also $\lambda_2$ does not depend on $\omega$. By applying Lemma \ref{lemma stime step 1}, by the definition \eqref{definizione alpha lambda 2} and by Lemmata  \ref{stima diofantea D omega inv}-$(ii)$, \ref{lemma diffeo inverso}, one gets that $|\lambda_2 - 1| \lesssim \e$ and that for any $0 < \sigma < \bar \sigma$, for $\e \gamma^{- 1} \leq \delta$, for some $\delta = \delta(\sigma, \bar \sigma)$ small enough, $\alpha, \widetilde \alpha \in {\mathcal H}(T_\sigma^\infty)$ with $\| \alpha \|_\sigma^\Lipg\,, \| \widetilde \alpha\|_\sigma^\Lipg \lesssim_{\sigma, \bar \sigma} \e \gamma^{- 1}$. Finally, recalling the definitions \eqref{seconda espressione cal L (2)}, using the properties on $a_0$ and $a_1$ stated in Lemma \ref{lemma stime step 1} and by applying Lemmata \ref{moser type lemma} (with $f(u) = \frac{1}{1 + u}$, $|u| \leq \frac12$), \ref{composizione funzioni analitiche T infty}, \ref{lemma om dot partial vphi}-$(ii)$, we can deduce the claimed properties on $b_0$ and $b_1$.  

\end{proof}
\subsection{Elimination of the $x$-dependence from the first order term}
The next aim is to eliminate the dependence on $x$ from the first order term in \eqref{terza espressione cal L (2)}. To this aim, we conjugate the vector field ${\mathcal L}^{(2)}(\f)$ by means of a multiplication operator 
\begin{equation}\label{definizione Phi (3)}
\Phi^{(3)}(\f) : u \mapsto e^{\ii p(x, \f)} u 
\end{equation}
where $p$ is an analytic real on real function which has to be determined. The following lemma holds. 
\begin{lemma}\label{lemma stime step 3}
Let $\omega \in \Dc$. For any $0 < \sigma < \bar \sigma$ there exists $\delta(\sigma, \bar \sigma) > 0$ such that if $\e \gamma^{- 1} \leq \delta$, the following holds. Define
\begin{equation}\label{soluzione omologica grado 3}
m_1(\f) := \frac{1}{2 \pi} \int_\T b_1(x, \f)\, d x\,, \quad p(x, \f) := \frac{\partial_x^{- 1}[ b_1(x, \f) - m_1(\f)]}{2 \lambda_2}\,. 
\end{equation}
$(i)$  the map $\T^\infty_\sigma \to {\mathcal B}({\mathcal H}\big(\T_{\sigma}), {\mathcal H}(\T_{\sigma}) \big)$, $\f \mapsto \Phi^{(3)}(\f)^{\pm 1}$ is bounded. 
\\
$(ii)$ For any $s \geq 0$, the map $\T^\infty \to {\mathcal B}(H^s(\T), H^s(\T) \big)$, $\f \mapsto \Phi^{(3)}(\f)^{\pm 1}$ is bounded. 
\\
$(iii)$ the operator $\Phi^{(3)}(\f)$ transforms  ${\mathcal L}^{(2)}(\f)$ in
\begin{equation}\label{versione finale cal L (3) nel lemma}
{\mathcal L}^{(3)}(\f) = \ii \lambda_2 \partial_{xx} + m_1(\f) \partial_x + c_0(x, \f)
\end{equation} 
where the functions $ p(\cdot; \omega), \im  c_0(\cdot; \omega) \in {\mathcal H}(\T_\s\times\T^\infty_\s), \,m_1(\cdot; \omega) \in {\mathcal H}(\T^\infty_\sigma)$ are real on real, well defined for $\omega \in \Dc$ and satisfy  for $\Oo\subseteq \Dc$ the estimates
\begin{equation}\label{stima p c0 m1}
\begin{aligned}
&  \| p \|_{\sigma}^\Lipg\,,\, \| c_0\|_{\sigma}^\Lipg, \| m_1\|_{\sigma}^\Lipg \lesssim_{\sigma, \bar \sigma} \e \,. 
\end{aligned}
\end{equation}
\end{lemma}
\begin{proof}
	Item $(i)$ follows by the definition \eqref{definizione Phi (3)}, by Lemmata \ref{Lemma prodotto}, \ref{moser type lemma} and by the estimates \eqref{stima p c0 m1} on $p$, which are a straightforward computation.
		\\
	$(ii)$ Since $p$ is analytic, then $p(\f, \cdot) \in {\mathcal C}^\infty(\T)$ for any $\f \in \T^\infty$ and $M(s) := \sup_{\f \in \T^\infty} \| p(\f, \cdot)\|_{{\mathcal C}^s(\T)} < \infty$ for any $s \geq 0$. 
	A direct calculation shows that 
	
	$\sup_{\f \in \T^\infty} \| \Phi^{(3)}(\f)^{\pm 1} \|_{{\mathcal B}(H^s(\T))} \leq \sup_{\f \in \T^\infty} \| {\rm exp}(\ii p) \|_{{\mathcal C}^s(\T)} \leq {\rm exp}(M(s))$. The latter estimate proves item $(ii)$. 
	\\
$(iii)$ A direct calculation shows that 
\begin{equation}\label{prima def cal L (3)}
\begin{aligned}
{\mathcal L}^{(3)}(\f) & := (\Phi^{(3)}_{\omega*}){\mathcal L}^{(2)}(\f) = \Phi^{(3)}(\f)^{- 1} {\mathcal L}^{(2)}(\f) \Phi^{(3)}(\f) - \Phi^{(3)}(\f)^{- 1} \omega \cdot \partial_\f \Phi^{(3)}(\f) \\
& = \ii \lambda_2 \partial_{xx} + c_1(x, \f) \partial_x + c_0(x, \f)
\end{aligned}
\end{equation}
where 
\begin{equation}\label{definizione c0 c1 c2}
\begin{aligned}
c_0 & :=  - \ii \lambda_2 p_x^2 -  \lambda_2 p_{xx}  + \ii b_1 p_x - \ii \omega \cdot \partial_\f p + b_0\,, \\
c_1 & := - 2  \lambda_2 p_x  + b_1\,.
\end{aligned}
\end{equation}
The definions of $p$ and $m_1$ given in \eqref{soluzione omologica grado 3} allow to solve the equation  
\begin{equation}\label{equazione omologica grado 3}
- 2  \lambda_2 p_x(x, \f)  + b_1(x, \f) = m_1(\f)\,. 
\end{equation}
Therefore, the operator ${\mathcal L}^{(3)}(\f)$ in \eqref{prima def cal L (3)}takes the form \eqref{versione finale cal L (3) nel lemma}.  

Note that the skew self-adjoint structure guarantees that $\ii m_1(\f)$ is a real function (meaning that it is real on real). 
The claimed properties on the functions $p$ and $m_1$ follow by their definitions \eqref{soluzione omologica grado 3} and by applying Lemma \ref{lemma stime step 2}. The claimed properties on the function $c_0$ defined in \eqref{definizione c0 c1 c2} follow by Lemma \ref{lemma stime step 2} and by applying Lemmata \ref{stima di cauchy}-$(ii)$, \ref{lemma om dot partial vphi}-$(ii)$. 
\end{proof}
\subsection{Reduction to constant coefficients of the first order term}
In order to reduce to constant coefficients the first order term in \eqref{versione finale cal L (3) nel lemma}, we consider the transformation
\begin{equation}\label{def trasformazione step 4}
\Phi^{(4)}(\f) : u(x) \mapsto u(x + q(\f))
\end{equation}
where $q$ is an analytic function on $\T^\infty_\sigma$ to be determined. Clearly, the inverse of $\Phi^{(4)}(\f)$ is given by 
$$
\Phi^{(4)}(\f)^{- 1} : u(x) \mapsto u(x - q(\f))\,. 
$$
\begin{lemma}\label{lemma stime step 4}
Let $\omega \in \Dc$. For any $\sigma < \bar \sigma$ there exists $\delta(\sigma, \bar \sigma) > 0$ such that if $\e \gamma^{- 1} \leq \delta$, and define 
\begin{equation}\label{soluzione eq omologica step 4}
\lambda_1 := \int_{\T^\infty} m_1(\f)\, d \f = \widehat m_1(0), \quad q(\f) :=  (\omega \cdot \partial_\f)^{- 1} [m_1(\f) - \lambda_1]\,. 
\end{equation}
$(i)$  the map $\T^\infty_\sigma \to {\mathcal B}({\mathcal H}\big(\T_{\sigma}), {\mathcal H}(\T_{\sigma}) \big)$, $\f \mapsto \Phi^{(4)}(\f)^{\pm 1}$ is bounded. 
\\
$(ii)$ For any $s \geq 0$, the map $\T^\infty \to {\mathcal B}(H^s(\T), H^s(\T) \big)$, $\f \mapsto \Phi^{(4)}(\f)^{\pm 1}$ is bounded. 
\\
$(iii)$ The map $\Phi^{(4)}(\f)$ transforms the operator ${\mathcal L}^{(3)}(\f)$ as 
 \begin{equation}\label{cal L (4) nel lemma finale}
 {\mathcal L}^{(4)}(\f) = \ii \lambda_2 \partial_{xx} + \lambda_1 \partial_x + d_0(x, \f)
 \end{equation}
 where the constant $\lambda_1 \in \R$ does not depend on $\omega$ and  $ q(\cdot; \omega) \in {\mathcal H}(\T^\infty_\sigma),  \im d_0(\cdot; \omega) \in {\mathcal H}(\T_\s\times\T^\infty_\s)$ are  real on real functions defined for $\omega \in \Dc$. Furthermore, the following bounds hold for any $\Oo\subseteq \Dc$
 \begin{equation}\label{stime q d0}
\begin{aligned}
&  \| q \|_{\sigma}^\Lipg\,,\, \| d_0\|_{\sigma}^\Lipg \lesssim_{\sigma, \bar \sigma} \e \,, \quad |\lambda_1| \lesssim \e\,.  
\end{aligned}
\end{equation}
\end{lemma}
\begin{proof}
Items $(i)$-$(ii)$ follow as the corresponding ones of Lemma \ref{lemma stime step 1}, by using the estimate \eqref{stime q d0} on the function $q(\f)$, which is a direct computation.

$(iii)$ A direct calculation shows that 
\begin{equation}\label{prima def cal L (4)}
\begin{aligned}
{\mathcal L}^{(4)}(\f) & := (\Phi^{(4)}_{\omega*} ){\mathcal L}^{(3)}(\f) = \ii \lambda_2 \partial_{xx} + \big( -  \omega \cdot \partial_\f q(\f) + m_1(\f) \big) \partial_x + d_0(x, \f)\,. \\ 
d_0(x, \f) & := c_0(x, \f - q(\f))\,. 
\end{aligned}
\end{equation}
By the definition \eqref{soluzione eq omologica step 4}, we solve the equation 
\begin{equation}\label{eq omologica step 4}
-  \omega \cdot \partial_\f q(\f) + m_1(\f) = \lambda_1\,.
\end{equation}
Then, the operator ${\mathcal L}^{(4)}$ defined in \eqref{prima def cal L (4)} takes the form given in \eqref{cal L (4) nel lemma finale}. We now show that $\lambda_1$ is independent of $\omega$. By \eqref{soluzione omologica grado 3}, \eqref{soluzione eq omologica step 4}, one has that 
$$
\lambda_1 = \frac{1}{2 \pi} \int_{\T^\infty} \int_\T b_1(\vartheta, x)\, d x\, d \vartheta
$$
where by \eqref{seconda espressione cal L (2)} and using the properties \eqref{formule omega alpha alpha tilde}, one has that 
$$
\begin{aligned}
b_1(\vartheta, x) & = \Big[\dfrac{a_1}{1 + \omega \cdot \partial_\f \alpha} \Big]\Big|_{\f = \vartheta + \omega \widetilde \alpha(\vartheta)}  \\
& = a_1(\vartheta + \omega \widetilde \alpha(\vartheta), x) \Big(1 + \omega \cdot \partial_\vartheta \widetilde \alpha(\vartheta) \Big)\,.
\end{aligned} 
$$
By expanding $a_1(x, \f)$ in Fourier series, i.e. $a_1(x, \f) = \sum_{j \in \Z} \sum_{\ell \in \Z^\infty_*} \widehat a_1(\ell, j) e^{\ii \ell \cdot \f} e^{\ii j x}$ one has that 
$$
\begin{aligned}
\lambda_1 & = \frac{1}{2 \pi} \int_{\T^\infty} \int_\T b_1(\vartheta, x)\, d x\, d \vartheta \\
& = \sum_{j \in \Z} \sum_{\ell \in \Z^\infty_*} \widehat a_1(\ell, j)  \int_\T e^{\ii j x}\, d x \int_{\T^\infty} e^{\ii \ell \cdot (\vartheta + \omega \widetilde \alpha(\vartheta))} \Big(1 + \omega \cdot \partial_\vartheta \widetilde \alpha(\vartheta) \Big)\, d \vartheta \\
& = \sum_{\ell \in \Z^\infty_*} \widehat a_1(\ell, 0) \int_{\T^\infty} e^{\ii \ell \cdot (\vartheta + \omega \widetilde \alpha(\vartheta))} \Big(1 + \omega \cdot \partial_\vartheta \widetilde \alpha(\vartheta) \Big)\, d \vartheta \\
& \stackrel{Lemma\, \ref{lemma per media lambda 1}}{=} \widehat a_1(0, 0) = \frac{1}{2 \pi}\int_\T \int_{\T^\infty} a_1(x, \f)\, d\f \, d x
\end{aligned}
$$
By Lemma \ref{lemma stime step 1}, the function $a_1$ does not depend on $\omega$ and therefore also $\lambda_1$ is independent of $\omega$. 

\noindent
The estimates on $\lambda_1, q, d_0$ given in \eqref{prima def cal L (4)}, \eqref{soluzione eq omologica step 4} follow by applying Lemmata \ref{lemma stime step 3}, \ref{composizione funzioni analitiche T infty}, \ref{stima diofantea D omega inv}-$(ii)$.  
\end{proof}

\subsection{Elimination of the $x$-dependence from the zero-th order term}
In order to eliminate the $x$-dependence from the zero-th order term in the operator ${\mathcal L}^{(4)}(\f)$ in \eqref{prima def cal L (4)}, we conjugate using \eqref{push forward}, by  a transformation 
\begin{equation}\label{definizione Phi (5)}
\Phi^{(5)}(\f) := {\rm exp}({\mathcal V}(\f)) \quad \text{where} \quad {\mathcal V}(\f) := \frac{1}{2}\big( v(x, \f) \circ \partial_x^{- 1} + \partial_x^{- 1} \circ v(x, \f) \big).
\end{equation}
where $v(x, \f)$ is a real on real function to be determined. Note that for real values of the angle $\f \in \T^\infty$, one has that ${\mathcal V}(\f) = - {\mathcal V}(\f)^*$, implying that $\Phi^{(5)}(\f)$ is a unitary operator. 
\begin{lemma}\label{lemma stime step 5}
Let $\omega \in \Dc$. For any $0 <\sigma < \bar \sigma$ there exists $\delta(\sigma, \bar \sigma) > 0$ such that if $\e \gamma^{- 1} \leq \delta$, the following holds. Define 
\begin{equation}\label{def v step 5}
v := \frac{1}{2 \ii \lambda_2} \partial_x^{- 1} \Big( \langle d_0 \rangle_x - d_0 \Big)\,.
\end{equation}
$(i)$ the map $\T^\infty_\sigma \to {\mathcal B}({\mathcal H}\big(\T_{\sigma}), {\mathcal H}(\T_{\sigma}) \big)$, $\f \mapsto \Phi^{(5)}(\f)^{\pm 1}$ is bounded. 
\\
$(ii)$ For any $s \geq 0$, the map $\T^\infty \to {\mathcal B}(H^s(\T), H^s(\T) \big)$, $\f \mapsto \Phi^{(5)}(\f)^{\pm 1}$ is bounded. 
\\
$(iii)$ The map $\Phi^{(5)}(\f)$ transforms the operator ${\mathcal L}^{(4)}(\f)$ in 
\begin{equation}\label{cal L (5) definitivo}
{\mathcal L}^{(5)}(\f) : = (\Phi^{(5)}_{\omega*}){\mathcal L}^{(4)}(\f) = \ii \lambda_2 \partial_{xx} + \lambda_1 \partial_x + \langle d_0 \rangle_x(\f) + e_{- 1}(x, \f) \partial_x^{- 1} + {\mathcal R}^{(5)}(\f)
\end{equation}
and the functions $v(\cdot; \omega) \in {\mathcal H}(\T_\s\times\T^\infty_\s)$ and the operator ${\mathcal R}^{(5)}(\omega) \in {\mathcal H}\big(\T^\infty_\sigma, {\mathcal B}^{\sigma, - 2} \big)$ defined for $\omega \in \Dc$
satisfy the estimates
\begin{equation}\label{stima v step 5}
\begin{aligned}
&  \| v \|_{\sigma}^\Lipg\,,\, \| e_{- 1}\|_{\sigma}^\Lipg, |{\mathcal R}^{(5)}|_{\sigma, - 2}^\Lipg \lesssim_{\sigma, \bar \sigma} \e \,.   
\end{aligned}
\end{equation}
\end{lemma}
\begin{proof}
 By the definition \eqref{def v step 5}, using the estimates on $d_0$ given in Lemma \ref{lemma stime step 4}, one gets that $v$ satisfies the estimate \eqref{stima v step 5}. By Lemma \ref{lemma norma an simboli omogenei}, one has that the operator ${\mathcal V}(\f)$ admits an expansion of the form 
\begin{equation}\label{espansione cal A f}
{\mathcal V}(\f) = v(x, \f) \partial_x^{- 1} - \frac12 v_x(x, \f) \partial_x^{- 2} + c_{- 3} v_{xx} \partial_x^{- 3}   + {\mathcal R}_{\mathcal V}(\f)
\end{equation}
where $c_{- 3} \in \R$ is a constant and for any $0 < \sigma < \bar \sigma$, ${\mathcal R}_{{\mathcal V}} \in {\mathcal H}\big(\T^\infty_\sigma, {\mathcal B}^{\sigma, - 4} \big)$ and 
\begin{equation}\label{stima cal R cal A}
|{\mathcal V}|_{\sigma, - 1}^\Lipg\,,\, |{\mathcal R}_{{\mathcal V}}|_{\sigma , - 4}^\Lipg \lesssim_{\sigma, \bar \sigma } \e \,. 
\end{equation} 
	By \eqref{definizione Phi (5)}, \eqref{stima cal R cal A}, Lemma \ref{lemma norma an simboli omogenei}-$(i)$ and the estimate \eqref{estimate exponential map},  there exists $\delta = \delta(\sigma, \overline \sigma) \in (0, 1)$ such that if $\e \gamma^{- 1} \leq \delta$, $|(\Phi^{(5)})^{\pm 1}|_\sigma \lesssim_{\sigma, \overline \sigma} 1$. Items $(i)$-$(ii)$ then follow by applying Lemmata \ref{embedding L infty}, \ref{lemma azione nostri pseudo}. 
	\\
$(iii)$	A direct calculation shows that 
\begin{equation}\label{prima definizione cal L5}
\begin{aligned}
{\mathcal L}^{(5)}(\f) & := (\Phi^{(5)}_{\omega*}) {\mathcal L}^{(4)}(\f)  = \Phi^{(5)}(\f)^{- 1} {\mathcal L}^{(4)}(\f) \Phi^{(5)}(\f) -  \Phi^{(5)}(\f)^{- 1} \omega \cdot \partial_\f \Phi^{(5)}(\f) \\
& = \ii \lambda_2 \partial_{xx} + \lambda_1 \partial_x + d_0(x, \f) +  [\ii \lambda_2 \partial_{xx} + \lambda_1 \partial_x , {\mathcal V}(\f)]  \\
& \quad - \Phi^{(5)}(\f)^{- 1} \omega \cdot \partial_\f \Phi^{(5)}(\f) + {\mathcal R}^{(I)}(\f)
\end{aligned}
\end{equation}
where the remainder ${\mathcal R}^{(I)}(\f)$ is given by 
\begin{equation}
\begin{aligned}
{\mathcal R}^{(I)}(\f) & :=  \int_0^1 (1 - t) {\rm exp}(-\tau {\mathcal V}(\f)) \, [[\ii \lambda_2 \partial_{xx} + \lambda_1 \partial_x, {\mathcal V}(\f)], {\mathcal V}(\f)]\, {\rm exp}(\tau {\mathcal V}(\f))\, d \tau \\
& \quad + \int_0^1 e^{- \tau {\mathcal V}(\f)} [d_0, {\mathcal V}(\f)] e^{  \tau {\mathcal V}(\f)}\, d \tau \,. 
\end{aligned}
\end{equation}
By recalling \eqref{espansione cal A f}, \eqref{stima cal R cal A}, by applying Lemma \ref{lemma norma an simboli omogenei} and using that $\lambda_2 = 1 + O(\e)$ and $\lambda_1 = O(\e)$, one obtains that 
$$
[\ii \lambda_2 \partial_{xx} + \lambda_1 \partial_x , {\mathcal A}(\f)] =  2\ii  \lambda_2 v_x(x, \f)  + a_v^{(- 1)}(x, \f) \partial_x^{- 1} + {\mathcal R}^{(II)}(\f)   
$$
where for any $0 < \sigma < \bar \sigma$, $a_v^{(1)} \in {\mathcal H}(\T_\s\times\T^\infty_\s)$, ${\mathcal R}^{(II)} \in {\mathcal H}\big( \T^\infty_\sigma,  {\mathcal B}^{\sigma, - 2}\big)$ with 
\begin{equation}\label{mafia romana 10}
\| a_v^{(- 1)}\|_{\sigma}^\Lipg\,,\, |{\mathcal R}^{(II)}|_{\sigma, - 2}^\Lipg \lesssim_{\sigma} \e \, 
\end{equation}
and 
\begin{equation}\label{doppio commutatore cal A vphi}
\begin{aligned}
& [[\ii \lambda_2 \partial_{xx} + \lambda_1 \partial_x , {\mathcal V}], {\mathcal V} ] \in {\mathcal H}(\T^\infty_\sigma, {\mathcal B}^{\sigma, - 2})\,,  \\
& \Big| [[\ii \lambda_2 \partial_{xx} + \lambda_1 \partial_x , {\mathcal V}], {\mathcal V} ] \Big|_{\sigma, - 2}^\Lipg \lesssim_{\sigma, \bar \sigma} \e\,. 
\end{aligned}
\end{equation}
Moreover, using the estimate on $d_0$ provided in Lemma \ref{lemma stime step 4} and by applying again Lemma \ref{lemma norma an simboli omogenei}, one gets that 
\begin{equation}\label{stima commutatore d0 cal A}
[d_0, {\mathcal V} ] \in {\mathcal H}(\T^\infty_\sigma, {\mathcal B}^{\sigma, - 2}), \quad |[d_0, {\mathcal V}]|_{\sigma, - 2}^\Lipg \lesssim_{\sigma, \bar \sigma} \e\,. 
\end{equation}
By applying Lemma \ref{proprieta norma sigma riducibilita vphi x}, using Lemma \ref{lemma mappa esponenziale} and the estimate \eqref{stima cal R cal A} to bound ${\rm exp}(\pm \tau {\mathcal V}(\f))$ and by applying the estimates \eqref{doppio commutatore cal A vphi}, \eqref{stima commutatore d0 cal A}, one obtains that 
\begin{equation}\label{stima cal R (I) step 5}
{\mathcal R}^{(I)} \in {\mathcal H}\big(\T^\infty_\sigma, {\mathcal B}^{\sigma, - 2} \big), \quad |{\mathcal R}^{(I)}|_{\sigma, - 2}^\Lipg \lesssim_{\sigma, \bar \sigma} \e\,. 
\end{equation} 
Moreover, recalling the definition of the operator $\Phi^{(5)}$ given in \eqref{definizione Phi (5)}, using \eqref{espansione cal A f}, \eqref{stima cal R cal A} and by applying Lemmata \ref{lemma norma an simboli omogenei}, \ref{lemma mappa esponenziale}, one obtains that 
\begin{equation}\label{Phi (5) om vphi Phi (5)}
\begin{aligned}
& - \Phi^{(5)}(\f)^{- 1} \omega \cdot \partial_\f \Phi^{(5)}(\f) = - \omega \cdot \partial_\f  v(x, \f) \partial_x^{- 1} + {\mathcal R}^{(III)}(\f), \\
&{\mathcal R}^{(III)}(\f) \in {\mathcal H}\big(\T^\infty_\sigma, {\mathcal B}^{\sigma, - 2} \big), \quad  |{\mathcal R}^{(III)}|_{\sigma, - 2}^\Lipg \lesssim_{\sigma, \bar \sigma} \e\,, \quad \forall 0 < \sigma < \bar \sigma\,.  
\end{aligned}
\end{equation}

 and therefore by \eqref{prima definizione cal L5} one gets 
 \begin{equation}\label{seconda definizione cal L 5}
 \begin{aligned}
 {\mathcal L}^{(5)}(\f) & = \lambda_2 \partial_{xx} + \lambda_1 \partial_x + d_0+ 2\lambda_2 v_x +  e_{- 1}(x, \f) \partial_x^{- 1} + {\mathcal R}^{(5)}(\f)\,, \\
 e_{- 1} (x, \f) & :=  a_v^{(- 1)}(x, \f)  - \omega \cdot \partial_\f  v(x, \f)\,, \quad {\mathcal R}^{(5)}(\f)  := {\mathcal R}^{(I)}(\f) + {\mathcal R}^{(II)}(\f) + {\mathcal R}^{(III)}(\f)\,.
 \end{aligned}
 \end{equation}
 The claimed statement then follows since $d_0+ 2\ii \lambda_2 v_x  = \langle d_0 \rangle_x$ (see \eqref{def v step 5}), by the estimate \eqref{stima v step 5} on $v$, the estimate \eqref{mafia romana 10} on $a_v^{(- 1)}$ and the estimates \eqref{mafia romana 10}, \eqref{stima cal R (I) step 5}, \eqref{Phi (5) om vphi Phi (5)} on ${\mathcal R}^{(I)}, {\mathcal R}^{(II)}, {\mathcal R}^{(III)}$. 
 \end{proof}
 
 \subsection{Elimination of the $x$ dependence from the order $- 1$.}
In order to eliminate the $x$-dependence from the term of order $- 1$ in the operator ${\mathcal L}^{(5)}$ given in \eqref{cal L (5) definitivo}, We conjugate such an operator by means of a transformation 
\begin{equation}\label{definizione Phi (5)a}
\Phi^{(6)}(\f) := {\rm exp}( {\mathcal G}(\f)) \quad \text{where} \quad {\mathcal G}(\f) := \frac{\ii}{2}\big( g(x, \f) \circ \partial_x^{- 2} + \partial_x^{- 2} \circ g(x, \f) \big).
\end{equation}
and $g(x, \f)$ is a real on real function to be determined. Note that for real values of the angle $\f \in \T^\infty$, one has that ${\mathcal G}(\f) = - {\mathcal G}(\f)^*$, implying that $\Phi^{(6)}(\f)$ is unitary. 
\begin{lemma}\label{lemma stime step 5a}
Let $\omega \in \Dc$. For any $\sigma < \bar \sigma$ there exists $\delta(\sigma, \bar \sigma) > 0$ such that if $\e \gamma^{- 1} \leq \delta$, the following holds. Define 
\begin{equation}\label{def v step 5a}
g (x, \f):= \frac{1}{2  \lambda_2} \partial_x^{- 1}\big[e_{- 1}(x, \f) - \langle e_{- 1} \rangle_x(\f)   \big]\,.
\end{equation}
$(i)$  the map $\T^\infty_\sigma \to {\mathcal B}({\mathcal H}\big(\T_{\sigma}), {\mathcal H}(\T_{\sigma}) \big)$, $\f \mapsto \Phi^{(6)}(\f)^{\pm 1}$ is bounded. 
\\
$(ii)$ For any $s \geq 0$, the map $\T^\infty \to {\mathcal B}(H^s(\T), H^s(\T) \big)$, $\f \mapsto \Phi^{(6)}(\f)^{\pm 1}$ is bounded. 
\\
$(iii)$ The map $\Phi^{(6)}(\f)$ transform the operator ${\mathcal L}^{(5)}(\f)$ as 
\begin{equation}\label{cal L (5) definitivoa}
{\mathcal L}^{(6)}(\f) = (\Phi^{(6)}_{\omega*}){\mathcal L}^{(5)}(\f) = \lambda_2 \partial_{xx} + \lambda_1 \partial_x + \langle d_0 \rangle_x(\f) + \langle e_{- 1} \rangle_x(\f) \partial_x^{- 1} + {\mathcal R}^{(6)}(\f)
\end{equation}
where the function $g(\cdot; \omega)  \in {\mathcal H}(\T_\s\times\T^\infty_\s)$ is real on real and the operator ${\mathcal R}^{(6)}(\omega) \in {\mathcal H}\big(\T^\infty_\sigma, {\mathcal B}^{\sigma, - 2} \big)$ is skew self-adjoint. Moreover they are defined $\omega \in \Dc$ and
satisfy  for all $\Oo\subseteq \Dc$, the estimates
\begin{equation}\label{stima v step 5a}
\begin{aligned}
&  \| g \|_{\sigma}^\Lipg\,,\, |{\mathcal R}^{(6)}|_{\sigma, - 2}^\Lipg \lesssim_{\sigma, \bar \sigma} \e \,.   
\end{aligned}
\end{equation}
\end{lemma}
\begin{proof}
 By the definition \eqref{def v step 5a}, using the estimates on $e_{- 1}$ given in Lemma \ref{lemma stime step 5}, one gets that $g$ satisfies the estimate \eqref{stima v step 5a}. By Lemma \ref{lemma norma an simboli omogenei} and by the estimate on $g$ one has that for any $0 < \sigma < \bar \sigma$, \begin{equation}\label{stima cal R cal Aa}
{\mathcal G} \in {\mathcal H}\big(\T^\infty_\sigma, {\mathcal B}^{\sigma, - 2} \big)\,, \quad |{\mathcal G}|_{\sigma, - 2}^\Lipg \lesssim_{\sigma} \e \,. 
\end{equation}
The above estimate and Lemma \ref{lemma mappa esponenziale}, using that $\omega \cdot \partial_\f \Phi^{(6)} = \omega \cdot \partial_\f (\Phi^{(6)} - {\rm Id})$, imply that for any $0 < \sigma < \bar \sigma$
\begin{equation}\label{stima Phi (6) nel lemma}
\begin{aligned}
\sup_{\tau \in [0, 1]}| {\rm exp}(\pm \tau {\mathcal G})|_\sigma^\Lipg \lesssim_{\sigma, \bar \sigma} 1, \quad |\omega \cdot \partial_\f (\Phi^{(6)})|_{\sigma, - 2 }^\Lipg \lesssim_{\sigma, \bar \sigma} \e\,.
\end{aligned}
\end{equation}
Items $(i)$-$(ii)$ follow by the estimate \eqref{stima Phi (6) nel lemma} and by applying Lemmata \ref{embedding L infty}, \ref{lemma azione nostri pseudo}. 
\\
$(iii)$ A direct calculation shows that 
\begin{equation}\label{prima definizione cal L5a}
\begin{aligned}
{\mathcal L}^{(6)}(\f) & := (\Phi^{(6)}_{\omega*}) {\mathcal L}^{(5)}(\f)  = \Phi^{(6)}(\f)^{- 1} {\mathcal L}^{(5)}(\f) \Phi^{(6)}(\f) -  \Phi^{(6)}(\f)^{- 1} \omega \cdot \partial_\f \Phi^{(6)}(\f) \\
& = \ii \lambda_2 \partial_{xx} + \lambda_1 \partial_x + \langle d_0 \rangle(\f) + e_{- 1}(x, \f) \partial_x^{- 1} +  [\ii \lambda_2 \partial_{xx} + \lambda_1 \partial_x , {\mathcal G}(\f)]  \\
& \quad +  {\mathcal R}^{(I)}(\f)
\end{aligned}
\end{equation}
where the remainder ${\mathcal R}(\f)$ is given by 
\begin{equation}\label{cal R (I) step 6}
\begin{aligned}
{\mathcal R}^{(I)}(\f) & :=  \int_0^1 (1 - t) {\rm exp}(-\tau {\mathcal G}(\f)) \, [[\ii \lambda_2 \partial_{xx} + \lambda_1 \partial_x, {\mathcal G}(\f)], {\mathcal G}(\f)]\, {\rm exp}(\tau {\mathcal G}(\f))\, d \tau \\
& \quad + \int_0^1 e^{-  \tau {\mathcal G}(\f)} \Big([\langle d_0 \rangle_x + e_{- 1} \partial_x^{- 1}, {\mathcal G}(\f)] \Big) e^{ \tau {\mathcal G}(\f)}\, d \tau -   \Phi^{(6)}(\f)^{- 1} \omega \cdot \partial_\f \Phi^{(6)}(\f)  \,. 
\end{aligned}
\end{equation}
By recalling  the estimate of Lemma \ref{lemma stime step 4} on $d_0$, the estimate of Lemma \ref{lemma stime step 5} on $e_{- 1}$, the estimate \eqref{stima cal R cal Aa} on ${\mathcal G}$, by applying Lemmata \ref{lemma norma an simboli omogenei}, \ref{proprieta norma sigma riducibilita vphi x} and using that $\lambda_2 = 1 + O(\e)$ and $\lambda_1 = O(\e)$, one obtains that for any $0 < \sigma < \bar \sigma$ 
\begin{equation}\label{mafia romana 10a}
\begin{aligned}
& [[\lambda_2 \partial_{xx} + \lambda_1 \partial_x, {\mathcal G}(\f)], {\mathcal G}(\f)]\,,\, [\langle d_0 \rangle_x + e_{- 1} \partial_x^{- 1}, {\mathcal G}(\f)] \in {\mathcal H}\big(\T^\infty_\sigma, {\mathcal B}^{\sigma, - 2} \big)\,, \\
& \Big| [[\lambda_2 \partial_{xx} + \lambda_1 \partial_x, {\mathcal G}(\f)], {\mathcal G}(\f)] \Big|_{\sigma, - 2}^\Lipg\,,\, \Big| [\langle d_0 \rangle_x + e_{- 1} \partial_x^{- 1}, {\mathcal G}(\f)]\Big|_{\sigma, - 2}^\Lipg \lesssim_{\sigma, \bar \sigma} \e\,.  
\end{aligned}
\end{equation}
Therefore, the estimates \eqref{mafia romana 10a}, \eqref{stima Phi (6) nel lemma} and Lemma \ref{proprieta norma sigma riducibilita vphi x} imply that the remainder ${\mathcal R}^{(I)}$ defined in \eqref{cal R (I) step 6} satisfies 
\begin{equation}\label{cal R (I) stima step 6}
{\mathcal R}^{(I)} \in {\mathcal H}\big(\T^\infty_\sigma, {\mathcal B}^{\sigma, - 2} \big), \quad |{\mathcal R}^{(I)}|_{\sigma, - 2}^\Lipg \lesssim_{\sigma, \bar \sigma} \e, \quad \forall 0 < \sigma < \bar \sigma\,. 
\end{equation}
Recalling the definition of ${\mathcal G}$, using the estimate \eqref{stima v step 5a} on $g$, by applying Lemma \ref{lemma norma an simboli omogenei} and using that $\lambda_2 = 1 + O(\e)$, $\lambda_1 = O(\e)$, one gets that 
\begin{equation}\label{espansione commutatore cal G}
[\ii \lambda_2 \partial_{xx} + \lambda_1 \partial_x , {\mathcal G}(\f)] = - 2 \lambda_2  g_x\partial_x^{- 1} + {\mathcal R}^{(II)}(\f)
\end{equation}
where for any $0 < \sigma < \bar \sigma$, 
\begin{equation}\label{stima cal R (II) step 6}
{\mathcal R}^{(II)} \in {\mathcal H}\big(\T^\infty_\sigma, {\mathcal B}^{\sigma, - 2} \big), \quad |{\mathcal R}^{(II)}|_{\sigma, - 2}^\Lipg \lesssim_{\sigma, \bar \sigma} \e\,. 
\end{equation}
Therefore by \eqref{prima definizione cal L5a}, one gets 
 \begin{equation}\label{seconda definizione cal L 5a}
 \begin{aligned}
 {\mathcal L}^{(6)}(\f) & = \lambda_2 \partial_{xx} + \lambda_1 \partial_x + \langle d_0 \rangle_x + \big(- 2\lambda_2 g_x + e_{- 1} \big)\partial_x^{- 1} + {\mathcal R}^{(6)}(\f)\,, \\
 {\mathcal R}^{(6)}(\f)  &:= {\mathcal R}^{(I)}(\f) + {\mathcal R}^{(II)}(\f)\,.
 \end{aligned}
 \end{equation}
 The claimed statement then follows since $e_{- 1}- 2\lambda_2 g_x  = \langle e_{- 1} \rangle_x$ (see \eqref{def v step 5a}) and by recalling \eqref{cal R (I) stima step 6}, \eqref{stima cal R (II) step 6}.
 \end{proof}
\subsection{Reduction to constant coefficients up to order $- 2$.}
In the last step of our regularization procedure, we eliminate the $\f$-dependence from the term $\langle d_0 \rangle_x(\f) + \langle e_{- 1} \rangle(\f) \partial_x^{- 1}$. To achieve this purpose, we consider the map 
\begin{equation}\label{definizione Phi (7)}
\Phi^{(7)}(\f) := {\rm exp}({\mathcal F}(\f)), \quad {\mathcal F}(\f) :=  {\rm diag}_{j \in \Z} f_j(\f)
\end{equation}
where for any $j \in \Z$, $f_j$ are analytic functions to be determined which are purely imaginary for any real value of the angle $\f$. We prove the following lemma. 
\begin{lemma}\label{lemma stime step 7}
Let $\omega \in \Dc$. For any $0 < \sigma < \bar \sigma$ there exists $\delta(\sigma, \bar \sigma) > 0$ such that if $\e \gamma^{- 1} \leq \delta$, the following holds. Define 
\begin{equation}\label{soluzione eq omologica step 7}
\begin{aligned}
& \lambda_0 := \frac{1}{\ii}\langle d_0 \rangle_{x, \f}, \quad \lambda_{- 1} := \langle e_{- 1} \rangle_{x, \f}\,, \\
& {\mathcal F}(\f) := (\omega \cdot \partial_\f)^{- 1}[\langle d_0 \rangle_x -\ii \lambda_0] + (\omega \cdot \partial_\f)^{- 1}[e_{- 1} - \lambda_{- 1}] \partial_x^{- 1}\,. 
\end{aligned}
\end{equation}
$(i)$  the map $\T^\infty_\sigma \to {\mathcal B}({\mathcal H}\big(\T_{\sigma}), {\mathcal H}(\T_{\sigma}) \big)$, $\f \mapsto \Phi^{(7)}(\f)^{\pm 1}$ is bounded. 
\\
$(ii)$ For any $s \geq 0$, the map $\T^\infty \to {\mathcal B}(H^s(\T), H^s(\T) \big)$, $\f \mapsto \Phi^{(7)}(\f)^{\pm 1}$ is bounded. 
\\
$(iii)$ The map $\Phi^{(7)}(\f)$ transform the operator ${\mathcal L}^{(6)}(\f)$ in 
\begin{equation}\label{cal L (7)}
{\mathcal L}^{(7)}(\f) : = (\Phi^{(7)}_{\omega*}){\mathcal L}^{(6)}(\f) = \ii \lambda_2 \partial_{xx} + \lambda_1 \partial_x + \ii \lambda_0 + \lambda_{- 1} \partial_x^{- 1} + {\mathcal R}^{(7)}(\f)
\end{equation}
where $\lambda_0, \lambda_{- 1} \in \R$ and the operator ${\mathcal R}^{(7)} \in {\mathcal H}\big(\T^\infty_\sigma, {\mathcal B}^{\sigma, - 2} \big)$
satisfy the estimates
\begin{equation}\label{stima lambda 0 - 1}
\begin{aligned}
&  |\lambda_0|^\Lipg, |\lambda_{- 1}|^\Lipg \lesssim \e\,,\, |{\mathcal R}^{(7)}|_{\sigma, - 2}^\Lipg \lesssim_{\sigma, \bar \sigma} \e \,.   
\end{aligned}
\end{equation}
\end{lemma}
\begin{proof}
Since the operator ${\mathcal F}(\f)$ is a diagonal operator, one has that $[{\mathcal F}(\f), \partial_x^k] = 0$ for any $k \in \Z$ and a direct calculation shows that 
\begin{equation}\label{pippino 0}
 \Phi^{(7)}(\f)^{- 1} \omega \cdot \partial_\f \Phi^{(7)}(\f) =  \omega \cdot \partial_\f {\mathcal F}(\f)\,. 
\end{equation}
Therefore, by the definition \eqref{soluzione eq omologica step 7}, we solve the homological equation 
\begin{equation}\label{eq omologica step 7}
- \omega \cdot \partial_\f {\mathcal F}(\f)+ \langle d_0 \rangle_x + \langle e_{- 1} \rangle_x \partial_x^{- 1} = \ii \lambda_0 + \lambda_{- 1} \partial_x^{- 1}\,. 
\end{equation}
By the estimates \eqref{stime q d0} on $d_0$ and \eqref{stima v step 5} on $e_{- 1}$ one gets that $|\lambda_0|^\Lipg, |\lambda_{- 1}|^\Lipg \lesssim \e$ and 
by applying Lemmata \ref{stima diofantea D omega inv}, \ref{lemma norma an simboli omogenei} one obtains that for any $0 < \sigma < \bar \sigma$, 
\begin{equation}\label{stima cal F step 7}
{\mathcal F} \in {\mathcal H}(\T^\infty_\sigma, {\mathcal B}^\sigma), \quad |{\mathcal F}|_\sigma^\Lipg \lesssim_{\sigma, \bar \sigma} \e \gamma^{- 1}\,. 
\end{equation}
The latter estimate, together with Lemma \ref{lemma mappa esponenziale} imply that 
\begin{equation}\label{Phi (7) stima nel lemma}
(\Phi^{(7)})^{\pm 1} \in {\mathcal H}(\T^\infty_\sigma, {\mathcal B}^\sigma), \quad |(\Phi^{(7)})^{\pm 1}|_\sigma^\Lipg \leq 1 + C(\sigma, \bar \sigma) \e \gamma^{- 1}
\end{equation}
for some constant $C(\sigma, \bar \sigma) > 0$. 
Hence, one obtains that 
\begin{equation}\label{cal L (7) a}
\begin{aligned}
{\mathcal L}^{(7)}(\f) & = (\Phi^{(7)}_{\omega*}){\mathcal L}^{(6)}(\f) = \ii \lambda_2 \partial_{xx} + \lambda_1 \partial_x - \omega \cdot \partial_\f {\mathcal F}(\f)+ \langle d_0 \rangle_x + \langle e_{- 1} \rangle_x \partial_x^{- 1} + {\mathcal R}^{(7)}(\f)\,, \\
{\mathcal R}^{(7)}(\f) & := \Phi^{(7)}(\f)^{- 1} {\mathcal R}^{(6)}(\f) \Phi^{(7)}(\f)\,. 
\end{aligned}
\end{equation}
The estimate \eqref{stima lambda 0 - 1} on the operator ${\mathcal R}^{(7)}$, defined in \eqref{cal L (7) a}, follows by the composition Lemma \ref{proprieta norma sigma riducibilita vphi x}, by the estimate \eqref{stima v step 5a} on ${\mathcal R}^{(6)}$ and by the estimate \eqref{Phi (7) stima nel lemma} on $(\Phi^{(7)})^{\pm 1}$. 
\end{proof}

\section{The KAM reducibility scheme}\label{sez KAM redu}
In this section we carry out the reducibility of the equation $\partial_t u = {\mathcal L}_0(\omega t) u$ where the operator ${\mathcal L}_0 \equiv {\mathcal L}^{(7)}$ is given in Lemma \ref{lemma stime step 7}. We fix
\begin{equation}\label{costante iniziale rid}
\sigma_0 := \frac{\bar \sigma}{2}\,. 
\end{equation}
The operator ${\mathcal L}_0(\f) \equiv {\mathcal L}_0(\f; \omega)$ defined for $\omega \in \Dc$, has the form
\begin{equation}\label{def primo op rid}
{\mathcal L}_0(\f) = \ii {\mathcal D}_0 + {\mathcal P}_0(\f)
\end{equation}
where  for all $\Oo\in \Dc$
\begin{equation}\label{cal D0 cal P0 inizio KAM}
\begin{aligned}
& {\mathcal D}_0 :=  \lambda_2 \partial_{xx} + \frac{1}{\ii}\lambda_1 \partial_x +  \lambda_0 + \frac{1}{\ii}\lambda_{- 1} \partial_x^{- 1}\,, \\
& \lambda_2, \lambda_1, \lambda_0, \lambda_{- 1} \in \R, \quad |\lambda_2 - 1|, |\lambda_1|, |\lambda_0|^\Lipg, |\lambda_{- 1}|^\Lipg \lesssim \e\,, \\
& |{\mathcal P}_0|_{\sigma_0, - 2}^\Lipg \lesssim_{\sigma_0} \e\,. 
\end{aligned}
\end{equation}
Note that, as we pointed out in the previous section, the real constants $\lambda_2, \lambda_1$ do not depend on the parameter $\omega$. The linear operator ${\mathcal D}_0$ is a $2 \times 2$ block diagonal operator ${\mathcal D}_0 = {\rm diag}_{j \in \N_0} {\mathcal D}_0(j)$ where for any $j \in \N_0$, the $2 \times 2$ block ${\mathcal D}_0(j)$ is given by 
\begin{equation}\label{blocco iniziale riducibilita}
\begin{aligned}
& {\mathcal D}_0(j) :=  \begin{pmatrix}
\mu_j^{(0)} & 0 \\
0 & \mu_{- j}^{(0)}
\end{pmatrix}\,, \\
&
 \mu_{j}^{(0)} := - \lambda_2 j^2 + \lambda_1 j + \lambda_0 - \lambda_{- 1} j^{- 1}\,, \quad  \mu_{- j}^{(0)} :=  -\lambda_2 j^2 - \lambda_1 j + \lambda_0 +  \lambda_{- 1} j^{- 1}\,.
\end{aligned}
\end{equation}
In order to state our reducibility Theorem, we fix some other constants.
For $n\ge 1$, we set
\begin{equation}\label{costanti convergenza riducibilita}
\chi \in (1, 2), \quad \s_n = \s_0\Big( 1- \frac{1}{4\pi} \sum_{j=1}^{n} \frac{1}{j^2} \Big)\,, \quad  N_n =  \langle n \rangle^3 \chi^n  N_0
\end{equation}
and to shorten notation, we set  
\begin{equation}\label{def eta (ell)}
{\divisor}(\ell) := \prod_{n\in \N}(1+|\ell_n|^{4} \jap{n}^{4}), \quad \forall \ell \in \Z^\infty_*\,. 
\end{equation}
\begin{theorem}[{\bf Reducibility}]\label{teorema di riducibilita}
Let $\gamma \in (0, 1)$. Then there exists $\delta \in (0, 1)$ small enough such that if $\e \gamma^{- 1} \leq \delta$, for any $n \geq 0$, the following holds. 

\noindent
 $\bf{(S1)_n}$ There exists a linear skew self-adjoint vector field 
 \begin{equation}\label{cal L n rid}
 {\mathcal L}_n(\f) = \ii {\mathcal D}_n + {\mathcal P}_n(\f)
 \end{equation}
 where ${\mathcal D}_n$ is a $2 \times 2$ self-adjoint block diagonal operator ${\mathcal D}_n = {\rm diag}_{j \in \N_0} {\mathcal D}_n(j)$,
  ${\mathcal P}_n \in {\mathcal H}\big(\T^\infty_{\sigma_n}, {\mathcal B}^{\sigma_n, - 2}\big)$ is  skew self-adjoint, moreover both are defined for  $\omega \in \Omega_n(\gamma)$, where $\Omega_0(\gamma) := \Dc$ and for any $n \geq 1$
 \begin{equation}\label{definizione Omega n (gamma) rid}
 \begin{aligned}
 \Omega_n(\gamma) & := \Big\{ \omega \in \Omega_{n - 1}(\gamma) : \| {\mathcal O}_{n - 1}(\ell, j, j')^{- 1}\|_\Op \leq \frac{{\divisor}(\ell)}{\gamma}, \quad \forall (\ell, j, j') \in \Z^\infty_* \times \N_0 \times \N_0, \\
 & j \neq j' \quad\text{and} \quad \| {\mathcal O}_{n - 1}(\ell, j, j)^{- 1}\|_\Op \leq \frac{{\divisor}(\ell) j^2}{\gamma} \quad \forall (\ell, j) \in (\Z^\infty_* \setminus \{ 0 \}) \times \N_0, \quad |\ell|_\zia \leq N_{n - 1} \Big\}\,.
 \end{aligned}
 \end{equation}
 For any $(\ell, j, j') \in \Z^\infty_* \times \N_0 \times \N_0$, the operators ${\mathcal O}_{n - 1}(\ell, j, j') : {\mathcal B}({\bf E}_{j'}, {\bf E}_j) \to {\mathcal B}({\bf E}_{j'}, {\bf E}_j)$ are defined by 
 \begin{equation}\label{def cal O n ell j j'}
 {\mathcal O}_{n - 1}(\ell, j, j'):= \omega \cdot \ell \,{\rm Id} + M_L({\mathcal D}_{n - 1}(j)) - M_R({\mathcal D}_{n - 1}(j'))\,. 
 \end{equation}
 For any $j \in \N_0$, 
 \begin{equation}\label{D0 j - Dn j}
 \| {\mathcal D}_n(j) - {\mathcal D}_0(j) \|^\Lipn{n}_{\HS} \lesssim \e \,.
 \end{equation}
 and 
 \begin{equation}\label{stima cal P sigma n}
 |{\mathcal P}_n|_{\sigma_n, - 2}^\Lipn{n} \leq C_* \e e^{- \chi^n}  
 \end{equation}
 for some constant $C_* > 0$. 
 \\
 For $n \geq 1$, there exists a map $\Phi_n(\f) := {\rm exp}( {\mathcal F}_n(\f))$, where ${\mathcal F}_n \in {\mathcal H}\Big( \T^\infty_{\frac{\sigma_{n - 1} + \sigma_n}{2}},  {\mathcal B}^{\frac{\sigma_{n - 1} + \sigma_n}{2}}\Big)$  is skew self adjoint and defined for  $\omega \in \Omega_n(\gamma)$, which satisfies 
 \begin{equation}\label{coniugazione passo n KAM}
 {\mathcal L}_n(\f) = (\Phi_n)_{\omega*} {\mathcal L}_{n - 1}(\f)\,.
 \end{equation}
 The operator ${\mathcal F}_{n }$ satisfies the estimate 
 \begin{equation}\label{stima trasformazione cal F n - 1}
 |{\mathcal F}_{n }|_{\frac{\sigma_{n - 1} + \sigma_n}{2}}^\Lipn{n} \lesssim  \e \g^{-1} e^{- \frac{\chi^{n - 1}}{2}}
 \end{equation}
 
 \noindent
 ${\bf (S2)_n}$ For any $j \in \N_0$ there exists a Lipschitz extension of the function ${\mathcal D}_n(j ; \cdot ) : \Omega_n(\gamma) \to {\mathcal S}({\bf E}_j)$ to the set $\Dc$, denoted by $\widetilde{\mathcal D}_n(j; \cdot ) : \Dc \to {\mathcal S}({\bf E}_j)$ that, for any $n\ge 1$, satisfies the estimate 
 \begin{equation}\label{stima differenze blocchi estesi}
 \begin{aligned}
 & \sup_{\omega \in \Dc} \| \widetilde{\mathcal D}_n(j; \omega) - \widetilde{\mathcal D}_{n - 1}(j; \omega)\|_{\HS}  \lesssim \langle j \rangle^{- 2}\e e^{- \chi^{n - 1}}\,, \\
 & \| \widetilde{\mathcal D}_n(j) - \widetilde{\mathcal D}_{n - 1}(j)\|^{\rm lip}_{{\HS}}  \lesssim \e \gamma^{- 1} e^{- \chi^{n - 1}}
 \end{aligned}
 \end{equation}
\end{theorem}
\subsection{Proof of Theorem \ref{teorema di riducibilita}}
{\sc Proof  of ${\bf ({S}i)}_{0}$, $i=1, 2$.} 
The claims hold by recalling the properties of the operator ${\mathcal L}_0$ listed in \eqref{def primo op rid}-\eqref{blocco iniziale riducibilita}. 

${\bf({S}2)}_0 $ holds, since the constants $\lambda_2$ and $\lambda_1$ are independent of $\omega$ and $\lambda_0, \lambda_{- 1}$ are already defined on ${\mathtt D}_\gamma$.  
\subsubsection{The reducibility step}
{\sc Proof  of ${\bf ({S}1)}_{n + 1}$.} We now describe the inductive step, showing how to define a symplectic transformation 
$ \Phi_{n + 1} := \exp( {\mathcal F}_{n + 1} ) $ 
so that the transformed vector field $ {\mathcal L}_{n +1 }(\f) = (\Phi_{n + 1})_{\omega*} {\mathcal L}_n(\f) $ has the desired properties. We write $\Pi_n$ instead of $\Pi_{N_n}$ to denote the projector on the Fourier modes $|\ell|_\zia \leq N_n$, where $N_n$ is defined in \eqref{costanti convergenza riducibilita}. A direct calculation shows that 
\begin{equation}\label{prima espansione cal L n + 1}
\begin{aligned}
{\mathcal L}_{n + 1}(\f) & = (\Phi_{n + 1})_{\omega*}{\mathcal L}_n(\f) = \Phi_{n + 1}(\f)^{- 1} {\mathcal L}_n(\f )\Phi_{n + 1}(\f) - \Phi_{n + 1}(\f)^{- 1} \omega \cdot \partial_\f \Phi_{n + 1}(\f) \\
& = \ii {\mathcal D}_n  -  \omega \cdot \partial_\f {\mathcal F}_{n + 1} + [\ii {\mathcal D}_n, {\mathcal F}_{n + 1}] +  \Pi_n {\mathcal P}_n + \Pi_n^\bot {\mathcal P}_n \\
& \quad + \int_0^1 (1 - \tau) e^{- \tau {\mathcal F}_{n + 1}} [[\ii {\mathcal D}_n, {\mathcal F}_{n + 1}], {\mathcal F}_{n + 1}] e^{ \tau {\mathcal F}_{n + 1}}\, d \tau \\
& \quad +  \int_0^1 e^{-  \tau {\mathcal F}_{n + 1}}[{\mathcal P}_n, {\mathcal F}_{n + 1}] e^{ \tau {\mathcal F}_{n + 1}}\, d \tau \\
& \quad  -  \int_0^1 (1 - \tau) e^{-  \tau {\mathcal F}_{n + 1}} [ \omega \cdot \partial_\f {\mathcal F}_{n + 1},  {\mathcal F}_{n + 1}] e^{ \tau {\mathcal F}_{n + 1}}\, d \tau
\end{aligned}
\end{equation}
Our next aim is to solve the Homological equation 
\begin{equation}\label{equazione omologica riducibilita}
-  \omega \cdot \partial_\f {\mathcal F}_{n + 1} + [\ii {\mathcal D}_n, {\mathcal F}_{n + 1}] +  \Pi_n {\mathcal P}_n = [\widehat{\mathcal P}_n(0)]
\end{equation}
where the diagonal part of the operator $\widehat{\mathcal P}_n(0)$ is defined according to \eqref{notazione operatore diagonale a blocchi}. 
\begin{lemma}\label{homolu}
For all $\omega\in \Omega_{n + 1}(\gamma)$ (see \eqref{definizione Omega n (gamma) rid}), there exists a unique solution $F_{n + 1} \in {\mathcal H}\big(\T^\infty_{\sigma_n - \rho},  {\mathcal B}^{\sigma_n - \rho}\big)$ with $\rho > 0$, $\sigma_n - \rho > 0$ of the Homological equation \eqref{equazione omologica riducibilita} satisfying the bound 
	\begin{equation}\label{stima F}
	|{\mathcal F}_{n + 1} |_{\s_n-\rho}^{\Lipn{n+1}} \lesssim  \g^{-1}{\rm exp}\Big(\frac{\tau}{\rho^{\frac{1}{\zia}}} \ln\Big(\frac{\tau}{\rho} \Big) \Big)|{\mathcal P}_n|_{\s_n,-2}^{\Lipn{n}}
\end{equation}
for some appropriate constant $\tau>1$.
\end{lemma}
\begin{proof}
In order to simplify notations in this proof, we drop the index $n$ and we write $+$ instead of $n + 1$. Passing to the $2 \times 2$ block representation of operators and taking the Fourier transform w.r. to $\f$, one gets that the equation \eqref{equazione omologica riducibilita} is equivalent to 	
\begin{equation}\label{decomposizione a blocchi eq omologica}
\begin{aligned}
& \ii \Big( -  \omega \cdot \ell \Pi_j \widehat{\mathcal F}_+(\ell) \Pi_{j'}  + {\mathcal D}(j) \Pi_j \widehat{\mathcal F}_+(\ell) \Pi_{j'} -  \Pi_j \widehat{\mathcal F}_+(\ell) \Pi_{j'} {\mathcal D}(j') \Big)  + \Pi_j \widehat{\mathcal P}(\ell) \Pi_{j'}  = 0 \\
& \forall (\ell, j, j') \in \Z^\infty_* \times \N_0 \times \N_0, \quad (\ell, j, j') \neq (0, j, j), \quad |\ell|_\zia \leq N, \\
& \quad \text{and} \quad \Pi_j \widehat{\mathcal F}_+(0) \Pi_j = 0, \quad \forall j \in \N_0\,. 
\end{aligned}
\end{equation}
According to the definition given in \eqref{def cal O n ell j j'}, for any $\omega \in \Omega_+(\gamma) \equiv \Omega_{n + 1}(\gamma)$, since the operator
\begin{equation}
{\mathcal O}(\ell, j, j') := \omega \cdot \ell \, {\rm Id} - M_L({\mathcal D}(j)) + M_R({\mathcal D}(j'))
\end{equation}
is invertible, one defines ${\mathcal F}_+$ as 
\begin{equation}\label{definizione sol eq omologica nel lemma}
\Pi_j \widehat{\mathcal F}_+(\ell) \Pi_{j'} := \begin{cases}
- \ii {\mathcal O}(\ell, j, j')^{- 1} \Pi_j \widehat{\mathcal P}(\ell) \Pi_{j'}, \quad \forall (\ell, j, j') \neq (0, j, j) \\
0 \quad \forall (\ell, j, j') = (0, j, j)\,. 
\end{cases}
\end{equation}
For any $(\ell, j, j') \neq (0, j, j)$, $j \neq j'$, $|\ell| \leq N$ one obtains that 
\begin{equation}\label{prima stima norma sup eq omologica}
\| \Pi_j \widehat{\mathcal F}_+(\ell) \Pi_{j'} \|_\HS \leq \frac{{\divisor}(\ell)}{\gamma} \| \Pi_j \widehat{\mathcal P}(\ell) \Pi_{j'}\|_\HS
\end{equation}
and for $\ell \neq 0$, $|\ell|_\zia \leq N$, 
\begin{equation}\label{seconda stima norma sup eq omologica}
\| \Pi_j \widehat{\mathcal F}_+(\ell) \Pi_{j} \|_\HS \leq \frac{{\divisor}(\ell) \langle j \rangle^2}{\gamma} \| \Pi_j \widehat{\mathcal P}(\ell) \Pi_j \|_\HS\,. 
\end{equation}
Let $\sigma \equiv \sigma_n$. By recalling the definition \eqref{definizione classe cal B sigma}, the estimates \eqref{prima stima norma sup eq omologica}, \eqref{seconda stima norma sup eq omologica} imply that  for any $\ell \in \Z^\infty$, $|\ell|_\zia \leq N$
\begin{equation}
\| \widehat{\mathcal F}_+(\ell) \|_{{\mathcal B}^{\sigma - \rho}} \leq {\divisor}(\ell) \gamma^{- 1} \| \widehat{\mathcal P}(\ell) \|_{{\mathcal B}^{\sigma, - 2}}\,. 
\end{equation}
Hence in view of the definition \eqref{definizione norma del decay}, one obtains that 
\begin{equation}\label{stima sup cal F+}
\begin{aligned}
|{\mathcal F}_+|_{\sigma - \rho} & \leq \gamma^{- 1}\sum_{\ell \in \Z^\infty_*} {\divisor}(\ell) e^{(\sigma - \rho) |\ell|_\zia} \| \widehat{\mathcal P}(\ell)\|_{{\mathcal B}^{\sigma, - 2}} \leq \gamma^{- 1} \Big(\sup_{\ell \in \Z^\infty_*} {\divisor}(\ell) e^{- \rho |\ell|_\zia} \Big)  |{\mathcal P}|_{\sigma, - 2} \\
& \stackrel{Lemma \,\ref{bound per stima di Cauchy}}{\leq} \gamma^{- 1} {\rm exp}\Big(\frac{\tau}{\rho^{\frac{1}{\zia}}} \ln\Big(\frac{\tau}{\rho} \Big) \Big) |{\mathcal P}|_{\sigma, - 2}\,. 
\end{aligned}
\end{equation}
Now we show the Lipschitz estimate. Let $\omega_1, \omega_2 \in \Omega_+(\gamma)$. Then for any $(\ell, j, j') \neq (0, j, j')$, $|\ell|_\zia \leq N$, 
\begin{equation}\label{formula delta 12 blocchi eq omologica}
\begin{aligned}
	\Delta_{\omega_1 \omega_2}\big( \Pi_j \widehat{\mathcal F}_+(\ell) \Pi_{j'}\big)  & = - \ii {\mathcal O}(\ell, j, j'; \omega_1)^{- 1}\Delta_{\omega_1\omega_2}\big( \Pi_j \widehat{\mathcal P}(\ell) \Pi_{j'}\big) \\
	& \quad  + \ii {\mathcal O}(\ell, j, j'; \omega_1)^{- 1} \big(\Delta_{\omega_1 \omega_2} {\mathcal O}(\ell, j, j') \big) {\mathcal O}(\ell, j, j'; \omega_2)^{- 1}   \Pi_j \widehat{\mathcal P}(\ell; \omega_2) \Pi_{j'}\,.
	\end{aligned}
\end{equation}
By \eqref{norma operatoriale ML MR}, \eqref{cal D0 cal P0 inizio KAM}, \eqref{blocco iniziale riducibilita}, \eqref{D0 j - Dn j}, one obtains that 
\begin{equation}\label{stima lip op sec mel}
\begin{aligned}
\| \Delta_{\omega_1 \omega_2} {\mathcal O}(\ell, j, j')\|_\Op & \leq \|\omega_1 - \omega_2 \|_\infty |\ell|_\zia + 2 \sup_{j \in \N_0} \|\Delta_{\omega_1\omega_2} {\mathcal D}(j) \|_\HS \\
& \lesssim (1 + |\ell|_\zia) \| \omega_1 - \omega_2\|_\infty \,. 
\end{aligned}
\end{equation}
Hence since $\omega_1, \omega_2 \in \Omega_+(\gamma)$, the formula \eqref{formula delta 12 blocchi eq omologica} and the estimate \eqref{stima lip op sec mel} imply that for any $\ell \in \Z^\infty_*$, $j \neq j'$, $|\ell|_\zia \leq N$
\begin{equation}\label{prima stima delta 12 cal F+}
\begin{aligned}
\| \Delta_{\omega_1 \omega_2}\big( \Pi_j \widehat{\mathcal F}_+(\ell) \Pi_{j'}\big) \|_\HS&  \lesssim \frac{{\divisor}(\ell)^2}{\gamma^2} (1 + |\ell|_\zia) \| \Pi_j \widehat{\mathcal P}(\ell; \omega_2) \Pi_{j'}\|_\HS \\
& \quad + \frac{{\divisor}(\ell)}{\gamma} \| \Delta_{\omega_1 \omega_2}\big( \Pi_j \widehat{\mathcal P}(\ell) \Pi_{j'} \big) \|_\HS
\end{aligned}
\end{equation}
and for any $\ell \in \Z^\infty_* \setminus \{ 0 \}$, $j \in \N_0$, $|\ell|_\zia \leq N$, 
\begin{equation}\label{seconda stima delta 12 cal F+}
\begin{aligned}
\| \Delta_{\omega_1 \omega_2}\big( \Pi_j \widehat{\mathcal F}_+(\ell) \Pi_{j}\big) \|_\HS&  \lesssim \frac{{\divisor}(\ell)^2 \langle j \rangle^4}{\gamma^2} (1 + |\ell|_\zia) \| \Pi_j \widehat{\mathcal P}(\ell; \omega_2) \Pi_{j}\|_\HS \|\omega_1 - \omega_2 \|_\infty  \\
& \quad + \frac{{\divisor}(\ell) \langle j \rangle^2 }{\gamma} \| \Delta_{\omega_1 \omega_2}\big( \Pi_j \widehat{\mathcal P}(\ell) \Pi_{j} \big) \|_\HS\,. 
\end{aligned}
\end{equation}
Recalling the definition \eqref{definizione classe cal B sigma} and using the estimates \eqref{prima stima delta 12 cal F+}, \eqref{seconda stima delta 12 cal F+}, one obtains that  
\begin{equation}\label{stima delta 12 F+ ell}
\begin{aligned}
\| \Delta_{\omega_1 \omega_2} \widehat{\mathcal F}_+(\ell) \|_{{\mathcal B}^{\sigma - \rho, 2}} & \lesssim \frac{{\divisor}(\ell)^2 }{\gamma^2} (1 + |\ell|_\zia) \| \widehat{\mathcal P}(\ell; \omega_2) \|_{{\mathcal B}^{\sigma, - 2}} \|\omega_1 - \omega_2 \|_\infty \\
& \quad + \frac{{\divisor}(\ell)}{\gamma} \| \Delta_{\omega_1 \omega_2}\widehat{\mathcal P}(\ell) \|_{{\mathcal B}^{\sigma}}\,. 
\end{aligned}
\end{equation}
Hence, recalling the definition \eqref{definizione norma del decay}, one gets 
\begin{equation}\label{stima delta 12 cal F+}
\begin{aligned}
& |\Delta_{\omega_1\omega_2} {\mathcal F}_+|_{\sigma - \rho, 2}  \lesssim \gamma^{- 2}\Big( \sup_{\ell \in \Z^\infty_*} {\divisor}(\ell)^2 e^{- \rho |\ell|_\zia} ( 1 + |\ell|_\zia) \Big)\| \omega_1 - \omega_2 \|_\infty \sup_{\omega \in \Omega} | {\mathcal P}(\omega)|_{\sigma, - 2}  \\
& \quad + \gamma^{- 1} \Big( \sup_{\ell \in \Z^\infty_*} {\divisor}(\ell) e^{- \rho |\ell|_\zia} \Big) | \Delta_{\omega_1 \omega_2}{\mathcal P} |_{\sigma} \\
& \stackrel{Lemma \,\ref{bound per stima di Cauchy}}{\lesssim} \gamma^{- 2} {\rm exp}\Big( \frac{\tau}{\rho^{\frac{1}{\zia}}} \ln\Big(\frac{\tau}{\rho} \Big) \Big) \Big(\| \omega_1 - \omega_2 \|_\infty \sup_{\omega \in \Omega} | {\mathcal P}(\omega)|_{\sigma, - 2}  + \gamma | \Delta_{\omega_1 \omega_2}{\mathcal P} |_{\sigma}\Big)
\end{aligned}
\end{equation}
for some $\tau > 0$. The bounds \eqref{stima sup cal F+}, \eqref{stima delta 12 cal F+}, together with the definition \eqref{definizione norma lip gamma matrici} imply the claimed bound. 
\end{proof}
By the formula \eqref{prima espansione cal L n + 1} and using that the operator ${\mathcal F}_{n + 1}$ solves the homological equation \eqref{equazione omologica riducibilita}, one obtains that 
\begin{equation}\label{formula finale cal L n + 1}
\begin{aligned}
 {\mathcal L}_{n + 1}(\f) & := \ii {\mathcal D}_{n + 1} + {\mathcal P}_{n + 1}(\f)\,, \\
 {\mathcal D}_{n + 1} & := {\mathcal D}_n + {\mathcal Z}_n, \quad {\mathcal Z}_n : = \frac{1}{\ii} [\widehat{\mathcal P}_n(0)]\,, \\
{\mathcal P}_{n + 1} & := \Pi_n^\bot {\mathcal P}_n  + \int_0^1 (1 - \tau) e^{- \tau {\mathcal F}_{n + 1}} [[\widehat{\mathcal P}_n(0)] - \Pi_n {\mathcal P}_n, {\mathcal F}_{n + 1}] e^{ \tau {\mathcal F}_{n + 1}}\, d \tau \\
& \quad +  \int_0^1 e^{-  \tau {\mathcal F}_{n + 1}}[{\mathcal P}_n, {\mathcal F}_{n + 1}] e^{ \tau {\mathcal F}_{n + 1}}\, d \tau \,.
\end{aligned}
\end{equation}
{\bf The new block-diagonal part ${\mathcal D}_{n + 1}$.} Since by the inductive hypothesis the operator ${\mathcal P}_{n}(\f)$ is skew self-adjoint, then also the $2 \times 2$ block-diagonal operator $[\widehat{\mathcal P}_n(0)] = {\rm diag}_{j \in \N_0} \Pi_j \widehat{\mathcal P}_n(0) \Pi_j$ is skew self-adjoint, therefore the $2 \times 2$ block diagonal operator ${\mathcal Z}_n := \frac{1}{\ii} [\widehat{\mathcal P}_n(0)]$ is self-adjoint. Hence using the induction hypothesis, one gets that ${\mathcal D}_{n + 1}$ is a $2 \times 2$ self-adjoint block diagonal operator. 
We then set 
\begin{equation}\label{blocchi cal D n+1 (j)}
{\mathcal D}_{n + 1}(j) := \Pi_j {\mathcal D}_{n + 1} \Pi_j := {\mathcal D}_n(j) + {\mathcal Z}_n(j), \quad {\mathcal Z}_n(j) := \Pi_j {\mathcal Z}_n \Pi_j, \quad \forall j \in \N_0\,.
\end{equation}
By the inductive estimate \eqref{stima cal P sigma n}, one gets that for any $\sigma \leq \sigma_n$
\begin{equation}\label{stima nuova parte diagonale a}
\begin{aligned}
& |{\mathcal Z}_n|_{\sigma, - 2}^\Lipn{n} = |{\mathcal D}_{n + 1} - {\mathcal D}_n|_{\sigma, - 2}^\Lipn{n} \leq |{\mathcal P}_n|_{\sigma_n, - 2}^\Lipn{n} \lesssim \e e^{- \chi^n}\,. 
\end{aligned}
\end{equation}
The latter estimate, implies that
\begin{equation}\label{stima nuova parte diagonale b}
\begin{aligned}
& \sup_{\omega \in \Omega_n(\gamma)}  \|{\mathcal Z}_n(j; \omega) \|_\HS \lesssim \e e^{- \chi^n} \langle j \rangle^{- 2}\,, \\
& \sup_{\begin{subarray}{c}
\omega_1, \omega_2 \in \Omega_n(\gamma) \\
\omega_1 \neq \omega_2
\end{subarray}} \frac{\|{\mathcal Z}_n(j; \omega_1) - {\mathcal Z}_n(j; \omega_2) \|_\HS}{\| \omega_1 - \omega_2 \|_\infty} \lesssim \e \gamma^{- 1} e^{- \chi^n}
\end{aligned}
\end{equation} 
uniformly w.r. to $j \in \N_0$. The estimate \eqref{def cal O n ell j j'} at the step $n + 1$ then follows by applying \eqref{stima nuova parte diagonale a}, using a telescoping argument. 

\medskip

\noindent
{\bf The new remainder ${\mathcal P}_{n + 1}$.} By applying Lemma \ref{proprieta norma sigma riducibilita vphi x}-$(ii)$, one obtains the estimates 
\begin{equation}\label{stima P n + 1 induttiva a}
\begin{aligned}
& |\Pi_n^\bot {\mathcal P}_n|_{\sigma_{n + 1}, - 2}^\Lipn{n} \leq e^{- N_n (\sigma_n - \sigma_{n + 1})} |{\mathcal P}_n|_{\sigma_n, - 2}^\Lipn{n}\,.
\end{aligned}
\end{equation}
Furthermore, by applying iteratively Lemma \ref{proprieta norma sigma riducibilita vphi x}-$(i)$, $(iii)$  one obtains that if $\rho > 0$ satisfies $\sigma_{n + 1} + 3 \rho < \sigma_n$
\begin{equation}\label{stima P n + 1 induttiva b}
\begin{aligned}
& \Big| e^{-  \tau {\mathcal F}_{n + 1}}[{\mathcal P}_n, {\mathcal F}_{n + 1}] e^{ \tau {\mathcal F}_{n + 1}} \Big|_{\sigma_{n + 1}, - 2}^\Lipn{n+1} + \Big| e^{- \tau {\mathcal F}_{n + 1}} [[\widehat{\mathcal P}_n(0)] - \Pi_n {\mathcal P}_n, {\mathcal F}_{n + 1}] e^{ \tau {\mathcal F}_{n + 1}} \Big|_{\sigma_{n + 1}, - 2}^\Lipn{n+1} \\
& \quad  \lesssim \rho^{- \mathtt a} \Big( \sup_{\tau \in [0, 1]} |e^{\pm \tau {\mathcal F}_{n + 1}}|_{\sigma_{n + 1} + 3 \rho}^\Lipn{n+1} \Big) |{\mathcal P}_n|_{\sigma_n, - 2}^\Lipn{n} |{\mathcal F}_{n + 1}|_{\sigma_{n + 1} + 2 \rho}^\Lipn{n+1} \,.  
\end{aligned}
\end{equation}
for some constant $\mathtt a > 0$. 

\noindent
Now we want to use Lemma \ref{lemma mappa esponenziale} in order to estimate $ \sup_{\tau \in [0, 1]} |e^{\pm \tau {\mathcal F}_{n + 1}}|_{\sigma_{n + 1} + 3 \rho}^\Lipn{n+1}$. We fix $\rho := \frac{\sigma_n - \sigma_{n + 1}}{8}$ so that $\sigma_{n + 1} + 4 \rho = \sigma_{n + 1} + \frac{\sigma_n - \sigma_{n + 1}}{2} = \frac{\sigma_n + \sigma_{n + 1}}{2} < \sigma_n $. With this choice of $\rho$, by applying Lemma \ref{homolu} and the inductive estimate \eqref{stima cal P sigma n} on ${\mathcal P}_n$, one obtains that 
\begin{equation}\label{stima cal F n + 1 iterazione}
\begin{aligned}
|{\mathcal F}_{n + 1}|_{\frac{\sigma_n + \sigma_{n + 1}}{2}}^{\Lipn{n+1}} & = |{\mathcal F}_{n + 1}|_{\sigma_{n + 1} + 4 \rho}^{\Lipn{n+1}}  \lesssim \g^{-1}{\rm exp}\Big(\frac{\tau}{(\sigma_n - \sigma_{n + 1})^{\frac{1}{\zia}}} \ln\Big(\frac{\tau}{\sigma_n - \sigma_{n + 1}} \Big) \Big)|{\mathcal P}_n|_{\s_n,-2}^{\Lipn{n}} \\
& \lesssim \e \g^{-1}{\rm exp}\Big(\frac{\tau}{(\sigma_n - \sigma_{n + 1})^{\frac{1}{\zia}}} \ln\Big(\frac{\tau}{\sigma_n - \sigma_{n + 1}} \Big) - \chi^n \Big)  \\
& \lesssim \e \gamma^{- 1} e^{- \frac{\chi^n}{2}}
\end{aligned}
\end{equation}
using that, by \eqref{costanti convergenza riducibilita}, one has 
$$
\sup_{n \in \N}\Big\{ {\rm exp}\Big(\frac{\tau}{(\sigma_n - \sigma_{n + 1})^{\frac{1}{\zia}}} \ln\Big(\frac{\tau}{\sigma_n - \sigma_{n + 1}} \Big) - \frac{\chi^n}{2} \Big) \Big\} < \infty\,. 
$$
The estimate \eqref{stima cal F n + 1 iterazione} proves the estimate \eqref{stima trasformazione cal F n - 1} at the step $n + 1$. Furthermore, \begin{equation}\label{piccolezza cal F n + 1}
\frac{1}{(\sigma_n - \sigma_{n + 1})^2} |{\mathcal F}_{n + 1}|_{\frac{\sigma_n + \sigma_{n + 1}}{2}}^\Lipn{n+1} \leq \delta
\end{equation}
for some $\delta \in (0, 1)$ small enough by taking $\e \gamma^{- 1}$ small enough and using that by \eqref{costanti convergenza riducibilita}
$$
\lim_{n \to \infty}  \frac{1}{(\sigma_n - \sigma_{n + 1})^2} e^{- \frac{\chi^n}{2}} = 0\,. 
$$
The smallness condition \eqref{smallness lemma exp} of Lemma \ref{lemma mappa esponenziale} is verified and therefore we get the estimate 
\begin{equation}\label{stima exp F n + 1 tau iterazione}
\sup_{\tau \in [0, 1]} |e^{\pm \tau {\mathcal F}_{n + 1}}|_{\sigma_n + 3 \rho}^\Lipn{n+1} \lesssim 1\,. 
\end{equation}
The estimates \eqref{stima P n + 1 induttiva a}, \eqref{stima P n + 1 induttiva b}, \eqref{stima cal F n + 1 iterazione}, \eqref{stima exp F n + 1 tau iterazione} (recalling the definition of the remainder ${\mathcal P}_{n + 1}$ given in \eqref{formula finale cal L n + 1}) lead to the inductive estimate
\begin{equation}\label{stima induttiva cal P n + 1 cal Pn}
\begin{aligned}
& |{\mathcal P}_{n + 1}|_{\sigma_{n + 1}, - 2}^\Lipn{n+1} \leq e^{- N_n (\sigma_n - \sigma_{n + 1})} |{\mathcal P}_n|_{\sigma_n, - 2}^\Lipn{n}  \\
& \quad + C \g^{-1} \frac{1}{(\sigma_n - \sigma_{n + 1})^{\frak a}}{\rm exp}\Big(\frac{\tau}{(\sigma_n - \sigma_{n + 1})^{\frac{1}{\zia}}} \ln\Big(\frac{\tau}{\sigma_n - \sigma_{n + 1}} \Big) \Big) (|{\mathcal P}_n|_{\s_n,-2}^{\Lipn{n}})^2\, 
\end{aligned}
\end{equation}
where $C > 0$ is a positive constant and $\frak a > 0$ is the constant appearing in the estimate \eqref{stima P n + 1 induttiva b}. The latter estuimate, together with the inductive estimate \eqref{stima cal P sigma n} on $|{\mathcal P}_n|_{\sigma_n, - 2}^\Lipn{n}$ imply that 
\begin{equation}\label{stima induttiva cal P n + 1 cal Pn a}
\begin{aligned}
& |{\mathcal P}_{n + 1}|_{\sigma_{n + 1}, - 2}^\Lipn{n+1} \leq e^{- N_n (\sigma_n - \sigma_{n + 1})} C_* \e e^{- \chi^n } \\
& \quad + C \g^{-1} \frac{1}{(\sigma_n - \sigma_{n + 1})^{\frak a}}{\rm exp}\Big(\frac{\tau}{(\sigma_n - \sigma_{n + 1})^{\frac{1}{\zia}}} \ln\Big(\frac{\tau}{\sigma_n - \sigma_{n + 1}} \Big) \Big) C_*^2 \e^2 e^{- 2 \chi^n}\,  \\
& \leq C_* \e e^{- \chi^{n + 1}}
\end{aligned}
\end{equation}
provided 
\begin{equation}\label{condizioni costanti passo induttivo}
\begin{aligned}
& \sup_{n \in \N} \Big\{{\rm exp}\Big( \chi^n (\chi - 1) - N_n (\sigma_n - \sigma_{n + 1}) \Big) \Big\} \leq \frac12 \,, \\
& C C_* \e \gamma^{- 1} \sup_{n \in \N} \Big\{  \frac{1}{(\sigma_n - \sigma_{n + 1})^{\frak a}}{\rm exp}\Big(\frac{\tau}{(\sigma_n - \sigma_{n + 1})^{\frac{1}{\zia}}} \ln\Big(\frac{\tau}{\sigma_n - \sigma_{n + 1}} \Big) - (2 - \chi) \chi^n \Big)  \Big\} \leq \frac12\,. 
\end{aligned}
\end{equation}
The first condition in \eqref{condizioni costanti passo induttivo} holds by recalling \eqref{costanti convergenza riducibilita} and by taking $N_0>0$ large enough. The second condition in \eqref{condizioni costanti passo induttivo} holds by recalling \eqref{costanti convergenza riducibilita} and by taking $\e \gamma^{- 1}$ small enough. 

\medskip

\noindent
{\sc Proof  of ${\bf ({S}2)}_{n + 1}$.} By recalling the estimate \eqref{stima nuova parte diagonale b}, for any $j \in \N_0$, the function $\Omega_{n + 1}(\gamma) \to {\mathcal S}({\bf E}_j)$, $\omega \mapsto {\mathcal Z}_n(j; \omega ) = {\mathcal D}_{n + 1}(j; \omega) - {\mathcal D}_n(j; \omega)$ is a Lipschitz function. Hence by using the Kirszbraun Theorem there exists a Lipschitz extension $\widetilde{\mathcal Z}_n(j; \cdot) : \Dc\to {\mathcal S}({\bf E}_j)$ of ${\mathcal Z}_n(j)$ preserving the sup norm and the Lipschitz seminorm, namely $\sup_{\omega \in \Dc}\| \widetilde{\mathcal Z}_n(j; \omega) \|_\HS \lesssim \sup_{\omega \in \Omega_{n + 1}(\gamma)} \| {\mathcal Z}_n(j; \omega)\|_\HS$, $\| \widetilde{\mathcal Z}_n(j)\|^{\rm lip}_\HS \lesssim \| {\mathcal Z}_n(j)\|^{\rm lip}_\HS$. Therefore, using the bounds \eqref{stima nuova parte diagonale b} and defining $\widetilde{\mathcal D}_{n + 1}(j) := \widetilde{\mathcal D}_n(j) + \widetilde{\mathcal Z}_n(j)$, the claimed statement follows. 

\subsection{Convergence}
{\bf Final blocks.} By applying Theorem \ref{teorema di riducibilita}-$\bf (S2)_n$ the sequence of the Lipschitz functions $\widetilde{\mathcal D}_n(j; \cdot) : \Dc \to {\mathcal S}({\bf E}_j)$, $j \in \N_0$ is a Cauchy sequence w.r. to the norm $\| \cdot \|^\Lipn{0}$ and therefore, we can define the {\it final blocks} 
\begin{equation}\label{definizione blocchi finali}
{\mathcal D}_\infty(j) := \lim_{n \to \infty} \widetilde{\mathcal D}_n(j), \quad \forall j \in \N_0\,. 
\end{equation}
By using a telescoping argument one obtains that for any $j \in \N_0$, for any $n \in \N$, the following estimates hold
\begin{equation}\label{stima blocco parziale blocco finale}
\begin{aligned}
&\sup_{\omega \in \Dc}\| {\mathcal D}_\infty(j; \omega) - \widetilde{\mathcal D}_n(j; \omega) \|_\HS \lesssim \langle j \rangle^{- 2}\e e^{- \chi^{n}}  \,, \\
& \| {\mathcal D}_\infty(j) - \widetilde{\mathcal D}_n(j) \|^{\rm lip}_\HS  \lesssim  \e \gamma^{- 1} e^{- \chi^{n}}
\end{aligned}
\end{equation}
Then, recalling the definition of the norm $|\cdot|^\Lipg_{\sigma, m}$ given in \eqref{definizione norma lip gamma matrici}, if we define the $2 \times 2$ block diagonal operators 
\begin{equation}\label{def operatore cal D infty}
\begin{aligned}
\widetilde{\mathcal D}_n := {\rm diag}_{j \in \N_0} \widetilde{\mathcal D}_n(j), \quad \forall n \in \N\,, \quad {\mathcal D}_\infty :=  {\rm diag}_{j \in \N_0} {\mathcal D}_\infty(j)
\end{aligned}
\end{equation}
one gets that for any $\sigma > 0$, $n \in \N$ and $\Omega\in\Dc$
\begin{equation}\label{stima D infty - Dn}
|{\mathcal D}_\infty - \widetilde{\mathcal D}_n|_{\sigma, - 2}^\Lipg \lesssim \e e^{- \chi^n}\,. 
\end{equation}
{\bf Final Cantor set.} For any $\ell \in \Z^\infty_*$, $j, j' \in \N_0$, we define the linear operator ${\mathcal O}_\infty(\ell, j, j') : {\mathcal B}({\bf E}_{j'}, {\bf E}_j) \to {\mathcal B}({\bf E}_{j'}, {\bf E}_j)$
\begin{equation}
{\mathcal O}_\infty(\ell, j, j') := \omega \cdot \ell \, {\rm Id} - M_L({\mathcal D}_\infty(j)) + M_R({\mathcal D}_\infty(j'))
\end{equation}
and we define the set 
\begin{equation}\label{definizione Omega infinito}
\begin{aligned}
\Omega_\infty(\gamma) &:= \Big\{ \omega \in \Dc : \| {\mathcal O}_{\infty}(\ell, j, j')^{- 1}\|_\Op \leq \frac{{\divisor}(\ell)}{2 \gamma}, \quad \forall (\ell, j, j') \in \Z^\infty_* \times \N_0 \times \N_0, \\
 & j \neq j' \quad\text{and} \quad \| {\mathcal O}_{\infty}(\ell, j, j)^{- 1}\|_{\rm Op} \leq \frac{{\divisor}(\ell) j^2}{2 \gamma} \quad \forall (\ell, j) \in (\Z^\infty_* \setminus \{ 0 \}) \times \N_0 \Big\}\,.
\end{aligned}
\end{equation}
The following lemma holds
\begin{lemma}\label{lemma inclusione cantor}
	One has that 
	\[
	\Omega_\infty(\gamma) \subseteq \cap_{n\in\N_0}\Omega_n(\gamma) \,.
	\]
	\end{lemma}
\begin{proof}
We proceed by induction. By definition $	\Omega_\infty(\gamma) \subseteq \Dc$. Now assume that $\Omega_\infty(\gamma) \subseteq \Omega_n(\gamma)$ for some $n \geq 0$ and let us show that $\Omega_\infty(\gamma) \subseteq \Omega_{n + 1}(\gamma)$. Let $\omega\in \Omega_\infty(\gamma)$. Since by the induction hypothesis $\omega \in \Omega_n(\gamma)$, the $2 \times 2$ blocks ${\mathcal D}_n(j; \omega)$, $j \in \N_0$, are well defined and ${\mathcal D}_n(j; \omega)= \widetilde{\mathcal D}_n(j; \omega)$ on such set. By the estimates \eqref{stima blocco parziale blocco finale}, recalling the property \eqref{norma operatoriale ML MR}, one obtains that 
$$
\|M_L\big({\mathcal D}_\infty(j) - {\mathcal D}_n(j) \big))   \|_{{\rm Op}}\,,\,\|M_R\big({\mathcal D}_\infty(j) - {\mathcal D}_n(j) \big))   \|_{{\rm Op}} \lesssim \e \langle j \rangle^{- 2} e^{- \chi^n}
$$
and using that 
$$
{\mathcal O}_n(\ell, j, j') - {\mathcal O}_\infty(\ell , j, j')  = - M_L({\mathcal D}_n(j) - {\mathcal D}_\infty(j)) + M_R({\mathcal D}_n(j') - {\mathcal D}_\infty(j')),
$$
the latter estimate implies that for any $\ell \in \Z^\infty_*$, $|\ell|_\zia \leq N_n$, $j, j' \in \N_0$, $j \neq j'$
\begin{equation}\label{ML MR D infty Dn}
\| {\mathcal O}_n(\ell, j, j') - {\mathcal O}_\infty(\ell , j, j') \|_{{\rm Op}} \lesssim \e e^{- \chi^n} 
\end{equation}
and for any $\ell \in \Z^\infty_* \setminus \{ 0 \}$, $|\ell|_\zia \leq N_n$, $j \in \N_0$
\begin{equation}\label{ML MR D infty Dn a}
\| {\mathcal O}_n(\ell, j, j) - {\mathcal O}_\infty(\ell , j, j) \|_{{\rm Op}} \lesssim \e e^{- \chi^n} \langle j \rangle^{- 2}\,.  
\end{equation}
Since $\omega \in \Omega_\infty(\gamma) \subseteq \Omega_n(\gamma)$, we can write 
$$
\begin{aligned}
{\mathcal O}_n(\ell, j, j') & = {\mathcal O}_\infty(\ell , j, j') + {\mathcal O}_n(\ell, j, j') - {\mathcal O}_\infty(\ell , j, j')  \\
& = {\mathcal O}_\infty(\ell , j, j')\Big({\rm Id} + {\mathcal O}_\infty(\ell , j, j')^{- 1}\Big[ {\mathcal O}_n(\ell, j, j') - {\mathcal O}_\infty(\ell , j, j') \Big] \Big)
\end{aligned}
$$
and using the estimates \eqref{ML MR D infty Dn}, \eqref{ML MR D infty Dn a}, we get for any $(\ell, j, j') \neq (0, j, j)$, $|\ell|_\zia \leq N_n$, the bound
\begin{equation}\label{bound inclusione cantor neumann}
\begin{aligned}
& \|{\mathcal O}_\infty(\ell , j, j')^{- 1}\Big[ {\mathcal O}_n(\ell, j, j') - {\mathcal O}_\infty(\ell , j, j') \Big] \|_{{\rm Op}} \lesssim  \e \gamma^{- 1} e^{- \chi^n} \sup_{|\ell|_\zia \leq N_n} {\divisor}(\ell) \\
& \stackrel{Lemma \ref{small divisor con taglio}}{\lesssim} \e \gamma^{- 1} e^{- \chi^n} (1+N_n)^{C(\zia,\mu)N_n^{\frac{1}{1+\zia}}} \\
& \lesssim \e \gamma^{- 1} \sup_{n \in \N} {\rm exp}\Big(- \chi^n + C(\zia) N_n^{\frac{1}{1 + \zia}} \ln(1 + N_n) \Big)\,. 
\end{aligned}
\end{equation}
By the choice of $N_n$ provided in \eqref{costanti convergenza riducibilita}, one obtains that 
$$
\sup_{n \in \N} {\rm exp}\Big(- \chi^n + C(\zia,\mu) N_n^{\frac{1}{1 + \zia}} \ln(1 + N_n) \Big) < \infty
$$
implying that for $\e \gamma^{- 1}$ small enough
$$
\|{\mathcal O}_\infty(\ell , j, j')^{- 1}\Big[ {\mathcal O}_n(\ell, j, j') - {\mathcal O}_\infty(\ell , j, j') \Big] \|_{{\rm Op}} \leq \frac12\,.
$$
Hence by Neumann series ${\mathcal O}_n(\ell, j, j')$ is invertible and $\omega \in \Omega_{n + 1}(\gamma)$. 
\end{proof}
{\bf KAM transformations} 

\noindent
For every $n \geq 1$, we define the transformation $\Psi_n$ as 
\begin{equation}\label{def trasformazioni KAM}
\Psi_n := \Phi_1 \circ \ldots \circ \Phi_n\,. 
\end{equation}
where for any $n \geq 1$, the transformation $\Phi_n = {\rm exp}({\mathcal F}_n)$ is constructed in Theorem \ref{teorema di riducibilita}. Note that for any $n \in \N$, the map $\Psi_n$ is invertible and the inverse is given by 
\begin{equation}\label{trasformazioni KAM inverso}
\Psi_n^{- 1} := \Phi_n^{- 1} \circ \ldots \circ \Phi_1^{- 1}\,. 
\end{equation}
We now show the convergence of the sequence of transformations $(\Psi_n)_{n \in \N}$, in the space ${\mathcal H}\big(\T^\infty_{\frac{\sigma_0}{2}}, {\mathcal B}^{\frac{\sigma_0}{2}} \big)$. 
\begin{lemma}\label{convergenza trasformazioni KAM}
$(i)$ The sequence of transformation $(\Psi_n)_{n \in \N}$ converges to an invertible transformations $\Psi_\infty$, for $\omega \in \Omega_\infty(\gamma)$ w.r. to the norm $|\cdot |_{\frac{\sigma_0}{2}}^{{\rm Lip}(\gamma, \Omega_\infty(\gamma))}$. Furthermore the following bounds hold: ,
$$
|\Psi_\infty - {\rm Id}|_{\frac{\sigma_0}{2}}^{{\rm Lip}(\gamma, \Omega_\infty(\gamma))}\,,\, |\Psi_\infty^{- 1} - {\rm Id}|_{\frac{\sigma_0}{2}}^{{\rm Lip}(\gamma, \Omega_\infty(\gamma))} \lesssim \e \gamma^{- 1}\,. 
$$
$(ii)$ For any $0 < \sigma \leq \frac{\sigma_0}{2}$, for any $s \geq 0$, the maps $\T^\infty_\sigma \to {\mathcal B}({\mathcal H}(\T_\sigma), {\mathcal H}(\T_\sigma))$, $\f \mapsto \Psi_\infty(\f)^{\pm 1}$ and $\T^\infty \to {\mathcal B}(H^s(\T), H^s(\T))$, $\f \mapsto \Psi_\infty(\f)^{\pm 1}$ are bounded. 
\end{lemma}
\begin{proof}
{\sc Proof of $(i)$.}
For any $n \geq 1$, one has that
$$
\Psi_{n + 1} = \Psi_n \circ \Phi_{n + 1} \Longrightarrow \Psi_{n + 1} - \Psi_n = \Psi_n \circ (\Phi_{n + 1} - {\rm Id})\,. 
$$
We estimate now $|\Psi_{n + 1} - \Psi_n|_{\sigma_{n + 1}}^{{\rm Lip}(\gamma, \Omega_\infty(\gamma))}$. Fix $\rho := \frac{\sigma_n - \sigma_{n + 1}}{4}$ such that $\sigma_{n + 1} < \sigma_{n + 1} + 2 \rho = \frac{\sigma_n + \sigma_{n + 1}}{2}$. By applying Lemma \ref{proprieta norma sigma riducibilita vphi x}-$(i)$, one has that 
\begin{equation}\label{stima Psi n + 1 - Psi n}
|\Psi_{n + 1} - \Psi_n|_{\sigma_{n + 1}}^{{\rm Lip}(\gamma, \Omega_\infty(\gamma))} \lesssim \rho^{- 2} |\Psi_n|_{\sigma_{n + 1}}^{{\rm Lip}(\gamma, \Omega_\infty(\gamma))} |\Phi_{n + 1} - {\rm Id}|_{\sigma_{n + 1} + \rho}^{{\rm Lip}(\gamma, \Omega_\infty(\gamma))}\,. 
\end{equation}
Moreover, since $\Phi_{n + 1} = {\rm exp}({\mathcal F}_{n + 1})$, using the estimate \eqref{stima trasformazione cal F n - 1} (at the step $n + 1$) and by applying Lemma \ref{lemma mappa esponenziale} one gets 
\begin{equation}\label{stima Phi n + 1 - Id conv}
\begin{aligned}
|\Phi_{n + 1} - {\rm Id}|_{\sigma_{n + 1} + \rho}^{{\rm Lip}(\gamma, \Omega_\infty(\gamma))} & \lesssim \rho^{- 2} |{\mathcal F}_{n + 1}|_{\sigma_{n + 1} + 2 \rho} \stackrel{\sigma_{n + 1} + 2 \rho = \frac{\sigma_n + \sigma_{n + 1}}{2}}{\lesssim } (\sigma_n - \sigma_{n + 1})^{- 2} |{\mathcal F}_{n + 1}|_{\frac{\sigma_n + \sigma_{n + 1}}{2} }^{{\rm Lip}(\gamma, \Omega_\infty(\gamma))} \\
& \lesssim  (\sigma_n - \sigma_{n + 1})^{- 2} \e \gamma^{- 1} e^{- \frac{\chi^n}{2}} \lesssim \e \gamma^{- 1} 
\end{aligned}
\end{equation}
Thus, the estimates \eqref{stima Psi n + 1 - Psi n}, \eqref{stima Phi n + 1 - Id conv} imply that 
\begin{equation}\label{stima Psi n + 1 - Psi n a}
\begin{aligned}
|\Psi_{n + 1} - \Psi_n|_{\sigma_{n + 1}}^{{\rm Lip}(\gamma, \Omega_\infty(\gamma))} &\lesssim (\sigma_n - \sigma_{n + 1})^{- 4} \e \gamma^{- 1} e^{- \frac{\chi^n}{2}} |\Psi_n|_{\sigma_{n + 1}}^{{\rm Lip}(\gamma, \Omega_\infty(\gamma))} \\
& \lesssim \e \gamma^{- 1} e^{- \frac{\chi^n}{3}} |\Psi_n|_{\sigma_{n}}^{{\rm Lip}(\gamma, \Omega_\infty(\gamma))}
\end{aligned}
\end{equation}
where in the last inequality we have used that $\sigma_{n + 1} < \sigma_n$ and  
$$
\sup_{n \in \N} \Big\{ (\sigma_n - \sigma_{n + 1})^{- 4} e^{- \frac{\chi^n}{2} + \frac{\chi^n}{3}}\Big\} < \infty
$$
and by triangular inequality 
\begin{equation}\label{stima Psi n + 1 - Psi n b}
|\Psi_{n + 1}|_{\sigma_{n + 1}}^{{\rm Lip}(\gamma, \Omega_\infty(\gamma))} \leq |\Psi_{n}|_{\sigma_n}^{{\rm Lip}(\gamma, \Omega_\infty(\gamma))} \Big( 1 + C \e \gamma^{- 1} e^{- \frac{\chi^n}{3}} \Big)
\end{equation}
for some constant $C > 0$. By iterating the latter bound one obtains that 
\begin{equation}\label{stima Psi n + 1 - Psi n c}
|\Psi_{n + 1}|_{\sigma_{n + 1}}^{{\rm Lip}(\gamma, \Omega_\infty(\gamma))} \leq \prod_{j = 0}^n \Big(1 + C  \e \gamma^{- 1} e^{- \frac{\chi^j}{3}}  \Big)\,.
\end{equation}
Passing to the logarithm in the above inequality and using that the series $\sum_{j \geq 0} e^{- \frac{\chi^j}{3}}  $ is convergent, one obtains that 
\begin{equation}\label{stima Psi n + 1 - Psi n d}
C_0 := \sup_{n \in \N}|\Psi_n|^{{\rm Lip}(\gamma, \Omega_\infty(\gamma))}_{\sigma_n} < \infty\,. 
\end{equation}
Now let $n, k \geq 1$. One has that 
\begin{equation}
\begin{aligned}
|\Psi_{n + k} - \Psi_n|_{\frac{\sigma_0}{2}}^{{\rm Lip}(\gamma, \Omega_\infty(\gamma))} & \leq \sum_{j \geq n} |\Psi_{j + 1} - \Psi_j|_{\sigma_{j + 1}}^{{\rm Lip}(\gamma, \Omega_\infty(\gamma))} \stackrel{\eqref{stima Psi n + 1 - Psi n a}, \eqref{stima Psi n + 1 - Psi n d}}{\lesssim} \e \gamma^{- 1} \sum_{j \geq n} e^{- \frac{\chi^j}{3}} \lesssim \e \gamma^{- 1} e^{- \frac{\chi^n}{3}}\,. 
\end{aligned}
\end{equation}
Hence $(\Psi_n)_{n \in \N}$ is a Cauchy sequence w.r. to the norm $|\cdot |^{{\rm Lip}(\gamma, \Omega_\infty(\gamma))}_{\frac{\sigma_0}{2}}$ and hence it coverges to $\Psi_\infty$ with a bound 
$$
|\Psi_\infty - \Psi_n|_{\frac{\sigma_0}{2}}^{{\rm Lip}(\gamma, \Omega_\infty(\gamma))} \lesssim \e \gamma^{- 1} e^{- \frac{\chi^n}{3}}, \quad \forall n \in \N\,. 
$$
Similarly one shows that also the sequence $(\Psi_n^{- 1})_{n \in \N}$ converges to a transformation $\Gamma_\infty$ w.r. to the norm $|\cdot |_{\frac{\sigma_0}{2}}^{{\rm Lip}(\gamma, \Omega_\infty(\gamma))}$ with the same rate of convergence. Furthermore since $\Psi_n \Psi_n^{- 1} = \Psi_n^{- 1} \Psi_n = {\rm Id}$, passing to the limit one obtains that $\Gamma_\infty = \Psi_\infty^{- 1}$. The claimed statement has then been proved. 

\noindent
{\sc Proof of $(ii)$.} The claimed statement follows by the item $(i)$ and by applying Lemmata \ref{embedding L infty}, \ref{lemma azione nostri pseudo}.
\end{proof}

\medskip

\noindent
{\bf Final normal form}

\noindent
We now show the following 
\begin{lemma}\label{proposizione coniugazione finale}
For any $\omega \in \Omega_\infty(\gamma)$ and for any $\f \in \T^\infty_{\sigma_0/3}$, the operator ${\mathcal L}_0(\f; \omega)$ defined in \eqref{def primo op rid} is congugated to the $2 \times 2$ block diagonal operator $\ii {\mathcal D}_\infty$ (see \eqref{definizione blocchi finali}, \eqref{def operatore cal D infty}), namely $(\Psi_\infty)_{\omega*}{\mathcal L}_0(\f; \omega) = \ii {\mathcal D}_\infty(\omega)$
\end{lemma}
\begin{proof}
By applying Theorem \ref{teorema di riducibilita}, by recalling the definition \eqref{def trasformazioni KAM} of the maps $\Psi_n$, $n \in \N$ and using that by Lemma \ref{lemma inclusione cantor}, $\Omega_\infty(\gamma) \subseteq \cap_{n \geq 0} \Omega_n(\gamma)$, one gets that for any $n \in \N$
\begin{equation}\label{sugo di cinghiale}
\ii {\mathcal D}_n(\omega) + {\mathcal P}_n(\f; \omega) = {\mathcal L}_n = (\Psi_n)_{\omega*}{\mathcal L}_0(\f; \omega), \quad \forall \omega \in \Omega_\infty(\gamma)\,. 
\end{equation}
By \eqref{def primo op rid}, \eqref{cal D0 cal P0 inizio KAM} and by Lemmata \ref{convergenza trasformazioni KAM}, \ref{lemma om dot partial vphi}, one has 
\begin{equation}\label{crostata di ciliegie}
|\omega \cdot \partial_\f (\Psi_\infty - \Psi_n) |_{\frac{\sigma_0}{2} - \rho}^{{\rm Lip}(\gamma, \Omega_\infty(\gamma))} \lesssim \rho^{- 1} |\Psi_\infty - \Psi_n|_{\frac{\sigma_0}{2}}^{{\rm Lip}(\gamma, \Omega_\infty(\gamma))} \to 0 \quad \text{as} \quad n \to \infty, \quad \text{and } \quad |{\mathcal L}_0|_{\sigma_0, - 2}^{{\rm Lip}(\gamma, \Omega_\infty(\gamma))} \lesssim 1
\end{equation}
for $\rho > 0$ so that $\frac{\sigma_0}{2} - \rho > 0$. Therefore, by recalling the definition \eqref{push forward}, by the estimates \eqref{crostata di ciliegie} and by applying Lemma \ref{proprieta norma sigma riducibilita vphi x}-$(i)$, one gets that 
\begin{equation}\label{convergenza fine a}
\lim_{n \to \infty} |(\Psi_n)_{\omega*}{\mathcal L}_0 - (\Psi_\infty)_{\omega*}{\mathcal L}_0|_{\frac{\sigma_0}{3}}^{{\rm Lip}(\gamma, \Omega_\infty(\gamma))} = 0\,. 
\end{equation}
By the estimates \eqref{stima cal P sigma n}, \eqref{stima D infty - Dn}, \eqref{convergenza fine a} and passing to the limit in \eqref{sugo di cinghiale} one obtains the claimed statement. 
\end{proof}
\section{Measure estimates}\label{sezione stime di misure}
It remains only to estimate the measure of the set $\Omega_\infty(\gamma)$, defined in \eqref{definizione Omega infinito}.In order to do this, let us start with some preliminary considerations. For any $j \in \N_0$, the $2 \times 2$ block ${\mathcal D}_\infty(j; \omega)$, $\omega \in \Dc$ is self-adjoint and depends in a Lipschitz way on the parameter $\omega$. By \eqref{definizione blocchi finali}, \eqref{stima blocco parziale blocco finale} and by recalling \eqref{cal D0 cal P0 inizio KAM}, \eqref{blocco iniziale riducibilita}, for any $j \in \N$, we can write that 
\begin{equation}\label{prima espansione D infty j}
{\mathcal D}_\infty(j) = \lambda_2 j^2 {\rm Id} + R_\infty(j; \omega)
\end{equation}
where the self-adjoint $2 \times 2$ block $R_\infty(j; \omega)$ satisfies the estimate 
\begin{equation}\label{stima resto D infty j}
\begin{aligned}
&\sup_{\omega \in \mathtt D_\gamma} \| R_\infty(j; \omega) \|_\HS \lesssim \e \langle j \rangle\,, \quad  \| R_\infty(j) \|^{\rm lip}_\HS \lesssim \e \gamma^{- 1}\,.
\end{aligned}
\end{equation}
By applying Lemma \ref{risultato astratto operatori autoaggiunti}, one then obtains that for any $j \in\N$, 
$$
{\rm spec}({\mathcal D}_\infty(j; \omega)) = \{ \mu_j^{(+)}(\omega), \mu_j^{(-)}(\omega)\}, \quad {\rm spec}(R_\infty(j; \omega)) = \{ r_j^{(+)}(\omega), r_j^{(-)}(\omega)\}
$$
where $\mu_j^{(\pm)}$ and $r_j^{(\pm)}$ depend in a Lipschitz way on the parameter $\omega \in \mathtt D_\gamma$ and they satisfy 
\begin{equation}\label{proprieta autovalori D infty j}
\begin{aligned}
& \mu_j^{(\pm)}(\omega) = \lambda_2 j^2 + r_j^{(\pm)}(\omega)\,, \\
& |\lambda_2 - 1| \lesssim \e\,, \quad \sup_{\omega \in \mathtt D_\gamma} |r_j^{(\pm)}(\omega)| \lesssim \e \langle j \rangle, \quad |r_j^{(\pm)}|^{\rm lip} \lesssim \e \gamma^{- 1}\,. 
\end{aligned}
\end{equation}
If $j = 0$ one has $|\mu_0|^{{\rm Lip}(\gamma, \mathtt D_\gamma)} \lesssim \e$. For compactness of notations we set $\mu_0^{(+)} = \mu_0^{(-)} = \mu_0$. 
By applying Lemmata \ref{properties operators matrices} and \ref{risultato astratto operatori autoaggiunti}-$(ii)$ one then obtains that the set $\Omega_\infty(\gamma)$ can be written as 
\begin{equation}\label{espressione Omega infty autovalori}
\begin{aligned}
\Omega_\infty(\gamma) & = \Big\{ \omega \in \mathtt D_\gamma  : |\omega \cdot \ell + \mu_j^{(\sigma)} - \mu_{j'}^{(\sigma')}| \geq \frac{2 \gamma}{{\divisor}(\ell)}, \quad \forall (\ell, j, j') \in \Z^\infty_* \times \N_0 \times \N_0, \quad j \neq j', \quad \sigma, \sigma' \in \{ +, - \} \\
& |\omega \cdot \ell + \mu_j^{(\sigma)} - \mu_{j}^{(\sigma')}| \geq \frac{2 \gamma}{{\divisor}(\ell) \langle j \rangle^2}, \quad \forall (\ell, j) \in (\Z^\infty_* \setminus\{ 0 \}) \times \N_0 , \quad \sigma, \sigma' \in \{ +, - \} \Big\}\,,
\end{aligned}
\end{equation}
where we recall 
\[
{\divisor}(\ell) := \prod_{n\in \N}(1+|\ell_n|^{4} \jap{n}^{4}), \quad \forall \ell \in \Z^\infty_*\,. 
\]
In the remaining part of this section we prove the following Proposition. 
\begin{proposition}\label{proposizione stima misura}
Assume that $\mu> 3$. For $\e \gamma^{- 1}$ and $\gamma$ small enough one has that ${\mathbb P}\Big(\Ro \setminus \Omega_\infty(\gamma) \Big) \lesssim \gamma$. 
\end{proposition}
We  note that
\begin{equation}\label{pappa col pomodoro}
{\mathbb P}\Big(\Ro \setminus \Omega_\infty(\gamma)\Big)\leq {\mathbb P} \Big(\Ro \setminus \Dc \Big) + {\mathbb P}\Big(\Dc  \setminus \Omega_\infty(\gamma)\Big)\,.
\end{equation}
In \cite{BMP1:2018}, it is proved that 
\begin{equation}\label{pappa col pomodoro 2}
{\mathbb P}\Big(\Ro \setminus \Dc \Big)\lesssim \gamma\,,
\end{equation}
therefore, we need to estimate the set $\Dc \setminus \Omega_\infty(\gamma)$. 
In order to shorten notations, we define 
\begin{equation}\label{insiemi di indici cal Z1 Z2}
\begin{aligned}
{\mathcal Z}_1 := \Big\{ (\ell, j, j' ) \in \Z^\infty_* \times \N_0 \times \N_0 : j \neq j' \Big\}, \quad {\mathcal Z}_2 := (\Z^\infty_* \setminus \{ 0 \}) \times \N_0\,. 
\end{aligned}
\end{equation}
One has that 
\begin{equation}\label{espressione complementare}
\Dc \setminus \Omega_\infty(\gamma) = \Big(\bigcup_{(\ell, j, j') \in {\mathcal Z}_1} {\mathcal R}_{\ell j j'}(\gamma) \Big) \bigcup\Big( \bigcup_{(\ell, j) \in {\mathcal Z}_2} {\mathcal Q}_{\ell j}(\gamma) \Big)
\end{equation}
where for any $(\ell, j, j') \in {\mathcal Z}_1$, we define 
\begin{equation}\label{def cal R gamma}
{\mathcal R}_{\ell j j'}(\gamma) := \bigcup_{\sigma, \sigma' \in \{+, - \}} \Big\{ \omega \in \Dc : |\omega \cdot \ell + \mu_j^{(\sigma)} - \mu_{j'}^{(\sigma')}| < \frac{2 \gamma}{{\divisor}(\ell)} \Big\}
\end{equation}
and for any $(\ell, j) \in {\mathcal Z}_2$, we define 
\begin{equation}\label{def cal Q gamma}
{\mathcal Q}_{\ell j}(\gamma) := \bigcup_{\sigma, \sigma' \in \{ +, - \}} \Big\{ \omega \in \Dc : |\omega \cdot \ell + \mu_j^{(\sigma)} - \mu_{j}^{(\sigma')}| < \frac{2 \gamma}{{\divisor}(\ell) \langle j \rangle^2} \Big\}\,. 
\end{equation}
\begin{lemma}\label{misura singolo}
$(i)$ Let $(\ell, j, j') \in {\mathcal Z}_1$. If ${\mathcal R}_{\ell j j'}(\gamma) \neq \emptyset$, then $|j^2 - j'^2| \leq C |\ell|_1$ and ${\mathbb P}({\mathcal R}_{\ell j j'}(\gamma)) \lesssim \frac{\gamma}{{\divisor}(\ell)}$. 

\noindent
$(ii)$ Let $(\ell, j) \in {\mathcal Z}_2$. If ${\mathcal Q}_{\ell j}(\gamma) \neq \emptyset$, then ${\mathbb P}({\mathcal Q}_{\ell j}(\gamma)) \lesssim \frac{\gamma}{\langle j \rangle^2 {\divisor}(\ell)}$. 
\end{lemma}
\begin{proof}
We prove item $(i)$. The proof of the item $(ii)$ can be done arguing in a similar fashion. Let $j, j' \in \N_0$, $j \neq j'$ and $\sigma, \sigma' \in \{ +, - \}$. By \eqref{proprieta autovalori D infty j} one has that for some constant $C > 0$,
$$
|\mu_j^{(\sigma)} - \mu_{j'}^{(\sigma')}| \geq |\lambda_2| |j^2 - j'^2| - C \e (j + j') - C \e\,.
$$
Using that $\lambda_2 = 1 + O(\e)$ and that $|j + j'| \leq |j^2 - j'^2|$ one obtains that for $\e$ small enough 
\begin{equation}\label{mu j sigma mu j' sigma'}
|\mu_j^{(\sigma)} - \mu_{j'}^{(\sigma')}| \geq \frac12 |j^2 - j'^2|
\end{equation}
implying that ${\mathcal R}_{0 j j'}(\gamma) = \emptyset$ for any $j \neq j'$. Hence if $(\ell, j, j') \in {\mathcal Z}_1$ and ${\mathcal R}_{\ell j j'}(\gamma) \neq \emptyset$ one has that $\ell \neq 0$. Furthermore if $\omega \in {\mathcal R}_{\ell j j'}(\gamma) \neq \emptyset$ one has that by using \eqref{mu j sigma mu j' sigma'}, one obtains that 
\begin{equation}\label{mu j sigma mu j' sigma'2}
\frac12 |j^2 - j'^2| \leq |\mu_j^{(\sigma)} - \mu_{j'}^{(\sigma')}| \leq \frac{2 \gamma}{{\divisor}(\ell)} + |\omega \cdot \ell| \lesssim 1 + \| \omega\|_\infty \| \ell \|_1 \lesssim 1 + \|\ell \|_1\,. 
\end{equation}
Now let 
$$
s := {\rm min}\{ n \in \N : \ell_n \neq 0\}, \quad S := {\rm max}\{ n \in \N : \ell_n \neq 0 \}\,. 
$$
and ${\bf e}^{(s)} = ({\bf e}^{(s)}_n)_{n \in \N}$ the vector whose $n$-th component is $0$ if $n \neq s$ and $1$ if $n = s$. Similarly we define the vector ${\bf e}^{(S)}$. Let 
$$
\psi(t) := (\omega + t {\bf e}^{(s)})\cdot \ell + \mu_j^{\sigma}(\omega + t {\bf e}^{(s)}) - \mu_{j'}^{(\sigma')}(\omega + t {\bf e}^{(s)})\,.
$$
By using the estimate \eqref{proprieta autovalori D infty j}, for $\e \gamma^{- 1}$ small enough, one has that 
$$
|\psi(t_1) - \psi(t_2)| \geq |t_1 - t_2||\ell_s| - C \e \gamma^{- 1 } |t_1 - t_2| \geq \frac12 |t_1 - t_2|\,.
$$
The latter estimate implies that 
$$
\Big| \Big\{ t: \omega + t {\bf e}^{(s)} \in {\mathcal R}_{\ell j j'}(\gamma), \quad |\psi(t)| < \frac{2 \gamma}{{\divisor}(\ell)} \Big\} \Big| \lesssim \frac{\gamma}{{\divisor}(\ell)}\,.
$$
Since ${\mathcal R}_{\ell j j'}(\gamma)$ is a cylinder with at most $S - s$ components, one obtains the desired bound. 

\end{proof}
{\sc Proof of Proposition \ref{proposizione stima misura}.} By recalling \eqref{espressione complementare} and by applying  Lemma \ref{misura singolo}, one gets the estimate 
$$
\begin{aligned}
{\mathbb P}\big(\Dc \setminus \Omega_\infty(\gamma) \big) &\lesssim \sum_{\begin{subarray}{c}
(\ell, j , j') \in {\mathcal Z}_1 \\
|j^2 - j'^2| \leq \|\ell \|_1
\end{subarray}} \frac{\gamma}{{\divisor}(\ell)} + \sum_{\begin{subarray}{c}
(\ell, j) \in {\mathcal Z}_2 
\end{subarray}} \frac{\gamma}{\langle j \rangle^2 {\divisor}(\ell)} \\
& \lesssim \gamma \Big(\sum_{\ell \in \Z^\infty_*} \frac{\| \ell \|_1^2}{{\divisor}(\ell)} + \sum_{\ell \in \Z^\infty_*} \frac{1}{{\divisor}(\ell)} \sum_{j \in \N_0} \frac{1}{\langle j \rangle^2} \Big) \stackrel{Lemma \ref{stima serie piccoli divisori}}{\lesssim} \gamma\,. 
\end{aligned}
$$
The claimed statement then follows by recalling \eqref{pappa col pomodoro}, \eqref{pappa col pomodoro 2}. 

\section{Proof of Theorem \ref{teorema principale} and Corollary \ref{corollario sobolev}}\label{sezione dim teo finali}
Let $\gamma := \e^{a}$, $a \in (0, 1)$. Then the smallness condition $\e \gamma^{- 1} \leq \delta$ is fullfilled by taking $\e \in (0, \e_0)$ with $\e_0$ small enough. By setting $\Omega_\e := \Omega_\infty(\gamma)$, the Proposition \ref{proposizione stima misura} implies \eqref{stima misura main theorem}. For any $\omega \in \Omega_\e$, we define 
\begin{equation}
{\mathcal W}_\infty(\f) := \Phi^{(1)}(\f) \circ \Phi^{(2)} \circ \ldots \circ \Phi^{(7)}(\f) \circ \Psi_\infty(\f) \quad \f \in \T^\infty_{\overline \sigma/ 4}
\end{equation}
where the maps $\Phi^{(1)}, \ldots, \Phi^{(7)}$ are constructed in Section \ref{sezione riduzione ordine} and the map $\Psi_\infty$ is given in Lemma \ref{convergenza trasformazioni KAM}. The properties $(1)$ and $(2)$ on the maps ${\mathcal W}_\infty(\f)^{\pm 1}$ stated in Theorem \ref{teorema principale} are easily deduced from Lemmata \ref{lemma stime step 1}, \ref{lemma stime step 2}, \ref{lemma stime step 3}, \ref{lemma stime step 4}, \ref{lemma stime step 5}, \ref{lemma stime step 5a}, \ref{lemma stime step 7}, \ref{convergenza trasformazioni KAM}-$(ii)$ and from remark \ref{remark riparametrizzazione tempo}. Furthermore, by the same  Lemmata and  \ref{proposizione coniugazione finale} one obtains that $u(t, x)$ is a solution of \eqref{main equation} if and only if $v(\cdot, t) := {\mathcal W}_\infty(\omega t)^{- 1} u(\cdot, t)$, $\omega \in \Omega_\e$ solves the time independent equation $\partial_t v = \ii {\mathcal D}_\infty v$ where ${\mathcal D}_\infty$ is the $2 \times 2$ time independent self-adjoint block-diagonal operator defined in \eqref{definizione blocchi finali}-\eqref{def operatore cal D infty}. The proof of Theorem \ref{teorema principale} is then concluded. 

\noindent
{\sc Proof of Corollary \ref{corollario sobolev}.} Since ${\mathcal D}_\infty$ is a $2 \times 2$ block diagonal self-adjoint operator, the general solution of the equation $\partial_t v = \ii {\mathcal D}_\infty v$ can be written as 
$$
v(x, t) = \sum_{j \in \N_0} e^{\ii t \Pi_j {\mathcal D}_\infty \Pi_j }[\Pi_j v_0]\,. 
$$
Since $\Pi_j {\mathcal D}_\infty \Pi_j : {\bf E}_j \to {\bf E}_j$ is self-adjoint (recall \eqref{bf E alpha}), one has that $\big\| e^{\ii t \Pi_j {\mathcal D}_\infty \Pi_j }[\Pi_j v_0] \big\|_{L^2} = \| \Pi_j v_0\|_{L^2}$ for any $j \in \N_0$. This implies that both analytic and Sobolev norms are preserved, namely for any $\sigma > 0$, $\| v(\cdot , t )\|_\sigma = \| v_0\|_\sigma$ and for any $s \geq 0$, $\| v( \cdot, t) \|_{H^s} = \| v_0\|_{H^s}$. Hence, by using the properties $(1)$ and $(2)$ stated in Theorem \ref{teorema principale}, one obtains that for any $\omega \in \Omega_\e$, the solution $u(\cdot, t) := {\mathcal W}_\infty(\omega t) v( \cdot, t)$ of \eqref{main equation} satisfies the desired bounds both in analytic and Sobolev norms. The proof of the Corollary is therefore concluded. 

\appendix
\section{Holomorphic functions on the infinite dimensional torus}\label{appendice funzioni olomorfe}

We start by proving that, just as in the finite dimensional case, $\cH(\T^\infty_{\s},X)$ is a space of holomorphic  functions in the following sense.
\\
Endow the thickened torus $\T^\infty_{\s}$ with any topology  such that the restriction to a finite dimensional subtorus is a metric, i.e. any topology which is finer that the product topology. Denote by ${\mathcal B}^\sigma(X)$, the space of the bounded, continuous functions $u : \T^\infty_\sigma \to X$ equipped with the sup norm $\| \cdot \|_{{\mathcal B}^\sigma(X)}$. 
	For $N \in \N$, define the space ${\mathcal H}(\T_\sigma^N,X)$ as the space of holomorphic functions from the $N$-dimensional torus $T_\sigma^N : = \prod_{i = 1 }^N\T_{\sigma \langle i \rangle^\zia}$ with values in $X$.
	 Finally let $\widetilde{\mathcal H}(\T^\infty_\sigma, X)$ be the closure of $\cup_{N \in \N} {\mathcal H}^\sigma_N(X)$ in ${\mathcal B}^\sigma(X)$ w.r. to $\| \cdot \|_{{\mathcal B}^\sigma(X)}$.  

\begin{proposition}
	For all $\s,\rho>0$ one has ${\mathcal H}(\T^\infty_\sigma, X)\subseteq \widetilde{\mathcal H}(\T^\infty_\sigma, X) \subseteq  {\mathcal H}(\T^\infty_{\sigma+\rho}, X)$ with the bounds
	\[
	\| u \|_{\widetilde{\mathcal H}(\T^\infty_\sigma, X)} \le \| u \|_{\sigma} \lesssim {\rm exp}\Big(\frac{\tau}{\rho^{\frac{1}{\zia}}} \ln\Big(\frac{\tau}{\rho} \Big) \Big) \| u \|_{\widetilde{\mathcal H}(\T^\infty_{\s+\rho}, X)}
	\]
\end{proposition}
\begin{proof} 
Given $N \in \N$, we define the set 
$$
Z^\infty_N := \Big\{ \ell \in \Z^\N : \ell_i = 0, \quad \forall i > N \Big\}\,. 
$$
Given a function $u : \T^\infty_\sigma\to X$, for any $N \in \N$ we define the truncated function 
$$
S_N u(\f) := \sum_{\ell \in \Z^\infty_N} \widehat u(\ell) e^{\im \ell \cdot \f}\,. 
$$ 

	Let us show that $u\in \cH(\T^\infty_\s,X)$ is the limit of $S_N u $ in ${\mathcal B}^\sigma(X)$. If $\ell \in \Z^\infty \setminus \Z_N^\infty$, then there exists $|i | > N$ such that $\ell_i \neq 0$ and hence by the definition of $|\ell|_\zia$ one has $|\ell|_\zia > N^\zia$. Therefore 
	\[
	\begin{aligned}
	\sup_{\f \in \T^\infty_\sigma} \| u(\f) - S_N u(\f)\|_X & = \sup_{\f\in\T^\infty_\s} \Big\| \sum_{ \substack{\ell\in \Z^\infty_* \setminus \Z^\infty_N}}\widehat u(\ell)e^{\im (\f\cdot \ell)} \Big\|_X\le
	\sum_{ \substack{\ell \in \Z^\infty_\zia: |\ell|_\zia > N^\zia}}\| \widehat u(\ell) \|_X e^{\s |\ell|_\zia} \\
	\end{aligned}
	\]
	The right hand side of the above inequality tends to $0$ as $N \to \infty$, since it is the tail of an absolutely convergent series.
To prove the second inclusion we consider
	$u \in \widetilde{\mathcal H}(T^\infty_{\sigma + \rho},X)$. By definition there exists  a sequence $(u_k)_{k \in \N}$ with $u_k \in {\mathcal H}(\T^{N_k}_{\sigma +\rho},X)$, such that $u_k \to u$ w.r. to $\| \cdot \|_{{\mathcal B}^\sigma(X)}$. Since $u_k$ is an analytic function of the finite dimensional torus $\T^{N_k}_{\sigma + \rho}$, we can apply the Cauchy estimate, namely 
	\begin{equation}\label{bla bla}
	\|\widehat u_k(\ell) \|_X \leq e^{- (\sigma + \rho) |\ell|_\zia} \| u_k \|_{{\mathcal H}(\T^{N_k}_{\sigma +\rho},X)}\,, \quad \forall \ell \in \Z^\infty_{N_k}\,. 
	\end{equation}
	Let $\ell \in \Z^\infty_*$ with $|\ell|_\zia < \infty$, then there exists an $\overline N > 0$ such that $\ell \in \Z^\infty_{\overline N}$. Then for any $k \geq k_0$, one has $\ell \in \Z^\infty_{N_k}$. For any $k\ge m \geq k_0$ one has 
	\begin{equation}\label{stima cauchy trasformata fourier}
	\|\widehat u_k (\ell) - \widehat u_m(\ell)\|_X \leq e^{- (\sigma + \rho) |\ell|_\zia} \| u_k - u_m \|_{{\mathcal H}(\T^{N_k}_{\sigma +\rho},X)},
	\end{equation}
	implying that the sequence $(\widehat u_k(\ell))_{k \in \N}$ is a Cauchy sequence. We define 
	$$
	\widehat u(\ell) := \lim_{k \to \infty} \widehat u_k(\ell)\, 
	$$
	and passing to the limit for $k \to \infty$ in \eqref{stima cauchy trasformata fourier}, one obtains that 
	\begin{equation}\label{stima cauchy trasformata fourier 2}
	\|\widehat u_m (\ell) - \widehat u(\ell)\|_X \leq e^{- (\sigma + \rho) |\ell|_\zia} \| u_k - u  \|_{\widetilde{\mathcal H}(\T^\infty_{\sigma +\rho},X)}. 
	\end{equation}
	
	Clearly, passing to the limit in \eqref{bla bla}, one has 
	\begin{equation}\label{bla bla 2}
	\|\widehat u(\ell)\|_X \leq e^{- (\rho + \sigma)|\ell|_\zia} \| u \|_{\widetilde{\mathcal H}(\T^\infty_{\sigma +\rho},X))}\,. 
	\end{equation}
	Let $v(\f ) := \sum_{\ell \in \Z^\infty_*} \widehat u(\ell) e^{\im \ell \cdot \f} $. We  show that $u = v$ by estimating $\|u(\f) - v(\f)\|_X$ pointwise for any $\f \in \T^\infty_\sigma$. We have $\| u(\f) - v(\f) \|_X  = \lim_{k \to \infty} \|u_k(\f) - v(\f )\|_X$ and we estimate 
	$$
	\begin{aligned}
	\|u_k(\f) - v(\f )\|_X   & \leq  \|\sum_{\ell \in \Z^\infty_{N_k}} \widehat u_k(\ell) e^{\im \ell \cdot \f} - \sum_{\ell \in \Z^\infty_*} \widehat u(\ell) e^{\im \ell \cdot \f} \|_X \\
	& \leq  \sum_{\ell \in \Z^\infty_{N_k}} \|\widehat u_k(\ell) - \widehat u(\ell)\|_X + \sum_{\ell \in \Z^\infty_* \setminus \Z^\infty_{N_k}} \|\widehat u(\ell)\|_X  \\
	& \leq  \sum_{\ell \in \Z^\infty_{N_k}} \|\widehat u_k(\ell) - \widehat u(\ell)\|_X + \sum_{|\ell|_\zia \geq N_k} \|\widehat u(\ell)\|_X \\
	& \stackrel{\eqref{stima cauchy trasformata fourier}, \eqref{stima cauchy trasformata fourier 2}}{\leq} \sum_{\ell \in \Z^\infty_{N_k}} e^{- (\sigma + \rho) |\ell|_\zia} \| u_k - u  \|_{{\mathcal H}^{\sigma + \rho}(X)} + \sum_{|\ell|_\zia > N_k} e^{- (\rho + \sigma)|\ell|_\zia} \| u \|_{{\mathcal H}^{\sigma + \rho}(X)}\,.
	\end{aligned}
	$$
	The first term converges to zero since $\sum_{\ell \in \Z^\infty_*} e^{- (\sigma + \rho) |\ell|_\zia}$ is convergent and  $\| u_k - u  \|_{{\mathcal H}^{\sigma + \rho}(X)} \to 0$. The second term converges to zero since it is the tail of a convergent series and $N_k \to \infty$. It remains to estimate $\| u \|_\sigma$. We have 
	$$
	\begin{aligned}
	\| u \|_\sigma & = \sum_{\ell \in \Z^\infty_*} e^{\sigma |\ell|_\zia} \|\widehat u(\ell)\|_X \stackrel{\eqref{bla bla 2}}{\leq} \| u \|_{{\mathcal H}^{\sigma + \rho}(X)} \sum_{\ell \in \Z^\infty_*} e^{- \rho |\ell|_\zia} \,. 
	\end{aligned}
	$$
\end{proof}

\section{technical lemmata}\label{appendiceA}
\subsection{ Linear operators in finite dimension}
Given an operator $A \in {\mathcal B}({\bf E}_j)$, we define its trace as 
\begin{equation}\label{definizione traccia}
\begin{aligned}
{\rm Tr}(A) := A_0^0, \quad A \in {\mathcal B}({\bf E}_0), \\
{\rm Tr}(A) :=  A_j^j  + A_{- j}^{- j}\,, \quad A \in {\mathcal B}({\bf E}_j), \quad  j \in \N\,. 
\end{aligned}
\end{equation}
It is easy to check that if $A, B \in {\mathcal B}({\bf E}_j)$, then 
\begin{equation}\label{proprieta traccia}
{\rm Tr}(A B) = {\rm Tr}(B A)\,.
\end{equation}
For all $j, j' \in \N_0$, the space ${\mathcal B}({\bf E}_{j'}, {\bf E}_j)$ is a Hilbert space\footnote{Actually all the norms on the finite dimensional space ${\mathcal B}({\bf E}_{j'}, {\bf E}_j)$ are equivalent. } equipped by the inner product given for any $X, Y \in {\mathcal B}({\bf E}_{j'}, {\bf E}_j)$ by
\begin{equation}\label{prodotto scalare traccia matrici}
\langle X, Y \rangle := {\rm Tr}(X Y^*)\,.
\end{equation}
This scalar product induces the $L^2$-norm  $\|\cdot\|_\HS$ defined in \eqref{norma L2 blocco}.

\noindent
Given a linear operator ${\bf L} : {\mathcal B}({\bf E}_{j'}, {\bf E}_j) \to {\mathcal B}({\bf E}_{j'}, {\bf E}_j)$, we denote by $\| {\bf L}\|_{{\rm Op}}$ its operatorial norm, when the space ${\mathcal B}({\bf E}_{j'}, {\bf E}_j)$ is equipped by the $L^2$-norm \eqref{norma L2 blocco}, namely
\begin{equation}\label{norma operatoriale su matrici alpha beta}
\| {\bf L}\|_{\rm Op}:= \sup\Big\{  \| {\bf L}(M) \|_\HS : M \in {\mathcal B}({\bf E}_{j'}, {\bf E}_j)\,, \quad \| M\|_\HS \leq 1\Big\}\,.
\end{equation} 
%We denote by ${\rm Id}_{j, j'}$ the identity operator on ${\mathcal B}({\bf E}_{j'}, {\bf E}_j)$, namely 
%\begin{equation}\label{operatore identita matrici alpha beta}
%{\rm Id}_{j,  j'} : {\mathcal B}({\bf E}_{j'}, {\bf E}_j) \to {\mathcal B}({\bf E}_{j'}, {\bf E}_j)\,, \qquad X \mapsto X\,.
%\end{equation}
For any operator $A \in {\mathcal B}({\bf E}_j)$ we denote by $M_L(A) : {\mathcal B}({\bf E}_{j'}, {\bf E}_j) \to {\mathcal B}({\bf E}_{j'}, {\bf E}_j)$ the linear operator defined for any $X \in {\mathcal B}({\bf E}_{j'}, {\bf E}_j)$ as 
\begin{equation}\label{definizione moltiplicazione sinistra matrici}
M_L(A) X := A X\,.
\end{equation} 
Similarly, given an operator $B \in {\mathcal B}({\bf E}_{j'})$, we denote by $M_R(B) : {\mathcal B}({\bf E}_{j'}, {\bf E}_j) \to {\mathcal B}({\bf E}_{j'}, {\bf E}_j)$ the linear operator defined for any $X \in {\mathcal B}({\bf E}_{j'}, {\bf E}_j)$ as 
\begin{equation}\label{definizione moltiplicazione destra matrici}
M_R(B) X := X B\,.
\end{equation}
The following elementary estimates hold:
\begin{equation}\label{norma operatoriale ML MR}
\| M_L(A)\|_{{\rm Op}} \leq \| A\|_\HS\,, \quad \| M_R(B)\|_{{\rm Op}} \leq \| B\|_\HS\,.
\end{equation}
We denote by ${\mathcal S}({\bf E}_j)$, the set of the self-adjoint operators form ${\bf E}_j$ onto itself, namely
\begin{equation}\label{cal S E alpha}
{\mathcal S}({\bf E}_j) := \Big\{ A \in {\mathcal L}({\bf E}_j) : A = A^*\Big\}\,. 
\end{equation}
Furthermore, for any $A \in {\mathcal B}({\bf E}_j)$ denote by ${\rm spec}(A)$ the spectrum of $A$. The following Lemma can be proved by using elementary arguments from linear algebra, hence the proof is omitted.
\begin{lemma}\label{properties operators matrices}
	Let $j, j' \in \N_0$, $A \in {\mathcal S}({\bf E}_j)$, $B \in {\mathcal S}({\bf E}_{j'})$, then the following holds: 
	
	\noindent
	$(i)$ The operators $M_L(A)$, $M_R(B)$ defined in \eqref{definizione moltiplicazione sinistra matrici}, \eqref{definizione moltiplicazione destra matrici} are self-adjoint operators with respect to the scalar product defined in \eqref{prodotto scalare traccia matrici}.
	
	\noindent
	$(ii)$ Let $j, j' \in \N$, $A \in {\mathcal S}({\bf E}_j)$, $B \in {\mathcal S}({\bf E}_{j'})$. The spectrum of the operator $M_L(A) \pm M_R(B)$ satisfies 
	$$
	{\rm spec}\Big( M_L(A) \pm M_R(B) \Big) = \Big\{ \lambda \pm \mu : \lambda \in {\rm spec}(A)\,,\quad \mu \in {\rm spec}(B) \Big\}\,.
	$$
	
	\noindent
	$(iii)$ Let $j \in \N$, $A \in {\mathcal S}({\bf E}_j)$ and $B \equiv \lambda_0 \in {\mathcal S}({\bf E}_0)$. Then, the spectrum of the operators $M_L(A) \pm M_R(\lambda_0) \equiv M_L(A) \pm \lambda_0 {\rm Id} : {\mathcal B}({\bf E}_0, {\bf E}_{j}) \to {\mathcal B}({\bf E}_0, {\bf E}_{j})  $ and $M_L(\lambda_0) \pm M_R(A) \equiv \lambda_0 {\rm Id} \pm M_R(A) : {\mathcal B}({\bf E}_j, {\bf E}_0) \to {\mathcal B}({\bf E}_j, {\bf E}_0)$ satisfy 
	$$
 {\rm spec}\Big( M_L(A) \pm \lambda_0 {\rm Id} \Big) = {\rm spec}\Big( \lambda_0 {\rm Id} \pm M_R(A) \Big) =  \Big\{ \lambda \pm \lambda_0 : \lambda \in {\rm spec}(A) \Big\}\,.
	$$
\end{lemma}
We finish this Section by recalling some well known facts concerning linear self-adjoint operators on finite dimensional Hilbert spaces. Let ${\mathcal H}$ be a finite dimensional Hilbert space of dimension $n$ equipped by the inner product $( \cdot\,,\,\cdot )_{\mathcal H}$. For any self-adjoint operator $A : {\mathcal H} \to {\mathcal H}$, we order its eigenvalues as
\begin{equation}\label{spettro hilbert astratto}
{\rm spec}(A) := \big\{\lambda_1(A) \leq \lambda_2(A) \leq \ldots \leq \lambda_n(A)\big\}\,.
\end{equation}
\begin{lemma}\label{risultato astratto operatori autoaggiunti}
	Let ${\mathcal H}$ be a Hilbert space of dimension $n$. Then the following holds:
	
	\noindent
	$(i)$ Let $A_1, A_2 : {\mathcal H} \to {\mathcal H}$ be self-adjoint operators. Then their eigenvalues, ordered as in \eqref{spettro hilbert astratto}, satisfy the Lipschitz property 
	$$
	|\lambda_k(A_1) - \lambda_k(A_2)| \leq \| A_1 - A_2 \|_{{\mathcal B}({\mathcal H})}\,, \qquad \forall k = 1, \ldots, n\,.
	$$
	
	\noindent
	$(ii)$ Let $A =  y{\rm Id}_{\mathcal H} + B$, where $y \in \R$, ${\rm Id}_{\mathcal H} : {\mathcal H} \to {\mathcal H}$ is the identity and $B :{\mathcal H} \to {\mathcal H}$ is selfadjoint. Then 
	$$
	\lambda_k(A) = y + \lambda_k(B) \,, \qquad \forall k = 1, \ldots , n\,. 
	$$ 
	
	\noindent
	$(iii)$ Let $A : {\mathcal H} \to {\mathcal H}$ be self-adjoint and assume that ${\rm spec}(A) \subset \R \setminus \{ 0 \}$. Then $A$ is invertible and its inverse satisfies
	$$
	\| A^{- 1}\|_{{\mathcal B}({\mathcal H})} = \dfrac{1}{\min_{k = 1, \ldots, n}|\lambda_k(A)|}\,.
	$$
\end{lemma}
\subsection{properties of $\cB^{\s,m}$}\label{linop}
\begin{lemma}\label{B algebra cal B sigma}
	Let $\sigma, \rho > 0 $, $m, m' \in \R$ ${\mathcal R} \in {\mathcal B}^{\sigma, m}, {\mathcal Q} \in {\mathcal B}^{\sigma + \rho, m'}$. Then ${\mathcal R} {\mathcal Q} \in {\mathcal B}^{\sigma, m+ m'}$ and $\| {\mathcal R} {\mathcal Q}\|_{{\mathcal B}^{\sigma, m+ m'}} \lesssim_m \rho^{- |m|} \| {\mathcal R}\|_{{\mathcal B}^{\sigma, m}} \| {\mathcal Q}\|_{{\mathcal B}^{\sigma + \rho , m'}}$. 
\end{lemma}
\begin{proof}
	{\sc Proof of $(i)$}
	By using the $2 \times 2$ block representation of linear operators, one has that the operator ${\mathcal C} := {\mathcal R} {\mathcal Q}$ admits the representation ${\mathcal C} = \sum_{j, j' \in \N_0} \Pi_j {\mathcal C} \Pi_{j'}$ where 
	\begin{equation}\label{def C R Q}
	\Pi_j {\mathcal C} \Pi_{j'} = \sum_{k \in \N_0} (\Pi_j {\mathcal R} \Pi_k)(\Pi_k {\mathcal Q} \Pi_{j'})\,, \quad \forall j, j' \in \N_0\,. 
	\end{equation}
	Using that by triangular inequality $e^{\sigma |j - j'|} \leq e^{\sigma |j - k|} e^{\sigma |k - j'|}$, for any $j'\in \Z$ 
	\begin{equation}\label{pippo 0}
	\begin{aligned}
	\sum_{j \in \N_0} e^{\sigma|j - j'|} \|\Pi_j {\mathcal C} \Pi_{j'}\|_\HS \langle j' \rangle^{- (m + m')} & \leq \sum_{j, k \in \N_0} e^{\sigma|j - j'|} \| \Pi_j {\mathcal R} \Pi_k \|_\HS \| \Pi_k  {\mathcal Q} \Pi_{j'}\|_\HS \langle j' \rangle^{- (m + m')} \\
	& \leq \sum_{j, k \in \N_0} e^{\sigma|j - k|} \| \Pi_j {\mathcal R} \Pi_k\|_\HS  \langle k \rangle^{- m } e^{\sigma |k - j'|} \| \Pi_k  {\mathcal Q} \Pi_{j'}\|_\HS \langle j' \rangle^{-  m'} \langle k \rangle^m \langle j' \rangle^{- m }\,. 
	\end{aligned}
	\end{equation}
	Using that 
	$$
	\langle k \rangle^m \langle j' \rangle^{- m} \lesssim_m 1 + \langle k - j' \rangle^{|m|} \lesssim_m \langle k - j' \rangle^{|m|}
	$$
	the inequality \eqref{pippo 0} implies that 
	\begin{equation}\label{pippo 1}
	\begin{aligned}
	\sum_{j \in \N_0} e^{\sigma|j - j'|} \| \Pi_j {\mathcal C} \Pi_{j'}\|_\HS \langle j' \rangle^{- (m + m')} & \lesssim_m \sum_{j, k \in \N_0} e^{\sigma|j - k|} \| \Pi_j {\mathcal R} \Pi_k\|_\HS  \langle k \rangle^{- m } e^{\sigma |k - j'|} \langle k - j' \rangle^{|m|} \| \Pi_k  {\mathcal Q} \Pi_{j'}\|_\HS \langle j' \rangle^{-  m'} \\
	& \lesssim_m \sup_{k \in \N_0}\Big( \sum_{j \in \N_0}e^{\sigma|j - k|} \| \Pi_j {\mathcal R} \Pi_k\|_\HS  \langle k \rangle^{- m }  \Big) \sum_{k \in \N_0} e^{\sigma |k - j'|} \langle k - j' \rangle^{|m|} \| \Pi_k  {\mathcal Q} \Pi_{j'}\|_\HS \langle j' \rangle^{-  m'} \\
	& \lesssim_m \| {\mathcal R}\|_{{\mathcal B}^{\sigma, m}} \sum_{k \in \N_0} e^{(\sigma + \rho) |k - j'|}  \langle k - j' \rangle^{|m|} e^{- \rho|k - j'|} \| \Pi_k  {\mathcal Q} \Pi_{j'}\|_\HS \langle j' \rangle^{-  m'}\,.
	\end{aligned}
	\end{equation}
	Using that 
	$$
	\sup_{x \geq 0} x^{|m|} e^{- \rho x} \lesssim_m \rho^{- |m|}
	$$
	one gets 
	$$
	\sum_{k \in \N_0} e^{(\sigma + \rho) |k - j'|}  \langle k - j' \rangle^{|m|} e^{- \rho|k - j'|} \| \Pi_k  {\mathcal Q} \Pi_{j'}\|_\HS \langle j' \rangle^{-  m'} \lesssim_m \rho^{- |m|} \| {\mathcal Q}\|_{{\mathcal B}^{\sigma + \rho, m'}}
	$$
	and then the claimed statement follows. 
\end{proof}

\begin{lemma}\label{stima moltiplicazione}
	Let $\sigma > 0$, $a \in {\mathcal H}(\T_{\sigma + \rho})$. Then the multiplication operator ${\mathcal M}_a : u(x) \mapsto a(x) u (x)$ is in ${\mathcal B}^\sigma$ and $\| {\mathcal M}_a\|_{{\mathcal B}^\sigma} \lesssim \rho^{- 1} \| a \|_{\sigma + \rho}$. 
\end{lemma}
\begin{proof}
	One easily see that the multiplication operator ${\mathcal M}_a$ admits the $2 \times 2$ block representation ${\mathcal M}_a = \sum_{j, j' \in \N_0} \Pi_j {\mathcal M}_a \Pi_{j'}$ where for any $j, j' \in \N_0$, the operator $\Pi_j {\mathcal M}_a \Pi_{j'}$ is represented by the matrices
	$$
	\begin{aligned}
	 \begin{pmatrix}
	\widehat a(j - j') & \widehat a(j + j') \\
	\widehat a(- j - j') & \widehat a(- j + j')\,,
	\end{pmatrix} \quad j, j' \in \N\,,\quad  \begin{pmatrix}
	\widehat a(j) \\
	\widehat a(- j)
	\end{pmatrix}\quad j \in \N\,, \quad  \big( \widehat a(j'), 
	\widehat a(- j')  \big)\quad j' \in \N\,.
	\end{aligned}
	$$
	Using that $a \in {\mathcal H}(\T_{\sigma + \rho})$, one obtains that 
	$$
	\begin{aligned}
	& |\widehat a(j - j')|, |\widehat a(- j + j')| \leq \| a \|_{\sigma + \rho} e^{- (\sigma + \rho)|j - j'|}\,, \\
	& |\widehat a(j + j')|, |\widehat a(- j - j')| \leq \| a \|_{\sigma + \rho} e^{- (\sigma + \rho)|j + j'|}\,.
	\end{aligned}
	$$
	Using that for any $j, j' \in \N_0$, $e^{- (\sigma + \rho)|j + j'|} \leq e^{- (\sigma + \rho)|j - j'|}$, one gets that 
	$$
	\| \Pi_j {\mathcal M}_a \Pi_{j'}\|_\HS \lesssim \| a \|_{\sigma + \rho} e^{- (\sigma + \rho)|j - j'|}, \quad \forall j, j' \in \N_0\,. 
	$$
	Therefore for any $j' \in \N_0$, 
	$$
	\sum_{j \in \N_0} e^{\sigma |j - j'|} \| \Pi_j {\mathcal M}_a \Pi_{j'}\|_\HS \lesssim \| a \|_{\sigma + \rho} \sum_{j \in \N_0} e^{- \rho |j - j'|} \lesssim \rho^{- 1} \| a \|_{{\mathcal H}^{\sigma + \rho}_x}\,.
	$$
	The thesis then follows by recalling the definition \eqref{definizione classe cal B sigma}. 
\end{proof}
\subsection{Properties of torus diffeomorphisms}
In Subsection \ref{sezione riparametrizzazione tempo}, we have considered  diffeomorphisms of the form 
\begin{equation}\label{omega alpha}
\f \mapsto \f + \omega \alpha(\f)
\end{equation}
where $\alpha \in {\mathcal H}(\T^\infty_{\sigma + \rho})$, $\sigma, \rho > 0$ and $\omega \in \Dc$. By Lemma \ref{lemma diffeo inverso}, for $\e = \e(\rho)$ small enough, if $\| \alpha\|_{{\mathcal H}^{\sigma + \rho}} \leq \e$, then the diffeomorphism \eqref{omega alpha} is invertible and its inverse has the form 
\begin{equation}\label{omega tilde alpha}
\vartheta \mapsto \vartheta + \omega \widetilde \alpha(\vartheta)
\end{equation}
where $\widetilde \alpha \in {\mathcal H}(\T^\infty_\sigma)$ and $\| \widetilde \alpha\|_{\sigma} \lesssim \| \alpha\|_{\sigma + \rho}$. 
Note that by \eqref{omega alpha}, \eqref{omega tilde alpha}, one can easily deduce the formulae
\begin{equation}\label{formule omega alpha alpha tilde}
\begin{aligned}
& 1 + \omega \cdot \partial_\vartheta \tilde \alpha(\vartheta) = \frac{1}{1 + \omega \cdot \partial_\f \alpha(\vartheta + \omega \widetilde \alpha(\vartheta))},  \\
&  1 + \omega \cdot \partial_\f  \alpha(\f) = \frac{1}{1 + \omega \cdot \partial_\vartheta \widetilde \alpha(\f  + \omega  \alpha(\f))}\,.
\end{aligned}
\end{equation}
The following lemma will be used in the reduction procedure of Section \ref{sezione riduzione ordine}, in order to show that some averages do not depend on the parameter $\omega \in \Omega$.
\begin{lemma}\label{lemma per media lambda 1}
	The following holds:
	
	\noindent
	Let $\omega \in \Dc$ be a Diophantine frequency and let $a$ be a function in ${\mathcal H}(\T^\infty_\sigma)$. Then $\int_{\T^\infty } \omega \cdot \partial_\vartheta a(\vartheta)\, d \vartheta = 0$. As a consequence one has
	\begin{equation}\label{media cambio variabile}
	\int_{\T^\infty}\Big( 1 + \omega \cdot \partial_\vartheta \widetilde \alpha(\vartheta) \Big) d \vartheta = 1
	\end{equation}
	and for any $\ell \in \Z^\infty_* \setminus \{ 0 \}$, 
	\begin{equation}\label{derivata totale cambio variabile}
	\int_{\T^\infty} e^{\ii \ell \cdot \big(\vartheta + \omega \widetilde \alpha(\vartheta) \big)} \Big( 1 + \omega \cdot \partial_\vartheta \widetilde \alpha(\vartheta)\Big)\, d \vartheta = 0\,.
	\end{equation}
\end{lemma}
\begin{proof}
	Let $N \in \N$. Then 
	We split 
	$$
	\omega \cdot \partial_\vartheta a(\vartheta) = \sum_{\ell \neq 0\,,\,|\ell|_\zia \leq N} \ii \omega \cdot \ell \widehat a(\ell) e^{\ii \ell \cdot \vartheta} + \sum_{|\ell|_\zia > N} \ii \omega \cdot \ell \widehat a(\ell) e^{\ii \ell \cdot \vartheta}\,. 
	$$
	Since $a$ is an analytic function, the second term on the right hand side goes to zero as $N \to + \infty$. Moreover 
	$$
	\int_{\T^N} \sum_{\ell \neq 0\,,\,|\ell|_\zia \leq N} \ii \omega \cdot \ell \widehat a(\ell) e^{\ii \ell \cdot \vartheta}\, d \vartheta = \sum_{\ell \neq 0\,,\,|\ell|_\zia \leq N} \ii \omega \cdot \ell \widehat a(\ell)  \int_{\T^N}e^{\ii \ell \cdot \vartheta} \, d \vartheta = 0\,.
	$$
	Therefore one deduces that 
	$$
	\int_{\T^\infty} a (\vartheta)\, d \vartheta = \lim_{N \to \infty} \frac{1}{(2 \pi)^N}\int_{\T^N} \sum_{|\ell|_\zia > N} \ii \omega \cdot \ell \widehat a(\ell) e^{\ii \ell \cdot \vartheta}\, d \vartheta = 0\,.
	$$
	The equality \eqref{media cambio variabile} follows immediately by the previous claim. The equality \eqref{derivata totale cambio variabile}, follows observing that since $\ell \in \Z^\infty_* \setminus \{ 0 \}$ and $\omega$ is Diophantine, one has that 
	$$
	\begin{aligned}
	e^{\ii \ell \cdot \big(\vartheta + \omega \widetilde \alpha(\vartheta) \big)} \Big( 1 + \omega \cdot \partial_\vartheta \widetilde \alpha(\vartheta)\Big) & = \frac{1}{\ii \omega \cdot \ell} \omega \cdot \partial_\vartheta \Big( e^{\ii \ell \cdot \big(\vartheta + \omega \widetilde \alpha(\vartheta) \big)}\Big)
	\end{aligned}
	$$
	hence the result follows by applying the first claim. 
\end{proof}

\begin{lemma}[\bf Moser composition lemma]\label{moser type lemma}
	Let $f : B_R(0) \to \C$ be an holomorphic function defined in a neighbourhood of the origin $B_R(0)$ of the complex plane $\C$. Then the composition operator $F (u) := f \circ u$ is a well defined non linear map ${\mathcal H}(\T^\infty_\sigma) \to {\mathcal H}(\T^\infty_\sigma)$.  
\end{lemma}
\begin{proof}
	Clearly, since $f(z) = \sum_{n \geq 0} a_n z^n$ is analytic, for any $z \in \C$, $|z| < R$, the series $\sum_{n \geq 0} |a_n| |z|^n$ is convergent. Moreover, Let $u \in {\mathcal H}(\T^\infty_\sigma)$ with $\| u \|_{\sigma} \leq r < R$. By applying Lemma \ref{Lemma prodotto}, for any $n \geq 1$, $u^n \in {\mathcal H}(\T^\infty_\sigma)$ and $\| u^n \|_{\sigma} \leq \| u \|_{\sigma}^n \leq r^n$. The series $\sum_{n \geq 0} a_n u^n$ is absolutely convergent w.r. to $\| \cdot \|_{\sigma}$. Indeed , one has
	$$
	\Big\| \sum_{n \geq 0} a_n u^n \Big\|_{\sigma} \leq \sum_{n \geq 0} |a_n| \| u \|^n_{\sigma} \leq \sum_{n \geq 0} |a_n| r^n < \infty\,.
	$$ 
	this implies that $F(u) = \sum_{n \geq 0} a_n u^n$ belongs to the space ${\mathcal H}(\T^\infty_\sigma)$ and the proof of the lemma is concluded. 
\end{proof}
\section{some estimates of constants}\label{appendiceB}

\begin{lemma}\label{bound per stima di Cauchy}

\noindent
$(i)$ Let $\mu_1,\mu_2 > 0$. Then 
$$
\sup_{\begin{subarray}{c}
\ell \in \Z^\infty_* \\
|\ell|_\zia < \infty
\end{subarray}} \prod_{i}( 1+  \langle i \rangle^{\mu_1} | \ell_i|^{\mu_2}) e^{- \rho |\ell|_\zia} \leq  {\rm exp}\Big(\frac{\tau}{\rho^{\frac{1}{\zia}}} \ln\Big(\frac{\tau}{\rho} \Big) \Big)
$$
for some constant $\tau = \tau(\zia, \mu_1, \mu_2) > 0$. 

\noindent
$(ii)$ Let $\rho > 0$. Then $\sum_{\ell \in \Z^\infty_*} e^{- \rho |\ell|_\zia}  \lesssim {\rm exp}\Big(\frac{\tau}{\rho^{\frac{1}{\zia}}} \ln\Big(\frac{\tau}{\rho} \Big) \Big)
$, for some constant $\tau = \tau(\zia) > 0$. 
\end{lemma}
\begin{proof}

\noindent
{\sc Proof of $(i)$.}  We remark that the left hand side can be expressed as
 \[
\exp\Big(\sum_i  - \rho \langle i \rangle^\zia |\ell_i|+ \ln\big( 1 + \langle i \rangle^{\mu_1} |\ell_i|^{\mu_2} \big) \Big)=:  {\rm exp}(\sum_i f_i(|\ell_i|))\]
where 
%Let $\ell \in \Z^\infty$, $|\ell|_\zia < \infty$. One has that 
%\begin{equation}\label{barbetta 0a}
%\begin{aligned}
%\prod_{i} ( 1 + \langle i \rangle^{\mu_1} | \ell_i |^{\mu_2}) e^{- \rho \langle i \rangle^\zia |\ell_i|} & = \prod_{i} {\rm exp}\Big( \ln \big( 1 + \langle i \rangle^{\mu_1} |\ell_i|^{\mu_2} \big) - \rho \langle i \rangle^\zia |\ell_i|  \Big)  \\
%& = {\rm exp}(\sum_i f_i(|\ell_i|))\,. 
%\end{aligned}
%\end{equation}
%where 
\begin{equation}\label{barbetta 2a}
 f_i(x) := \ln \big( 1 + \langle i \rangle^{\mu_1} x^{\mu_2} \big) - \rho \langle i \rangle^\zia x\,.  %  \mu_1 \ln(\langle i \rangle) + \mu_2 \ln(1 + x) - \rho \langle i \rangle^\zia x  \,. 
\end{equation}
then the result follows essentially word by word from Lemma 7.2 of \cite{BMP1:2018} where it is proved in the special case $\mu_1= 2+q$, $\mu_2= 2$.
Since $f_i(0) = 0$, it is enough to estimate ${\rm max }_{x \geq 1} f_i(x)$, in order to bound the series $\sum_{i} f_i(|\ell_i|)$. 
One has that for any $x \geq 1$
$$
\begin{aligned}
f_i(x) & \leq \ln(2 \langle i \rangle^{\mu_1} x^{\mu_2}) - \rho \langle i \rangle^\zia x \leq C_0(\mu_1) \ln(\langle i \rangle) + \mu_2 \ln(x) - \rho \langle i \rangle^\zia x =: g_i(x) 
\end{aligned}
$$
for some constant $C_0(\mu_1) > 0$ and hence 
\[{\rm max}_{x \geq 1} f_i \leq {\rm max}_{x \geq 1} g_i\,. 
\]
Using that $\ln(x) \leq x$ for any $x \geq 1$, one has that 
$$
g_i(x) \leq C_0(\mu_1) \ln(\langle i \rangle) - \frac{\rho \langle i \rangle^\zia}{2} x, \quad \forall i \geq \Big( \frac{2 \mu_2}{\rho} \Big)^{\frac{1}{\zia}}\,. 
$$
Furthermore, 
$$
C_0(\mu_1) \ln(\langle i \rangle) - \frac{\rho \langle i \rangle^\zia}{2} x \leq 0, \quad \forall i \geq \Big( \frac{2 C_0(\mu_1)}{\zia \rho} \Big)^{\frac{1}{\zia}}
$$
and hence 
$$
g_i(x) \leq 0, \quad \forall i \geq \Big(\frac{C_1}{\rho}\Big)^{\frac{1}{\zia}}, \quad C_1 \equiv C_1(\mu_1, \mu_2, \zia) := {\rm max}\{ \frac{2 C_0(\mu_1)}{\zia}\,, 2 \mu_2 \}\,. 
$$
If $i \leq \frac{C_1}{\rho^{\frac{1}{\zia}}}$, a direct calculation shows that the maximum of $g_i$ is achieved at the point 
$
x_i = \frac{\mu_2}{\rho \langle i \rangle^\zia}
$
and 
\begin{equation*}
g_i(x_i)  = C_0 \ln(\langle i \rangle) + \mu_2 \ln\Big( \frac{\mu_2}{\rho \langle i \rangle^\zia} \Big) - \mu_2  \leq  \frac{C_0}{\zia}
 \ln\Big( \frac{C_1}{\rho} \Big) + \mu_2 \ln\Big( \frac{\mu_2}{\rho} \Big) \leq C_2 \ln\Big( \frac{C_2}{\rho} \Big)
\end{equation*}
for some constant $C_2 = C_2(\zia, \mu_1, \mu_2) > 0$ large enough. Thus
$$
\sum_{i} f_i(x) \leq \sum_{i \leq {C_1}{\rho^{-\frac{1}{\zia}}}} g_i(x) \leq \frac{C_1}{\rho^{\frac{1}{\zia}}} C_2 \ln\Big( \frac{C_2}{\rho} \Big)
$$
%
%
%\bigskip
%
%****************
%
%\bigskip
%
%
%Using the inequality $\ln(1 + x) \leq x$, one gets that 
%\begin{equation}\label{barbetta 3a}
%f_i(x) \leq \big( k - \rho \langle i \rangle^\zia \big) x  \leq 0 
%\end{equation}
%provided
%\begin{equation}\label{barbetta 4a}
%\langle i \rangle \geq  m_{k, \rho} := \Big( \frac{k}{\rho} \Big)^{\frac{1}{\zia}}\,. 
%\end{equation}
%Then 
%\begin{equation}\label{barbetta 5a}
%\sum_{i} f_i(|\ell_i|) \leq \sum_{\langle i \rangle \leq m_{k, \rho}} f_i(|\ell_i|)\,. 
%\end{equation}
%For $\langle i \rangle \leq m_{k, \rho}$, the function $f_i$ admits a maximum at the point 
%$$
%x_i = \frac{k - \rho \langle i \rangle^\zia}{ \rho \langle i \rangle^\zia}
%$$
%and one computes
%\begin{equation}\label{barbetta 6a}
%f_i(x_i) = k \Big(  \ln\Big( \frac{k}{\rho \langle i \rangle^\zia} \Big) + 1 \Big) - \rho \langle i \rangle^\zia \leq k \Big(  \ln\Big( \frac{k}{\rho} \Big) + 1 \Big) \,. 
%\end{equation}
%Thus, recalling \eqref{barbetta 4a}, \eqref{barbetta 5a}, \eqref{barbetta 6a}, one obtains the estimate 
%\begin{equation}\label{barbetta 7a}
%\sum_i f_i(|\ell_i|) \lesssim  k m_{k, \rho} \Big(  \ln\Big( \frac{k}{\rho} \Big) + 1 \Big) \lesssim  \Big( \frac{k}{\rho} \Big)^{\frac{1}{\zia}} k  \Big(  \ln\Big( \frac{k}{\rho} \Big) + 1 \Big)
%\end{equation}

\noindent
{\sc Proof of $(ii)$.}
By Lemma 4.1 of \cite{BMP1:2018}, one has 
$$
\sum_{\ell \in \Z^\infty_*} \prod_{i}\frac{1}{1 + \langle i \rangle^{2} |\ell_i|^2} \leq C_0 < \infty\,.
$$
Therefore 
$$
\begin{aligned}
\sum_{\ell \in \Z^\infty_*} e^{- \rho |\ell|_\zia} & = \sum_{\ell \in \Z^\infty_*} \prod_{i}\frac{1}{1 + \langle i \rangle^{2} |\ell_i|^2} e^{- \rho \langle i \rangle^\zia |\ell_i|}\big( 1 + \langle i \rangle^{2} |\ell_i|^2 \big) \\
& \lesssim  \sup_{\ell \in \Z^\infty_*} \Big(\prod_i e^{- \rho \langle i \rangle^\zia |\ell_i|}\big( 1 + \langle i \rangle^{2} |\ell_i|^2 \big) \Big) \,. 
\end{aligned}
$$
The claimed statement then follows by item $(i)$ with $\mu_1=\mu_2=2$. 

\end{proof}
\begin{lemma}[Small divisor estimate]\label{small divisor con taglio}
Let $\mu_1, \mu_2 \geq  1$. We have the following estimate for $N\gg 1$
\begin{equation}
\label{taglio}
\sup_{\ell\in \Z^\infty_*:\; |\ell|_{\zia}<N} \prod_{i}(1+\jap{i}^{\mu_1} |\ell_i|^{\mu_2}) \le (1+N)^{C(\zia,\mu_1, \mu_2)N^{\frac{1}{1+\zia}}}
\end{equation}
for some constant $C(\zia, \mu_1, \mu_2)>0$.
\end{lemma}
\begin{proof}
For $\ell$  fixed, let us denote by $k$ the number of non-zero components of $\ell$. We claim that $k \lesssim_\zia N^{\frac{1}{1+\zia}}$, indeed
\[
N\ge |\ell|_\zia= \sum_{j=1}^k \jap{i_j}^\zia|\ell_{i_j}|\ge   \sum_{j=1}^k \jap{i_j}^\zia \ge  \sum_{j=1}^k j^\zia \simeq_\zia k^{1+\zia}
\]
and the claim follows. 
Now if $\zia\ge 1$ we have $\jap{i} |\ell_i| \le \jap{i}^\eta |\ell_i| \le N$ and setting $\mu := {\rm max}\{ \mu_1, \mu_2 \}$
\[
\sup_{\ell\in \Z^\infty_*:\; |\ell|_{\zia}\le N} \sum_{i}\ln (1+\jap{i}^{\mu_1} |\ell_i|^{\mu_2})  \lesssim_\zia N^{\frac{1}{1+\zia}} \ln (1+N^\mu)\lesssim_{\zia,\mu} N^{\frac{1}{1+\zia}} \ln (1+N).
\]
otherwise if  $\zia\le 1$  one has $\jap{i} |\ell_i| \le  (\jap{i}^\zia |\ell_i| )^{\frac{1}{\zia}} \le N^{\frac{1}{\zia}} $ and again
\[
\sup_{\ell\in \Z^\infty_*:\; |\ell|_{\zia}\le N} \sum_{i}\ln (1+\jap{i}^{\mu_1} |\ell_i|^{\mu_2})  \lesssim_\zia N^{\frac{1}{1+\zia}} \ln (1+N^{\frac{\mu}{\zia}})\lesssim_{\zia,\mu} N^{\frac{1}{1+\zia}} \ln (1+N).
\]

\end{proof}
\begin{lemma}\label{stima serie piccoli divisori}
For $\mu_1,\mu_2>3$, one  has that $\sum_{\ell \in \Z^\infty_*} \frac{\|\ell \|_1^2}{{\divisor}(\ell)} < \infty$ where $\divisor (\ell) := \prod_{i\in\N }(1 + \langle i \rangle^{\mu_1} |\ell_i|^{\mu_2})$. 
\end{lemma}
\begin{proof}
The proof is very similar to the one of the measure estimate Lemma 4.1 of \cite{BMP1:2018}.
	For $\ell\in \Z^\infty_*$ %with $ 0<|\ell|_\zia<\infty$ %we define 
%	\[
%	\mathcal R_\ell := \set{\omega\in \Ro\,:\;	|\omega\cdot \ell|\leq \g\prod_{n\neq 0}\frac{1}{(1+|\ell_j|^{\mu_1} | j|^{{\mu_2}})}}
%	\] 
let $s=s(\ell)$ be the smallest index $i$ such that $\ell_i \neq 0$ and $S=S(\ell)$ be the biggest. %Assume w.l.o.g. that $\ell_s,\ell_S\neq 0$,
%		Then  we have\footnote{Assume, e.g. that $\ell_s\neq 0$, then
%			$|\partial_{\xi_s}\omega\cdot\ell|\geq s^{-\fp}\,.$} 
%		\begin{equation}\label{singoloRis}
%		\mu(\mathcal R_\ell) \le \frac{\g  s^\fp }{\pa{1+|\ell_0|^{\mu_1}} }\prod_{n\neq 0}\frac{1}{(1+|\ell_n|^{\mu_1} |n|^{{\mu_2}+\fp})}.
%		\end{equation}
Recalling
		\begin{equation*}
	\prod_{n\in\N}\frac{1}{(1+|\ell_n|^{\mu_1} n^{{\mu_2}})}
	= \prod_{s(\ell)\le n\le S(\ell)}\frac{1}{(1+|\ell_n|^{\mu_1} |n|^{{\mu_2}})} 
	\end{equation*}	
	Now
	\begin{align}
	&\sum_{\ell \in \Z^\infty_*} \frac{\|\ell \|_1^2}{{\divisor}(\ell)}\le \sum_{s\in\N} \sum_{\substack{ \ell: s(\ell)= S(\ell)=s}}\frac{|\ell_s|^2}{(1+|\ell_s|^{\mu_1} |s|^{{\mu_2}})}\label{mamma2}\\
	+& \sum_{S\in\N}\sum_{0<s<S }(S-s)^2\sum_{\substack{ \ell: s(\ell)=s,\\ S(\ell)=S}}\prod_{s\le n\le S }\frac{\jap{\ell_n}^2}{(1+|\ell_n|^{\mu_1} |n|^{{\mu_2}})} .\label{mamma3}
	\end{align}
%	Let us estimate \eqref{mamma2}
%	\begin{align*}
%	&\sum_{s>0}\sum_{\ell_0\in \Z}\frac{1}{1 + |\ell_0|^{\mu_1}}\sum_{ \substack{\ell_s,\ell_{-s}\in\Z\\ |\ell_s|+|\ell_{-s}|>0}}\frac{\g s^\fp}{(1+|\ell_s|^{\mu_1} |s|^{{\mu_2}+\fp})} \frac{1}{(1+|\ell_{-s}|^{\mu_1} |s|^{{\mu_2}+\fp})}\\ 
%	&\le c(\mu_1)\g \sum_{s>0} s^\fp \sum_{ \substack{\ell_s,\ell_{-s}\in\Z\\ |\ell_s|+|\ell_{-s}|>0}}\frac{1}{(1+|\ell_s|^{\mu_1} |s|^{{\mu_2}+\fp})} \frac{1}{(1+|\ell_{-s}|^{\mu_1} |s|^{{\mu_2}+\fp})}
%	\end{align*}
	Now  for $\mu_1> 3$
	\[
	\sum_{h=1}^\infty  \frac{h^2}{(1+h^{\mu_1} |n|^{{\mu_2}})} \le \sum_{h=1}^\infty  \frac{1}{h^{\mu_1-2} |n|^{{\mu_2}}} \le  \frac{c(\mu_1)}{|n|^{{\mu_2}}}
	\]
	hence
	\[
	\sum_{h\in\Z} \frac{\jap{h}^2}{(1+|h|^{\mu_1} |n|^{{\mu_2}+\fp})} \le  1+ \frac{c(\mu_1)}{|n|^{{\mu_2}}}.
	\]
Consequently for $\mu_2>1$,
	\eqref{mamma2} is bounded by  
	\[
	c(\mu_1)\sum_{s>0} |s|^{-{\mu_2}} \le c_3(\mu_1,\mu_2)\g.
	\]
	\\
	Regarding  \eqref{mamma3}, 
	 we have
	\begin{align*}
&\sum_{\substack{ \ell: s(\ell)=s,\\ S(\ell)=S}}\prod_{s\le n\le S }\frac{\jap{\ell_n}^2}{(1+|\ell_n|^{\mu_1} |n|^{{\mu_2}})} \le \frac{c(\mu_1)^2}{|s|^{{\mu_2}}|S|^{{\mu_2}}}\prod_{s<n<S}(1+ \frac{c(\mu_1)}{|n|^{{\mu_2}}})= \frac{c(\mu_1)^2}{|s|^{{\mu_2}}|S|^{{\mu_2}}}\exp\Big(\sum_{s<n<S}\ln(1+ \frac{c(\mu_1)}{|n|^{{\mu_2}}})\Big)\le \\
&\frac{c(\mu_1)^2}{|s|^{{\mu_2}}|S|^{{\mu_2}}}\exp\Big(\sum_{n\in\N} \frac{c(\mu_1)}{|n|^{{\mu_2}}}\Big) \le \frac{c_1(\mu_1)}{|s|^{{\mu_2}}|S|^{{\mu_2}}}
	\end{align*}
	consequently \eqref{mamma3} is bounded by
	\[
	\sum_{S\in\N}\sum_{0<s<S }(S-s)^2 \frac{c_1(\mu_1)}{|s|^{{\mu_2}}|S|^{{\mu_2}}}<\infty
	\]
	provided that $\mu_2>3$.
\end{proof}

\bibliographystyle{alpha}
\bibliography{biblioAlmostPeriodic.bib}
\end{document}